\documentclass{amsbook}

\usepackage{hyperref}
\usepackage{dsfont}
\usepackage{amssymb}
\usepackage{esint}
\usepackage{mathptmx,times}
\usepackage[all]{xy}
\usepackage{mathtools}
\usepackage{enumitem}

\setitemize[0]{leftmargin=15pt}

\newtheorem{thm}{Theorem}[chapter]
\newtheorem{lemma}[thm]{Lemma}
\newtheorem{prop}[thm]{Proposition}
\newtheorem{cor}[thm]{Corollary}

\theoremstyle{definition}
\newtheorem{remark}[thm]{Remark}

\newtheorem{definition}[thm]{Definition}

\newtheorem{example}[thm]{Example}

\newcommand{\ist}[1]{\stackrel{\text{\makebox[0pt]{#1}}}{=}}
\newcommand{\ausruf}{!}

\newcommand{\ti}{\breve{\i}}
\newcommand{\vds}{\breve{p}}
\newcommand{\N}{{\mathbb N}}
\newcommand{\C}{{\mathbb C}}
\newcommand{\R}{{\mathbb R}}
\newcommand{\Z}{{\mathbb Z}}
\newcommand{\BB}{\mathcal{B}}
\newcommand{\CC}{\mathcal{C}}
\newcommand{\DD}{\mathcal{D}}
\newcommand{\HH}{\mathcal{H}}
\newcommand{\LL}{\mathcal{L}}
\newcommand{\PP}{\mathcal{P}}

\newcommand{\ZZ}{\mathcal{Z}}
\newcommand{\PB}{\mathrm{PB}}
\newcommand{\PBE}{\mathrm{PB}_E}
\newcommand{\PBdE}{\mathrm{PB}_{\partial E}}

\newcommand{\id}{\operatorname{id}}
\newcommand{\im}{{\operatorname{im}}}
\newcommand{\ev}{{\operatorname{ev}}}

\newcommand{\cov}{{\operatorname{cov}}}
\newcommand{\curv}{{\operatorname{curv}}}
\newcommand{\curvt}{{\widetilde{\mathrm{curv}}}}
\newcommand{\jt}{{\widetilde j}}
\newcommand{\ct}{{\widetilde{c}}}
\newcommand{\iotat}{{\widetilde{\iota}}}
\newcommand{\Ad}{\operatorname{Ad}}
\newcommand{\Hom}{\operatorname{Hom}}
\newcommand{\Ext}{\operatorname{Ext}}
\newcommand{\Tor}{\operatorname{Tor}}
\newcommand{\pr}{\operatorname{pr}}
\newcommand{\la}{\langle}       
\newcommand{\ra}{\rangle}    
\newcommand{\cdt}{\,\cdot\,}   
\newcommand{\Ul}{{\mathrm{U}(1)}}       
       
\newcommand{\hol}{\operatorname{Hol}}
\newcommand{\PT}{\operatorname{PT}}

\newcommand{\Hdr}{H_\mathrm{dR}}
\newcommand{\Ht}{{\widetilde H}}
\newcommand{\Ch}{{\tt{Chain}}}

\newcommand{\g}{{\mathfrak g}}       
\renewcommand{\phi}{\varphi}

\newcommand{\polhk}[1]{\setbox0=\hbox{#1}{\ooalign{\hidewidth
    \lower1.5ex\hbox{`}\hidewidth\crcr\unhbox0}}}

\makeindex

\begin{document}
%%%%%%%%%%%%%%%%%%%%%%%%%%%%%%%%%%%%%%%%%%%%%%%%%%%%%%%%%%%%%%%%%%%%%%%%%

\title{Differential Characters and Geometric Chains}
\author{Christian B\"ar and Christian Becker}
\address{Universit\"at Potsdam, Institut f\"ur Mathematik, Am Neuen Palais 10, 14469 Potsdam, Germany}
\email[C.~B\"ar]{baer@math.uni-potsdam.de}
\email[C.~Becker]{becker@math.uni-potsdam.de}
\keywords{Cheeger-Simons differential character, geometric chain, thin invariance, differential cohomology, multiplication of differential characters, fiber integration, up-down formula, relative differential character, holonomy, parallel transport, transgression, chain field theory}
\subjclass[2010]{53C08, 55N20}

\maketitle

\begin{abstract}
We study Cheeger-Simons differential characters and provide geometric descriptions of the ring structure and of the fiber integration map.
The uniqueness of differential cohomology (up to unique natural transformation) is proved by deriving an explicit formula for any natural transformation between a differential cohomology theory and the model given by differential characters.
Fiber integration for fibers with boundary is treated in the context of relative differential characters.
As applications we treat higher-dimensional holonomy, parallel transport, and transgression.
\end{abstract}

\tableofcontents
%%%%%%%%%%%%%%%%%%%%%%%%%%%%%%%%%%%%%%%%%%%%%%%%%%%%%%%%%%%%%%%%%%%%%%%%%
\chapter{Introduction}
%%%%%%%%%%%%%%%%%%%%%%%%%%%%%%%%%%%%%%%%%%%%%%%%%%%%%%%%%%%%%%%%%%%%%%%%%

Differential characters were introduced by Cheeger and Simons in \cite{CS83}.
Let $X$ be a differentiable manifold.
A differential character of degree $k$ on $X$ is a homomorphism $h: Z_{k-1}(X;\Z)\to \Ul$.
Here $Z_{k-1}(X;\Z)$ denotes the group of smooth integral-valued singular cycles of degree $k-1$. 
It is supposed that the evaluation on boundaries is given by integration of a form, more precisely, there exists a differential form $\curv(h)\in\Omega^k(X)$ such that $h(\partial c)=\exp \big( 2 \pi i  \int_c \curv(h) \big)$.
The form $\curv(h)$ is uniquely determined by $h$ and is called its \emph{curvature}.
We denote the set of all differential characters on $X$ of degree $k$ by $\widehat H^k(X;\Z)$.

In degree $k=1$ a differential character is essentially a smooth $\Ul$-valued function on $X$.
If one is given a $\Ul$-bundle over $X$ with connection, then one can associate a differential character by mapping any $1$-cycle to the holonomy of the bundle along this cycle.
This sets up a bijection between isomorphism classes of $\Ul$-bundles with connection to the set of differential characters of degree $k=2$. 
In a similar way, differential characters of higher degree correspond to ``higher $\Ul$-gauge theories'' like Hitchin gerbes in degree $k=3$.

The Chern class provides a bijection between $H^2(X;\Z)$ and the set of isomorphism classes of $\Ul$-bundles (without connection).
Hence $\widehat H^2(X;\Z)$ may be considered as a geometric enrichment of the singular cohomology group $H^2(X;\Z)$.
In fact, in any degree there is an analogous map $c:\widehat H^k(X;\Z) \to H^k(X;\Z)$ associating to a differential character its \emph{characteristic class}.
This observation can be axiomatized and leads to the concept of \emph{differential cohomology theory}.
Differential characters form a model for differential cohomology.
We give a constructive proof of the uniqueness of differential cohomology up to unique natural transformations by deriving an explicit formula for any natural transformation between a differential cohomology theory and differential characters.

Pointwise multiplication provides $\widehat H^k(X;\Z)$ with an obvious abelian group structure.
There is a less obvious \emph{multiplication} $\widehat H^k(X;\Z) \times \widehat H^l(X;\Z) \to \widehat H^{k+l}(X;\Z)$ which turns $\widehat H^*(X;\Z)$ into a ring.
We show that a set of natural axioms uniquely determines the ring structure.
Again, the proof is constructive and gives us an explicit geometric description of the ring structure, quite different from the original definition in \cite{CS83}.

Like for singular cohomology and for differential forms there is a concept of \emph{fiber integration} for differential characters.
We show that naturality and two compatibility conditions uniquely determine the fiber integration map. 
Let \mbox{$\pi:E\to X$} be a fiber bundle with closed oriented fibers $F$.
For the fiber integration map $\widehat\pi_!:\widehat H^{k+\dim(F)}(E;\Z)\to \widehat H^k(X;\Z)$ we obtain the geometric formula
\[
 (\widehat\pi_!h)(z)
= h(\lambda(z)) \cdot \exp \Big( 2\pi i \int_{a(z)} \fint_F \curv(h)  \Big).  
\]
Here $\lambda$ is a \emph{transfer map} and essentially does the following:
given a cycle $z$ in $X$ look at the homology class represented by $z$ and choose a closed manifold whose fundamental class also represents this homology class.
Then pull back the bundle $E$ to this manifold and take a representing cycle of the fundamental class of the resulting total space.
This is then a cycle in $E$ which can be inserted into $h$.
The ``correction factor'' $\exp \big( 2\pi i \int_{a(z)} \fint_F \curv(h)\big)$ involves the fiber integration $\fint$ of differential forms and a chain $a(z)$ associated with $z$.
It ensures that the construction is independent of the choices.

The uniqueness results for fiber integration and for differential cohomology together show that the various fiber integration maps for different models of differential cohomology in the literature are all equivalent.

There is the technical problem that not every homology class can be represented by a manifold.
For this reason we have to allow for certain ``manifolds'' with singularities, called \emph{stratifolds}.
We use stratifolds to define \emph{geometric chains} in order to provide a geometric description of singular homology theory.

There is a second reason to consider differential characters on more general ``smooth spaces'', rather than manifolds only.
Certain infinite-dimensional manifolds have to be allowed because we want to apply the theory to the loop space of a manifold, for instance.
 
The multiplication $*$ and the fiber integration map are compatible:
Given \mbox{$h \in \widehat H^k(X;\Z)$} and $f \in \widehat H^l(E;\Z)$, we show that the \emph{up-down formula} holds:
\[
\widehat \pi_!(\pi^*h * f) 
= h * (\widehat \pi_!f) \in \widehat H^{k+l-\dim F}(X;\Z).
\]

If the fibers of the bundle bound, then the fiber integrated differential character turns out to be topologically trivial.
This means that its characteristic class vanishes.
One finds an explicit topological trivialization involving the curvature.
A special case of this situation is the well-known \emph{homotopy formula}.
Let $f:[0,1] \times X \to Y$ be a homotopy between smooth maps $f_0,f_1: X \to Y$ and $h \in \widehat H^k(Y;\Z)$.
Then we find
$$
f_1^*h - f_0^*h 
=
\iota \Big( \int_0^1 f_s^*\curv(h) ds \,\Big) \,.
$$

We also consider the groups of \emph{relative differential characters}, denoted $\widehat H^k(X,A;\Z)$.
In degree $k=1$ they correspond to smooth $\Ul$-valued functions on $X$ with a lift to an $\R$-valued function over $A$.
In degree $k=2$ they correspond to $\Ul$-bundles with connection over $X$ with a section over $A$.
We derive long exact sequences relating absolute and relative differential characters.
Since differential cohomology theories are not cohomology theories in the usual sense, these exact sequences are more subtle than those in singular cohomology theory, for instance.
Our sequences provide criteria for a differential character to be topologically trivial over $A$.
\emph{Fiber integration for fibers with boundary} can now be defined.
It is a map $\widehat\pi_!^E: \widehat H^{k+\dim(F)}(E;\Z) \to \widehat H^{k+1}(X,X;\Z)$.

We apply fiber integration to construct transgression maps to the loop space $\LL(X)$ of a smooth manifolds $X$ and more general mapping spaces. 
Transgression along $S^1$ is a homomorphism $\widehat H^k(X;\Z) \to \widehat H^{k-1}(\LL(X);\Z)$.
It is constructed by pull-back of differential characters from $X$ to $\LL(X) \times S^1$ using the evaluation map followed by integration over the fiber of the trivial bundle.
Analogously, we define transgression along any oriented closed manifold $\Sigma$.
Using fiber integration for fibers with boundary we also define transgression along a compact oriented manifold with boundary.

Differential characters are \emph{thin invariant}:
A smooth singular chain $c \in C_k(X;\Z)$ is called \emph{thin} if the integral of any $k$-form over $c$ vanishes.
For instance this happens if $c$ is supported on a $(k-1)$-dimensional submanifold.
Differential characters of degree $k$ vanish on boundaries of thin $k$-chains.  
In particular, they are invariant under barycentric subdivision. 

We apply the notion of thin invariance to chain field theories, a modification of topological quantum field theories in the sense of Atiyah. 
Generalizing work of Bunke and others, we show that chain field theories are invariant under thin $2$-morphisms. 

\bigskip
\emph{Acknowledgment.}
It is a great pleasure to thank Matthias Kreck for very helpful discussion.
Moreover, the authors thank \emph{Sonderforschungsbereich 647} funded by \emph{Deutsche Forschungsgemeinschaft} for financial support.

%%%%%%%%%%%%%%%%%%%%%%%%%%%%%%%%%%%%%%%%%%%%%%%%%%%%%%%%%%%%%%%%%%%%%%%%%
\chapter{Smooth spaces} \label{sec:smoothspaces}
%%%%%%%%%%%%%%%%%%%%%%%%%%%%%%%%%%%%%%%%%%%%%%%%%%%%%%%%%%%%%%%%%%%%%%%%%

Differential characters were introduced by Cheeger and Simons in \cite{CS83} on finite-dimensional smooth manifolds.
We will need to consider differential characters on more general spaces $X$.
First of all, $X$ may be a manifold with a nonempty boundary.
Secondly, we have to allow certain infinite-dimensional spaces because we want to include examples such as the loop space $X=\LL(M)=C^\infty(S^1,M)$ of a finite-dimensional manifold $M$.
Thirdly, $X$ may also be any oriented compact regular $p$-stratifold as in \cite{K10}.
Stratifolds will be needed to represent homology classes.
\index{stratifold}

One convenient class of spaces to work with is that of differential spaces in the sense of Sikorski \cite{Si72}.
Recall their definition:

\begin{definition}\label{def:diffspace} \index{Definition!differential space}%
A \emph{differential space}\index{differential space} \index{space!differential $\sim$} is a pair $(X,C^\infty(X))$ where $X$ is a topological space and $C^\infty(X)$ is a subset of the set $C^0(X)$ of all continuous real-valued functions such that the following holds:
\begin{itemize}
\item 
\emph{Initial topology:}
$X$ carries the weakest topology for which all functions in $C^\infty(X)$ are continuous;
\item
\emph{Locality:}
If $f\in C^0(X)$ is such that for every point in $X$ there is a function $g\in C^\infty(X)$ coinciding with $f$ on a neighborhood of that point, then $f\in C^\infty(X)$;
\item
\emph{Composition with smooth functions:}
If $f_1,\ldots,f_k\in C^\infty(X)$ and $g$ is a smooth function defined on an open neighborhood of $f_1(X)\times f_k(X)\subset \R^k$, then $g\circ(f_1,\ldots,f_k)\in C^\infty(X)$.
\end{itemize}
\end{definition}
The functions in $C^\infty(X)$ are called \emph{smooth functions}.
A map $f:X\to Y$ between differential spaces is called \emph{smooth} if smooth functions on $Y$ pull back to smooth functions on $X$ along $f$.
This way we obtain the category of differential spaces.

On differential spaces one can define tangent vectors, $k$-forms, their exterior differential and one can pull back forms.
The usual rules such as Stokes's theorem apply \cite{MU94}.
In addition to that we will need that certain definitions of homology and cohomology which are equivalent in the case of manifolds remain equivalent.

\begin{definition}\label{def:smoothspace} \index{Definition!smooth space}%
A differential space is called a \emph{smooth space}\index{smooth space} \index{space!smooth $\sim$} if the following holds:
\begin{itemize}
\item 
\emph{Continuous versus smooth singular (co-)homology:}
The inclusion of the complex of smooth singular chains (with integral coefficients) into that of continuous singular chains induces isomorphisms for the corresponding homology and cohomology theories;
\item
\emph{deRham theorem:}
Integration of differential forms induces an isomorphism from deRham cohomology to smooth singular cohomology with real coefficients;
\item
\emph{Stratifold- versus singular homology:}
\index{stratifold homology}%
Pushing forward fundamental cycles induces an isomorphism from the bordism theory of oriented $p$-stratifolds to smooth singular homology theory with integral coefficients.
\end{itemize}
\end{definition}
Finite-dimensional manifolds (possibly with boundary), stratifolds and also infinite-dimensional Fr\'echet manifolds such as the loop space of a compact manifold are all examples for smooth spaces, see \cite[Ch.~7]{KM97} for infinite-dimensional manifolds and \cite{K10,E05} for stratifolds.

\begin{remark}
Instead of differential spaces one could also use diffeological spaces \index{diffeological space} \index{space!diffeological $\sim$} as in \cite{IZ13} to define smooth spaces in Definition~\ref{def:smoothspace}.
A smooth space would then be defined as a diffeological space satisfying the properties in Definition~\ref{def:smoothspace}.
These properties are not automatic; by \cite[p.~272]{IZ13} there are diffeological spaces for which the de Rham map fails to be an isomorphism.
\end{remark}

%%%%%%%%%%%%%%%%%%%%%%%%%%%%%%%%%%%%%%%%%%%%%%%%%%%%%%%%%%%%%%%%%%%%%%%%%
\chapter{Refined smooth singular homology}
\label{sec:refsingcohom}
%%%%%%%%%%%%%%%%%%%%%%%%%%%%%%%%%%%%%%%%%%%%%%%%%%%%%%%%%%%%%%%%%%%%%%%%%

Let $X$ be a smooth space in the sense explained above. 
For $n \in \N_0$, we denote by $C_n(X;\Z)$ \index{+CnXZ@$C_n(X;\Z)$, group of  singular chains} the abelian group of smooth singular $n$-chains in $X$ with integral coefficients.
The spaces of $n$-cycles and $n$-boundaries of the complex $(C_n(X;\Z),\partial)$ are denoted by $Z_n(X;\Z)$ \index{+ZnXZ@$Z_n(X;\Z)$, group of singular cycles} and $B_n(X;\Z)$,\index{+BnXZ@$B_n(X;\Z)$, group of singular boundaries} respectively.
Denote the space of smooth $n$-forms on $X$ by $\Omega^n(X)$.\index{+OnX@$\Omega^n(X)$, space of smooth $n$-forms}

\begin{definition}[Thin chains]\label{def:thin} \index{Definition!thin chains}%
A smooth singular chain $y\in C_n(X;\Z)$ is called {\em thin} \index{thin chain} if 
$$
\int_y \omega =0
$$
for all $\omega \in \Omega^{n}(X)$.
We denote by $S_n(X;\Z) \subset C_n(X;\Z)$ \index{+SnXZ@$S_n(X;\Z)$, group of thin $n$-chains} the subgroup of thin $n$-chains in $X$. 
\end{definition}

This definition of thin chains is similar to that of \emph{thin homotopies} in the literature, see e.g.\ \cite{B91,CP94}.
Thin homotopies will not occur in this paper, however.

If $X$ and $Y$ are smooth spaces and if $f:X\to Y$ is a smooth map, then if $c\in C_n(X;\Z)$ is thin, so is $f_*c\in C_n(Y;\Z)$.
Namely, for any $\omega \in \Omega^n(Y)$ we have
$$
\int_{f_*c}\omega = \int_c f^*\omega = 0.
$$
Hence $f_*(S_n(X;\Z)) \subset S_n(Y;\Z)$ and $f_*$ induces a homomorphism $f_* : C_n(X;\Z)/S_n(X;\Z) \to C_n(Y;\Z)/S_n(Y;\Z)$.

Denote the equivalence class of $c\in C_n(X;\Z)$ in $C_n(X;\Z)/S_n(X;\Z)$ by $[c]_{S_n}$.\index{+Sn@$[\cdot]_{S_n}$, equivalence class modulo thin chains}
By definition, integration of an $n$-form $\omega \in \Omega^n(X)$ descends to a linear map $C_n(X;\Z)/S_n(X;\Z) \to \R$, $[c]_{S_n} \mapsto \int_c\omega$.

Moreover, thin chains are preserved by the boundary operator. 
Namely, for $c\in S_{n+1}(X;\Z)$ and any $\eta\in\Omega^{n}$ we have by Stokes's theorem
$$
\int_{\partial c}\eta = \int_c d\eta = 0.
$$
Thus $\partial S_{n+1}(X;\Z) \subset S_{n}(X;\Z)$.
The boundary operator induces a homomorphism
$$
C_{n+1}(X;\Z)/S_{n+1}(X;\Z) \;\; \stackrel{\partial}{\longrightarrow} \;\; B_n(X;\Z) / \partial S_{n+1}(X;\Z) \,.
$$
Since $Z_n(X;\Z) \subset C_n(X;\Z)$ and $\partial S_{n+1}(X;\Z) \subset S_{n}(X;\Z)$ we have a natural homomorphism
\begin{equation}\label{map:inclproj}
Z_n(X;\Z)/\partial S_{n+1}(X;\Z) \;\; \longrightarrow \;\; C_n(X;\Z) / S_{n}(X;\Z)  \, .
\end{equation}
Denote the equivalence class of $z\in Z_n(X;\Z)$ in $Z_n(X;\Z)/\partial S_{n+1}(X;\Z)$ by $[z]_{\partial S_{n+1}}$.\index{+DSn@$[\cdot]_{\partial S_{n+1}}$, equivalence class modulo boundaries of thin chains}
Integration of differential forms induces well-defined maps
\[
\Omega^n(X) \times C_n(X;\Z) / S_{n}(X;\Z) \to \R , 
\quad
(\eta,[c]_{S_n}) \mapsto \int_{[c]_{S_n}} \eta := \int_{c} \eta , 
\]
and
\[
\Omega^n(X) \times Z_n(X;\Z) / \partial S_{n+1}(X;\Z) \to \R , 
\quad
(\eta,[z]_{\partial S_{n+1}}) \mapsto \int_{[z]_{\partial S_{n+1}}} \eta := \int_{z} \eta \, .
\]
Stokes's theorem says
$$
\int_{[c]_{S_n}} d\eta = \int_{\partial [c]_{S_{n}}} \eta .
$$
Recall that for a closed form $\omega \in \Omega^n(X)$, integration over a smooth singular cycle $z \in Z_n(X;\Z)$ corresponds to evaluation of the de Rham class $[\omega]_\mathrm{dR} \in \Hdr^n(X)$ on the homology class $[z] \in H_n(X;\Z)$, i.e.,
\index{+DR@$[\,\cdot\,]_\mathrm{dR}$, de Rham cohomology class} 
\index{+HdRnX@$\Hdr^n(X)$, $n$-the de Rham cohomology of $X$}
\index{+HnXZ@$H_n(X;\Z)$, $n$-th singular cohomology of $X$}
$$
\int_z \omega 
= \la [\omega]_\mathrm{dR} , [z] \ra \,.
$$

\begin{remark}
The quotients $C_n(X;\Z) / S_{n}(X;\Z)$ and $Z_n(X;\Z) / \partial S_{n+1}(X;\Z)$ are geo\-metrically very natural and appear in elementary constructions:
for instance, if $X$ is a closed smooth oriented $n$-manifold (or, more generally, an oriented compact $n$-dimensional regular $p$-stratifold without boundary) and if $c,c'\in Z_n(X;\Z)$ represent the fundamental class of $X$, then they are homologous, i.e., there exists $y\in C_{n+1}(X;\Z)$ with $c-c'=\partial y$.
For dimensional reasons $C_{n+1}(X;\Z)=S_{n+1}(X;\Z)$, hence $[c]_{\partial S_{n+1}} = [c']_{\partial S_{n+1}}$.
In fact, in this case $H_n(X;\Z) = Z_n(X;\Z) / B_{n}(X;\Z) = Z_n(X;\Z) / \partial S_{n+1}(X;\Z)$.

If $X$ has a boundary and $c,c' \in C_n(X;\Z)$ represent the fundamental class of $X$ in $H_n(X,\partial X;\Z)$, then we can find $y\in C_{n+1}(X;\Z)=S_{n+1}(X;\Z)$ such that $c-c'-\partial y$ is supported in the boundary of $X$ and is hence thin.
Therefore $[c]_{S_n} = [c']_{S_n}$ in this case.

Generalizations of these elementary observations are crucial for the construction of geometric chains in the next section.
\end{remark}

%%%%%%%%%%%%%%%%%%%%%%%%%%%%%%%%%%%%%%%%%%%%%%%%%%%%%%%%%%%%%%%%%%%%%%%%%
\chapter{Geometric chains}
\label{sec:geomchains}
%%%%%%%%%%%%%%%%%%%%%%%%%%%%%%%%%%%%%%%%%%%%%%%%%%%%%%%%%%%%%%%%%%%%%%%%%

We now define our notion of geometric chains.
The idea is to represent singular homology classes in $X$ by manifolds because this geometric description is well adapted for a geometric definition of fiber integration for Cheeger-Simons differential characters as we shall see. 
There is the problem however, that not all homology classes are representable by smooth manifolds.
Fortunately, Kreck's stratifolds \cite{K10} provide a suitable generalization of manifolds which repairs this defect.
\index{stratifold}

For $n\in \N_0$ let $\CC_n(X)$\index{+CnX@$\CC_n(X)$, group of geometric chains} \index{geometric chain} \index{chain!geometric $\sim$} be the set of diffeomorphism classes of smooth maps $f : M \to X$ where $M$ is an oriented compact $n$-dimensional regular $p$-stratifold with boundary, compare \cite[pp.~35 and 43]{K10}.
\index{+M@$M$, compact oriented $p$-stratifold}
Here two maps $f: M \to X$ and $f': M' \to X$ are called diffeomorphic if there is an orientation preserving diffeomorphism $F: M\to M'$ such that 
$$
\xymatrix{
M \ar[dr]^{f} \ar[d]_{F} &  \\
M' \ar[r]^{f'} & X
}
$$
commutes.
The equivalence class of $f:M\to X$ is denoted by $[M\xrightarrow{f}X]$.\index{+MfX@$[M\xrightarrow{f}X]$, geometric chain}
For $n<0$ put $\CC_n(X):=\{0\}$.
If $f:X \to Y$ is a smooth map, then we define $f_* : \CC_n(X) \to \CC_n(Y)$ by $f_*([M\xrightarrow{g}X]) := [M\xrightarrow{f\circ g}Y]$.

Disjoint union defines a structure of abelian semigroup on $\CC_n(X)$.
The boundary operator $\partial: \CC_n(X) \to \CC_{n-1}(X)$ is given by restriction to the geometric boundary.
For the boundary orientation we use the convention that an outward pointing tangent vector of $M$ at a regular point $p$ of $\partial M$ followed by an oriented basis of $T_p(\partial M)$ yields an oriented basis of $T_pM$.\index{orientation!of stratifolds with boundary}

We define a homomorphism $\phi_n: \CC_n(X) \to C_n(X;\Z)/S_n(X;\Z)$\index{+Phin@$\phi_n: \CC_n(X) \to C_n(X;\Z)/S_n(X;\Z)$} as follows:
For $f: M \to X$ choose a representing $n$-chain $c$ of the fundamental class of $M$ in $H_n(M,\partial M;\Z)$.
Then the equivalence class of $c$ in $C_n(M;\Z)/S_n(M;\Z)$ is independent of the particular choice of $c$ and we put $\phi_n([M\xrightarrow{f}X]) := [f_*(c)]_{S_n}$.

Similarly, if $\partial M = \emptyset$, then the equivalence class in $Z_n(M;\Z)/\partial S_{n+1}(M;\Z)$ of an $n$-cycle $c$ representing the fundamental class of $M$ in $H_n(M;\Z)$ does not depend on the particular choice of $c$ and we can define $\psi_n:\ZZ_n(X)\to Z_n(X;\Z)/\partial S_{n+1}(X;\Z)$\index{+Psin@$\psi_n:\ZZ_n(X)\to Z_n(X;\Z)/\partial S_{n+1}(X;\Z)$} by $\psi_n([M\xrightarrow{f}X]) := [f_*(c)]_{\partial S_{n+1}}$.

We call elements of $\CC_n(X)$ \emph{geometric chains} and elements of
\begin{align*}
\ZZ_n(X) &:= \{ \zeta \in \CC_n(X) \, | \, \partial \zeta =0 \} \\
\mbox{and} \qquad 
\BB_n(X) &:= \{ \zeta \in \CC_n(X) \,|\, \exists \beta \in \CC_{n+1}(X): \partial\beta=\zeta\} 
\end{align*}\index{+ZnX@$\ZZ_n(X)$, group of geometric cycles} \index{geometric cycle} \index{cycle!geometric $\sim$} \index{+BnX@$\BB_n(X)$, group of geometric boundaries} \index{geometric boundary} \index{boundary!geometric $\sim$}
\emph{geometric cycles} and \emph{geometric boundaries}, respectively.
We obtain the following commutative diagram:
\begin{equation}\label{eq:geomchainsboundary}
\xymatrix{
   \cdots \ar[r]
 & \CC_{n+1}(X) \ar[r]^{\partial} \ar[d]^{\phi_{n+1}} 
 & \BB_n(X) \ar[r]^{\mathrm{inclusion}} \ar[d]^{\psi_n|_{\BB_n(X)}} 
 & \ZZ_n(X) \ar[r]^{\mathrm{inclusion}} \ar[d]^{\psi_n} 
 & \CC_n(X)  \ar[d]^{\phi_n} \ar[r]
 & \cdots \\
   \cdots \ar[r]
 & \frac{C_{n+1}(X;\Z)}{S_{n+1}(X;\Z)} \ar[r]^{\partial}
 & \frac{B_n(X;\Z)}{\partial S_{n+1}(X;\Z)} \ar[r]^{\mathrm{inclusion}} 
 & \frac{Z_n(X;\Z)}{\partial S_{n+1}(X;\Z)} \ar[r]
 & \frac{C_n(X;\Z)}{S_{n}(X;\Z)}  \ar[r]
 & \cdots
}
\end{equation}
The map $Z_n(X;\Z)/\partial S_{n+1}(X;\Z) \; \to C_n(X;\Z)/S_{n}(X;\Z)$ is the one from \eqref{map:inclproj}.
Diagram~\eqref{eq:geomchainsboundary} is natural.
In particular, for any smooth map $f:X \to Y$ the diagram
$$
\xymatrix{
\CC_n(X) \ar[r]^{f_*} \ar[d]^{\phi_n} & 
\CC_n(Y)  \ar[d]^{\phi_n} \\
\frac{C_n(X;\Z)}{S_{n}(X;\Z)} \ar[r]^{f_*} &
\frac{C_n(Y;\Z)}{S_{n}(Y;\Z)}
}
$$
commutes and similarly for $\psi_n$.

From now on, we will, by slight abuse of notation, write $[\zeta]_{\partial S_{n+1}}$ instead of $\psi_n(\zeta)$ for $\zeta\in\ZZ_n(X)$ and $[\beta]_{S_n}$ instead of $\phi_n(\beta)$ for $\beta\in\CC_n(X)$.
\index{+ZetadelSn@$[\zeta]_{\partial S_{n+1}}$, image of $\zeta$ under $\psi_n$}
\index{+BetaSn@$[\beta]_{S_n}$, image of $\beta$ under $\phi_n$}

For an oriented stratifold $M$ we denote by $\overline{M}$ the same stratifold with reversed orientation.
\index{+Mn@$\overline{M}$, stratifold with reversed orientation}
\index{orientation!reversed $\sim$}%
Then $[M\xrightarrow{f}X]\mapsto[\overline{M}\xrightarrow{f}X]$ is an involution on $\CC_n(X)$ which commutes with $\partial$.
Furthermore, $\zeta + \overline\zeta \in \BB_n(X)$ for any $\zeta\in\ZZ_n(X)$ because $f \sqcup f: M \sqcup \overline M \to X$ is bounded by $f: [0,1] \times M \to X$.
In other words, the involution $\overline{\phantom{X}}: \ZZ_n(X) \to \ZZ_n(X)$ induces $-\id$ on homology, 
$$
[\overline{\zeta}] = - [\zeta] \mbox{ in } \HH_n(X) := \ZZ_n(X)/\BB_n(X).
$$
In particular, the geometric homology $\HH_n(X):=\ZZ_n(X)/\BB_n(X)$ is an abelian group, not just a semigroup.
\index{+HHnX@$\HH_n(X)$, geometric homology} \index{geometric homology}

The reason for using stratifolds instead of manifolds is the fact that the homomorphisms $\psi_n: \ZZ_n(X) \to Z_n(X;\Z)/\partial S_{n+1}(X;\Z)$ induce isomorphisms on homology (see \cite[Thm.~20.1]{K10})
:
$$
\HH_n(X) := \frac{\ZZ_n(X)}{\BB_n(X)}
\longrightarrow \frac{Z_n(X;\Z)/\partial S_{n+1}(X;\Z)}{B_n(X;\Z)/\partial S_{n+1}(X;\Z)} =  \frac{Z_n(X;\Z)}{B_n(X;\Z)}= H_n(X;\Z) \,.
$$
\index{stratifold homology}%
\index{geometric homology}%
The \emph{cross product} \index{cross product!of geometric chains} of geometric chains is defined by
\begin{align*}
\times: \CC_k(X) \otimes \CC_{k'}(X') &\to \CC_{k+k'}(X \times X'), \\
[M\xrightarrow{g}X] \otimes [M'\xrightarrow{g'}X'] &\mapsto [M \times M'\xrightarrow{g \times g'}X \times X'].
\end{align*}
By \cite[Thm.~20.1]{K10} this cross product in  $\HH_*$ is compatible with the usual cross product in $H_*$.
\index{+Timesa@$\times$, cross product of geometric chains}

\begin{remark}\label{rem:extend}
At various occasions we will have to extend homomorphisms $Z_n(X;\Z)\to G$ to homomorphisms $C_n(X;\Z) \to G$ where $G$ is an abelian group.
Since $B_{n-1}(X;\Z)$ is free, the exact sequence 
$$
0 \to Z_{n}(X;\Z) \xrightarrow{i} C_{n}(X;\Z) \xrightarrow{\partial} B_{n-1}(X;\Z) \to 0
$$ 
splits, though not canonically.
In particular, any basis of $Z_{n}(X;\Z)$ can be extended to a basis of $C_{n}(X;\Z)$. 
Therefore, any group homomorphism $Z_{n}(X;\Z) \to G$ can be extended as a group homomorphism to $C_{n}(X;\Z) \to G$ by defining it in an arbitrary manner on the complementary basis elements.
\end{remark}

\begin{lemma}[Representation by geometric chains]\label{lem:Liftazeta} \index{Lemma!representation by geometric chains}%
There are homomorphisms $\zeta : C_{n+1}(X;\Z) \to \CC_{n+1}(X)$, $a:C_n(X;\Z)\to C_{n+1}(X;\Z)$, and $y:C_{n+1}(X;\Z) \to Z_{n+1}(X;\Z)$ such that
\begin{align}
\partial \zeta(c) &= \zeta(\partial c) &\quad\quad \mbox{ for all } c\in C_{n+1}(X;\Z); \label{eq:delzetazetadel}\\
[\zeta(c)]_{S_{n+1}} &= [c - a(\partial c) - \partial a(c+y(c))]_{S_{n+1}}  &\quad\quad \mbox{ for all } c\in C_{n+1}(X;\Z); \label{eq:cadel}\\
[\zeta(z)]_{\partial S_{n+1}} &= [z-\partial a(z)]_{\partial S_{n+1}} &\quad\quad \mbox{ for all } z\in Z_{n+1}(X;\Z).\label{eq:zazeta}
\end{align}
\end{lemma}

\begin{proof}
a)
For any $z\in Z_n(X;\Z)$ the singular homology class represented by $z$ lies in the image of the map induced by $\psi_n$.
Hence we may choose a geometric cycle $\zeta(z) \in \ZZ_n(X)$ such that $[z]_{\partial S_{n+1}}-[\zeta(z)]_{\partial S_{n+1}}\in B_n(X;\Z)/\partial S_{n+1}(X;\Z)$.
We may thus choose a smooth singular chain $a(z) \in C_{n+1}(X;\Z)$ such that \eqref{eq:zazeta} holds.
In particular, if $z = \partial c \in B_n(X;\Z)$ is a smooth singular boundary, then $\zeta(z) = \zeta(\partial c) \in \BB_n(X)$ is a geometric boundary.

Since $Z_n(X;\Z)$ is free, the choices in $z\mapsto \zeta(z)$ and $z\mapsto a(z)$ can be made such that $\zeta:Z_n(X;\Z)\to \ZZ_n(X)$\index{+Zeta@$\zeta:Z_n(X;\Z)\to \ZZ_n(X)$} and $a:Z_n(X;\Z)\to C_{n+1}(X;\Z)$\index{+A@$a:Z_n(X;\Z)\to C_{n+1}(X;\Z)$} are homomorphisms.
One simply makes choices on elements of a basis of $Z_n(X;\Z)$ and extends as a homomorphism. 
In particular, we then have $\zeta(0) = 0$.
We perform this construction in all degrees $n \in \N_0$.
By Remark~\ref{rem:extend} we can extend $a$ to a homomorphism $a:C_n(X;\Z)\to C_{n+1}(X;\Z)$.

b)
We construct an extension of the homomorphism $\zeta$ to a homomorphism from singular chains to geometric chains such that it commutes with the boundary operations.
As an auxiliary tool, we first define a group homomorphism $\alpha:C_{n+1}(X;\Z) \to \CC_{n+1}(X)$ by choosing $\alpha(c)$ on basis elements and extending as a homomorphism.
On the basis elements of $Z_{n+1}(X;Z)$ we set $\alpha(c) = \zeta(c)$.
On the complementary basis elements we choose $\alpha(c)$ such that $\partial \alpha(c) = \zeta(\partial c)$.
This can be done since $\zeta(\partial c)$ is a geometric boundary. 
We then have 
\begin{equation}
[\partial (c - a(\partial c)-\partial a(c))]_{\partial S_{n+1}} 
= [\partial c - \partial a(\partial c)]_{\partial S_{n+1}} 
\stackrel{\eqref{eq:zazeta}}{=} 
[\zeta(\partial c)]_{\partial S_{n+1}} 
= \partial [\alpha(c)]_{S_{n+1}}.
\label{eq:delcac} 
\end{equation}
Hence there exists a smooth singular cycle $y(c) \in Z_{n+1}(X;\Z)$ such that 
\begin{equation}
[c - a(\partial c) -\partial a(c)- y(c)]_{S_{n+1}} = [\alpha(c)]_{S_{n+1}}.
\label{eq:cayalpha}
\end{equation}
We can choose $c \mapsto y(c)$ as a group homomorphism $y:C_{n+1}(X;\Z) \to Z_{n+1}(X;\Z)$ by defining it on basis elements, as explained above.
On the basis elements of $Z_{n+1}(X;\Z)$ we set $y(c) =0$.
Condition \eqref{eq:zazeta} implies that \eqref{eq:cayalpha} holds in this case. 
On the complementary basis elements, we choose $y(c) \in Z_{n+1}(X;\Z)$ such that \eqref{eq:cayalpha} holds.  

We have $\zeta(y(c)) \in \ZZ_{n+1}(X)$ and $a(y(c)) \in C_{n+2}(X;\Z)$ with 
$$
[y(c) - \partial a(y(c))]_{\partial S_{n+2}} =[\zeta(y(c))]_{\partial S_{n+2}}.
$$ 
If $c \in C_{n+1}(X;\Z)$ is a cycle we have $\alpha(c) + \zeta(y(c)) = \zeta(c) + \zeta(0)=\zeta(c)$. 
We may thus extend the homomorphism $\zeta:Z_{n+1}(X;\Z) \to \ZZ_{n+1}(X)$ constructed above to a homomorphism $\zeta:C_{n+1}(X;\Z) \to \CC_{n+1}(X)$ by setting $\zeta(c) := \alpha(c) + \zeta(y(c)) \in \CC_{n+1}(X)$.
We perform this construction in all degrees $n \in \N_0$.

c)
We have constructed a group homomorphism $\zeta : C_{n+1}(X;\Z) \to \CC_{n+1}(X)$ such that in addition to \eqref{eq:zazeta} we have for all $c\in C_{n+1}(X;\Z)$:
\begin{equation*}
\partial\zeta(c) = \partial\alpha(c)+\partial\zeta(y(c))=\zeta(\partial c)+0=\zeta(\partial c)
\end{equation*}
which is \eqref{eq:delzetazetadel} and 
\begin{align}
[c - a(\partial c) - \partial a(c+y(c))]_{S_{n+1}} 
&=
[c - a(\partial c) -\partial a(c)- y(c)]_{S_{n+1}} +  [y(c) - \partial a(y(c))]_{S_{n+1}}\nonumber\\
&\ist{\eqref{eq:cayalpha},\eqref{eq:zazeta}}\,\,\,
[\alpha(c)]_{S_{n+1}}+ [\zeta(y(c))]_{S_{n+1}} \nonumber\\
&=
[\zeta(c)]_{S_{n+1}}.
\end{align}
which is \eqref{eq:cadel}.
\end{proof}

Now we turn to fiber bundles.
Let $F \hookrightarrow E \twoheadrightarrow X$ be a fiber bundle whose fibers are compact oriented manifolds possibly with boundary.
\index{+F@$F$, typical fiber of a fiber bundle}
\index{+E@$E$, total space of a fiber bundle}
For $\zeta = [M\xrightarrow{g}X] \in \CC_k(X)$ (and $\zeta\in\ZZ_k(X)$ if $F$ has a boundary) let 
$$
\xymatrix{
g^*E \ar[d] \ar[r]^{\widetilde g} & E \ar[d] \\
M \ar[r]^g & X
}
$$
be the pull-back of the fiber bundle to $M$.
\index{+G@$\widetilde g$, pull-back map}
Since $M$ and $F$ do not both have a boundary, $g^*E$ is an $(k+\dim F)$-dimensional compact oriented stratifold with boundary.
The orientation of $g^*E$ \index{orientation!of fiber bundles} is chosen such that an oriented horizontal tangent basis (defined by the orientation of $M$) followed by an oriented tangent basis along the fiber yields an oriented tangent basis of the total space.
Put 
$$
\PBE(\zeta) := [g^*E\xrightarrow{\widetilde g}E] \in \CC_{k+\dim F}(E).
$$\index{+PB@$\PB_\bullet$, pull-back of geometric chains}
\index{pull-back operation on geometric chains}
This defines homomorphisms $\PBE: \ZZ_k(X) \to \CC_{k+\dim F}(E)$ and also $\PBE: \CC_k(X) \to \CC_{k+\dim F}(E)$ if $\partial F=\emptyset$.
The following holds:
\begin{itemize}
\item
For each $\zeta \in \ZZ_k(X)$ we have
\begin{equation}\label{axiomD:fiber_int_bound}
\partial(\PBE\zeta) 
= 
\begin{cases}
\overline{\PBdE(\zeta)}, & \mbox{ if $k$ is odd,} \\
\PBdE(\zeta), & \mbox{ if $k$ is even.}
\end{cases}
\end{equation}

\item
If $\partial F = \emptyset$, then we have for all $\zeta \in \CC_k(X)$ 
\begin{equation}\label{eq:partial_iota_commute}
\partial(\PBE\zeta) 
= 
\PBE(\partial\zeta).
\end{equation}

\item
$\PB_\bullet$ is natural in the following sense: 
Whenever we have a commutative diagram 
$$
\xymatrix{
E \ar[r]^H \ar[d] & E' \ar[d] \\
X \ar[r]^h & X'
}
$$
where $h$ is smooth and $H$ restricts to an orientation preserving diffeomorphism  $E_x \to E'_{h(x)}$ for any $x\in X$, then 
\begin{equation}\label{eq:PBnatuerlich}
\xymatrix{
\CC_{k+\dim F}(E) \ar[r]^{H_*} & \CC_{k+\dim F}(E') \\
\ZZ_{k}(X) \ar[u]_{\PBE} \ar[r]^{h_*} & \ZZ_k(X') \ar[u]_{\PB_{E'}}
}
\end{equation}
commutes (replace $\ZZ_k$ by $\CC_k$ if $\partial F=\emptyset$).
\index{+Ex@$E_x$, fiber of $E$ over $x$}
\item
$\PB_\bullet$ is compatible with integration of differential forms in the following sense:
For all differential forms $ \omega \in \Omega^{k+\dim F}(E)$ and all $\zeta \in \ZZ_k(X)$ we have
\begin{equation}
\int_{[\PB_E\zeta]_{S_{k+\dim F}}} \omega 
=
\int_{[\zeta]_{\partial S_{k+1}}} \fint_F \omega  \,. \label{axiomD:fiber_int_bound_forms}
\end{equation}
Here $\fint$ denotes the ordinary fiber integration of differential forms.
\index{+Int@$\fint$, fiber integration of differential forms} \index{fiber integration!for differential forms}
If $\partial F=\emptyset$ replace $[\zeta]_{\partial S_{k+1}}$ by $[\zeta]_{S_k}$ and demand \eqref{axiomD:fiber_int_bound_forms} for all $\zeta\in \CC_k(X)$.
\item 
$\PB_\bullet$ is functorial with respect to composition of fiber bundle projections: 
\index{functoriality!of pull-back operation}%
For a fiber bundle $\kappa:N \to E$ with compact oriented fibers over a fiber bundle $\pi: E \to X$ with compact oriented fibers, we have the composite fiber bundle $\pi \circ \kappa:N \to X$ with the composite orientation.
In this case, we have
\begin{equation}\label{eq:PB_ass}
\PB_{\pi \circ \kappa} 
= \PB_{\kappa} \circ \PB_{\pi}. 
\end{equation}
\item
$\PB_\bullet$ is compatible with the fiber product of bundles:
For fiber bundles $E \to X$ and $E' \to X'$ with compact oriented fibers and geometric chains \mbox{$\zeta = [M\xrightarrow{g}X] \in \CC_k(X)$} and \mbox{$\zeta' = [M'\xrightarrow{g'}X'] \in \CC_{k'}(X')$}, we have:
\begin{equation}\label{eq:PB_prod}
\PB_{E \times E'}(\zeta \times \zeta')
= 
(-1)^{k'\cdot\dim F} \PBE(\zeta) \times \PB_{E'}(\zeta') 
\in \CC_{k+k'+\dim F\times F'}(E \times E').  
\end{equation}
\end{itemize}

Properties \eqref{axiomD:fiber_int_bound}, \eqref{eq:partial_iota_commute}, \eqref{axiomD:fiber_int_bound_forms}, and \eqref{eq:PB_ass} are readily checked.
The sign in \eqref{eq:PB_prod} is caused by the conventions on orientations.
To verify \eqref{eq:PBnatuerlich} we observe that there is an orientation preserving diffeomorphism $J:E \to h^*E'$ such that
$$
\xymatrix{
E \ar[r]^H \ar[d] \ar@/^2pc/[rr]_J & E' \ar[d] & h^*E' \ar[l]_{\tilde h} \ar[d] \\
X \ar[r]^h \ar@/_1.9pc/[rr]^{\id}& X' & X \ar[l]_h
}
$$
commutes.
Now for any $\zeta = [M\xrightarrow{g}X] \in \CC_{k}(X)$ we get an induced orientation preserving diffeomorphism $g^*J : g^*E \to g^*h^*E'$ such that 
$$
\xymatrix{
g^*h^*E' \ar[r]^{\tilde h \circ \widetilde g} & E' \\
g^*E \ar[u]^{g^*J} \ar[ru]_{H\circ \widetilde g} & 
}
$$
commutes.
Thus $[g^*E \xrightarrow{H\circ \widetilde g} E']=[g^*h^*E'\xrightarrow{\tilde h \circ \widetilde g} E'] \in \CC_{k+\dim F}(E')$.
We compute
\begin{align*}
\PB_{E'}(h_*(\zeta)) 
&=
\PB_{E'}([M\xrightarrow{h \circ g}X']) \\
&=
[g^* h^* E' \xrightarrow{\tilde h \circ  \widetilde g}E'] \\
&=
[g^*E \xrightarrow{H\circ \widetilde g} E'] \\
&=
H_*([g^*E \xrightarrow{\widetilde g} E]) \\
&=
H_*(\PBE(\zeta)) 
\end{align*}
and \eqref{eq:PBnatuerlich} is shown.

\begin{remark}\label{rem:lambda}
\emph{Transfer map on cycles}.
We construct a transfer map on the level of singular cycles.
Let $\zeta:C_{k-\dim F}(X;\Z)\to \CC_{k-\dim F}(X)$ be the homomorphism from Lemma~\ref{lem:Liftazeta}.
We construct a homomorphism $\lambda:Z_{k-\dim F}(X;\Z) \to Z_{k}(E;\Z)$ such that
\begin{equation}
[\lambda(z)]_{\partial S_{k+1}} = [\PBE(\zeta(z))]_{\partial S_{k+1}}
\label{eq:lambdazzeta}
\end{equation} 
for all cycles $z\in Z_{k-\dim F}(X;\Z)$.
\index{+Lambda@$\lambda$, transfer map} \index{transfer map}
For any $z$ in a basis of $Z_{k-\dim F}(X;\Z)$ we choose a cycle $\lambda(z)\in Z_{k}(E;\Z)$ representing $[\PBE(\zeta(z))]_{\partial S_{k+1}}$ and extend $\lambda$ as a homomorphism.
In particular, $\lambda$ maps $B_{k-\dim F}(X;\Z)$ to $B_{k}(E;\Z)$.  
We perform this construction in all degrees $k \geq \dim F$.

\emph{Extension to chains}.
We extend the transfer map $\lambda:Z_{k-\dim F}(X;\Z) \to Z_k(E;\Z)$ to a homomorphism $\lambda:C_{k-\dim F}(X;\Z) \to C_k(E;\Z)$ in an appropriate manner.
First, we extend $\lambda:Z_{k-\dim F}(X;\Z) \to Z_k(E;\Z)$ to a homomorphism $\gamma:C_{k-\dim F}(X;\Z) \to C_k(E;\Z)$ as described in Remark~\ref{rem:extend}.
On the basis elements of $Z_{k-\dim F}(X;\Z)$ we set $\gamma(c):= \lambda(c)$.
On the complementary basis elements we choose $k$-chains $\gamma(c)$ such that $\partial \gamma(c) = \lambda(\partial c)$.
This is possible since $\lambda(\partial c)$ is a boundary.

We then have:
\begin{align}
\partial [\gamma(c)]_{S_k} 
&= [\partial \gamma(c)]_{\partial S_k} \nonumber \\
&= [\lambda(\partial c)]_{\partial S_k} \nonumber \\
&\ist{\eqref{eq:lambdazzeta}}\, [\PBE(\zeta(\partial c))]_{\partial S_k} \nonumber \\
&\ist{\eqref{eq:delzetazetadel}} [\PBE(\partial\zeta(c))]_{\partial S_k} \nonumber \\
&\ist{\eqref{eq:partial_iota_commute},\eqref{eq:geomchainsboundary}}\,\,\,\,\, \partial [\PBE(\zeta(c))]_{S_k} \nonumber \,.  
\end{align}
Hence there exists a cycle $w(c) \in Z_k(E;\Z)$ such that
\begin{equation}
[\gamma(c) - w(c)]_{S_k} = [\PBE(\zeta(c))]_{S_k} .
\label{eq:gammacw} 
\end{equation}
\index{+WCkdimF@$w:C_{k-\dim F}(X;\Z) \to Z_k(X;\Z)$}%
We can choose $c \mapsto w(c)$ as a group homomorphism $w:C_{k-\dim F}(X;\Z) \to Z_k(X;\Z)$ by defining it on basis elements, as explained above.
On the basis elements of $Z_{k-\dim F}(X;\Z)$ we set $w(c) =0$.
Condition \eqref{eq:lambdazzeta} implies that \eqref{eq:gammacw} holds in this case. 
On the complementary basis elements, we choose $w(c) \in Z_k(E;\Z)$ such that \eqref{eq:gammacw} holds.  

If $c \in C_{k-\dim F}(X;\Z)$ is a cycle, we have $\gamma(c) - w(c) = \lambda(c) + 0 = \lambda(c)$.
We set $\lambda(c):=\gamma(c) - w(c)$ for general $c\in C_{k-\dim F}(X;\Z)$.

\emph{Transfer map on chains}.
We have extended the transfer map on cycles to a group homomorphism $\lambda:C_{k-\dim F}(X;\Z) \to C_k(E;\Z)$\index{+Lambda@$\lambda:C_{k-\dim F}(X;\Z) \to C_k(E;\Z)$} with
\begin{equation}
\partial \lambda(c) = \lambda(\partial c)
\label{eq:dellambdalambdadel}
\end{equation}
and 
\begin{equation}
[\lambda(c)]_{S_k}
= [\gamma(c)-w(c)]_{S_k}
\stackrel{\eqref{eq:gammacw}}{=} [\PBE(\zeta(c))]_{S_k}
\label{eq:lambdaphizeta}
\end{equation} 
The transfer map $\lambda$ should be thought of as the pull-back mapping on the level of chains.
\end{remark}

\begin{remark}
\emph{Transfer map and fiber integration of differential forms}.
\index{fiber integration!for differential forms}
From \eqref{axiomD:fiber_int_bound_forms}, we conclude that for any differential form $\omega \in \Omega^k(E)$ and any smooth singular chain $c \in C_{k-\dim F}(X;\Z)$, we have:
\begin{equation}
\int_{\lambda(c)} \omega
=
\int_{[\zeta(c)]_{S_{k-\dim F}}} \fint_F \omega \,. \label{eq:fint_lambda}  
\end{equation}
In particular, if $\omega$ is a closed form, \eqref{eq:cadel} yields:
\begin{equation}
\int_{\lambda(c)} \omega
=
\int_{c-a(\partial c)} \fint_F \omega \,. \label{eq:fint_lambda_closed}
\end{equation}
For a cycle $z \in Z_{k-\dim F}(X;\Z)$ and $\omega \in \Omega^k(E)$, we also have:
\begin{equation}
\int_{\lambda(z)} \omega
=
\int_{[\zeta(z)]_{S_{k-\dim F}}} \fint_F \omega   
=
\int_{[\zeta(z)]_{\partial S_{k-\dim F+1}}} \fint_F \omega \,.  \label{eq:fint_lambda_cycle}  
\end{equation}
\end{remark}

\begin{remark}\label{rem:fiber_int_cohomology}
\emph{Transfer map and fiber integration on singular cohomology}.
Let $F \hookrightarrow E \to X$ be a fiber bundle with compact oriented fibers without boundary.
The construction of the Leray-Serre spectral sequence in \cite{S51} involves the construction of Eilenberg-Zilber type maps $EZ:C_p(X;\Z) \otimes C_q(F;\Z) \to E^0_{p,q}$ for all $p,q \in \N_0$.
\index{+EZ@$EZ$, Eilenberg-Zilber map} \index{Eilenberg-Zilber map} \index{Leray-Serre spectral sequence}%
These maps induce a map of bigraded chain complexes
$$
(C_\bullet(X;\Z) \otimes C_\bullet(F;\Z),\mathds 1 \otimes \partial_F) 
\xrightarrow{EZ} 
(E^0_{\bullet,\bullet},d_0) \,.
$$
The induced maps on homology yield identifications $C_p(X;H_q(F_x;\Z)) \xrightarrow{\overline{EZ}} E^1_{p,q}$.
Here $F_x$ denotes the fiber of the bundle over $x\in X$ and $\{H_q(F_x;\Z)\}_{x\in X}$ the corresponding local coefficient system.

We consider the special case $q=\dim(F)$.
Since the bundle $F \hookrightarrow E \to X$ has compact oriented fibers the local coefficient system $\{H_q(F_x;\Z)\}_{x\in X}$ has a canonical section $x\mapsto [F_x]$ where $[F_x]\in H_{\dim F}(F_x;\Z)$ is the fundamental class.

The maps $\Z\to H_{\dim(F)}(F_x;\Z)$, $k\mapsto k\cdot[F_x]$, induce a homomorphism of chain complexes
$$
(C_\bullet(X;\Z),\partial) \to
(C_\bullet(X;H_{\dim(F)}(F_x;\Z)) ,\partial) \xrightarrow{\overline{EZ}} 
(E^1_{\bullet,\bullet},d_1) \,.
$$
On the homology of the last two chain complexes we get the well-known identification $H_p(X;H_q(F_x;\Z)) \xrightarrow{\cong} E^2_{p,q}$ for the case $q=\dim F$.

Let $c \in C_{k-\dim F}(X;\Z)$ be a smooth singular chain in the base $X$.
Let $[\mu] \in H^k(E;\Z)$ be a cohomology class on the total space and $\mu \in C^k(E;\Z)$ a cocycle representing it.
Fiber integration for singular cohomology as constructed in \cite{BH53} maps the class $[\mu] \in H^k(E;\Z)$ to
\index{+Pi@$\pi_\ausruf$, fiber integration for singular cohomology}\index{fiber integration!for singular cohomology}
$$
\pi_![\mu]
:= \big[ c \mapsto \mu(\overline{EZ}(c \otimes [F_x])) \big] \in H^{k-\dim F}(X;\Z) \,. 
$$
By the constructions of the pull-back operation $\PBE$ on smooth chains and the transfer map $\lambda$ on singular chains, the chain $\lambda(c) \in C_k(E;\Z)$ represents the equivalence class $\overline{EZ}(( c - a(\partial c) - \partial a(c+y(c))) \otimes [F_x]) \in E^1_{k-\dim F,\dim F}$ of smooth singular $k$-chains in $E$.
Combining this observation with the definition of the map $\pi_!:H^k(E;\Z) \to H^{k-\dim F}(X;\Z)$ we obtain:
\begin{align}
\pi_![\mu]
&= 
\big[ c \mapsto \mu(\overline{EZ}(c \otimes [F_x])) \big] \nonumber \\
&=
\big[ c \mapsto \mu(\overline{EZ}(c \otimes [F_x])) \big] 
+ \big[ \delta(c \mapsto \mu(\overline{EZ}(a(c) \otimes [F_x]))) \big] \nonumber \\
&= 
\big[ c \mapsto \mu(\overline{EZ}(c \otimes [F_x])) \big] 
+ \big[ c \mapsto \mu(\overline{EZ}(a(\partial c)) \otimes [F_x]) \big] \nonumber \\
& \qquad 
+ \big[ c \mapsto \underbrace{\mu(\overline{EZ}(\partial a(c+y(c)) \otimes [F_x]))}_{=0} \big] \nonumber \\
&= 
\big[ c \mapsto \mu(\overline{EZ}(c - a(\partial c) - \partial a(c+y(c))) \otimes [F_x]) \big] \nonumber \\
&= 
\big[ c \mapsto \mu(\lambda(c)) \big] \nonumber \\
&=
[ \mu \circ \lambda ] .
\label{eq:fiber_int_singular_cohomology}
\end{align}
Thus pre-composition of cochains with the transfer map on chains yields the fiber integration on singular cohomology. 
\end{remark}

\begin{remark}
\emph{Transfer map on homology}.
As for fiber integration on singular cohomology, the Eilenberg-Zilber map from the Leray-Serre spectral sequence induces the so-called homology transfer $H_*(X;\Z) \to E^2_{*,\dim F} \twoheadrightarrow H_{*+\dim F}(E;\Z)$, $[z] \mapsto [\overline{EZ}(z \otimes [F_x])]$.
\index{homology transfer}%
\index{transfer!homology $\sim$}
\index{Leray-Serre spectral sequence}%
By construction, homology transfer is represented on the level of cycles by the transfer map $\lambda:Z_*(X;\Z) \to Z_{*+\dim F}(E;\Z)$ constructed in Remark~\ref{rem:lambda}.
Hence the name.
\end{remark}

\begin{remark}
\emph{Fiber integration, transfer and push-forward}.
In the literature, fiber integration is sometimes referred to as cohomology transfer.
Both homology and cohomology transfer can be defined for any smooth map between compact oriented smooth manifolds by conjugating the pull-back and push-forward maps with Poincar\'e duality, see e.g.~\cite[Ch.~VIII, \S~10]{D95}.
Therefore, fiber integration is also referred to as push-forward.  
\index{cohomology transfer} \index{homology transfer} \index{push-forward} \index{transfer!cohomology $\sim$}%
\end{remark}

%%%%%%%%%%%%%%%%%%%%%%%%%%%%%%%%%%%%%%%%%%%%%%%%%%%%%%%%%%%%%%%%%%%%%%%%%
\chapter{Differential characters}
%%%%%%%%%%%%%%%%%%%%%%%%%%%%%%%%%%%%%%%%%%%%%%%%%%%%%%%%%%%%%%%%%%%%%%%%%

Differential characters were introduced by Cheeger and Simons in \cite{CS83}.
The group $\widehat H^k(X;\Z)$ of differential characters in a smooth space has various equivalent descriptions.
\index{+HhatkXZ@$\widehat H^k(X;\Z)$, group of differential characters}
For instance, it is isomorphic to the smooth Deligne cohomology group $H^{k-1}_\DD(X;\Ul)$, see e.g.~\cite{CJM04}.
\index{Deligne cohomology}%
Differential characters can also be described by differential forms with singularities as in \cite{C73} or as de Rham-Federer currrents as in \cite{HL06,HL08,HLZ03}.
\index{de Rham-Federer currents}%
The groups of differential characters are often referred to as differential cohomology.
We use the original definition of differential characters due to Cheeger and Simons.

We first recall the definition and some elementary properties of Cheeger-Simons differential characters.
Then we give a new proof of a result of Simons and Sullivan saying that for any differential cohomology theory there is a unique natural transformation to the model given by differential characters.
Our proof yields an explicit formula for this natural transformation.
Similarly, we reprove the abstract uniqueness result for the ring structure due to Simons and Sullivan by deriving an explicit formula from the axioms.

Stratifolds enter the game because they can be used to represent \emph{homology classes}.
\index{stratifold}%
However, we do not modify the definition of differential characters as in \cite{BKS10}.
The usage of stratifolds in \cite{BKS10} to represent {\em cohomology classes} is responsible for the limitation to finite-dimensional manifolds.
Instead of stratifolds one could also use Baas-Sullivan pseudomanifolds.
\index{pseudomanifold}%
It was proposed in \cite{G99} to use them to describe differential characters.

\section{Definition and examples}\label{subsec:diff_char}
Let $X$ be a smooth space.
We denote by $\widehat H^k(X;\Z)$ the abelian group of degree $k\ge 1$ differential characters, i.e.\footnote{It is convenient to shift the degree of the differential characters by $+1$ as compared to the original definition from \cite{CS83}. 
Thus a degree $k$ differential character has curvature and characteristic class of degree $k$.},
\begin{equation}
\widehat H^k(X;\Z) 
:= 
\big\{\, h \in \Hom(Z_{k-1}(X;\Z),\Ul) \,\big|\, h \circ \partial \in \Omega^k(X) \,\big\} \,. \label{eq:def_diff_charact_1}
\end{equation}
\index{differential characters}
The notation $h \circ \partial \in \Omega^k(X)$ means that there exists a differential form $\omega \in \Omega^k(X)$ such that for every smooth singular chain $c \in C_k(X;\Z)$, we have: 
\begin{equation}
h(\partial c) 
=
\exp \Big( 2 \pi i  \int_c \omega \Big) \, .
\label{eq:def_diff_charact_2}
\end{equation}
The differential form $\omega$ is uniquely determined by the differential character $h \in \widehat H^k(X;\Z)$.
Moreover, it is closed and has integral periods.
This form $\omega =: \curv(h)$ is called the {\em curvature} of $h$.
\index{+Curv@$\curv$, curvature of a differential character} \index{curvature!of a differential character}
If $\curv(h) =0$, then $h$ is called a {\em flat} differential character. \index{flat differential character}%

Moreover, a differential character $h$ determines a class $c(h) \in H^k(X;\Z)$, constructed as follows:
\index{+Cc@$c$, characteristic class of a differential character} \index{characteristic class!of a differential character}
Since $Z_{k-1}(X;\Z)$ is a free $\Z$-module, there exists a real lift $\tilde h$ of the differential character $h$, i.e., $\tilde h \in \Hom(Z_{k-1}(X;\Z),\R)$ such that  $h(z) = \exp( 2\pi i \tilde h(z))$ for all $z \in Z_{k-1}(X;\Z)$.
\index{+Htilde@$\tilde h$, real lift of differential character $h$}
Then set
\begin{equation}
\mu^{\tilde h}: C_k(X;\Z) \to \Z , \quad
c \mapsto \int_c \curv(h) - \tilde h(\partial c) \,. \label{eq:charact_cocycle}
\end{equation}
\index{+Mu@$\mu$, cocycle for characteristic class}%
Since $\curv$ is closed $\mu^{\tilde h}$ is a cocycle, and it follows from equation \eqref{eq:def_diff_charact_2} that it takes integral values.
The cohomology class $[\mu^{\tilde h}] \in H^k(X;\Z)$ does not depend on the choice of the lift $\tilde h$.
Now $c(h) := [\mu^{\tilde h}] \in H^k(X;\Z)$ is called the {\em characteristic class} of $h$.
If $c(h)=0$, then $h$ is called a {\em topologically trivial} differential character. \index{topologically trivial differential character}%

By definition, any real lift $\tilde h$ of a differential character $h$ yields a cocycle for the characteristic class $c(h)$. 
Conversely, if $\mu \in C^k(X;\Z)$ is a cocycle representing the cohomology class $c(h) \in H^k(X;\Z)$, then we can find a real lift $\tilde h'$ such that $\mu = \mu^{\tilde h'} := \curv(h) - \delta \tilde h'$.
For if $\tilde h$ is any real lift of $h$, then $\mu$ and $\mu^{\tilde h}$ are cohomologous, i.e.~there exists a cochain $t \in C^{k-1}(X;\Z)$ such that $\delta t = \mu^{\tilde h} - \mu$.
Setting $\tilde h' := \tilde h + t$ yields a real lift of $h$ with $\mu^{\tilde h'} = \curv(h) - \delta\tilde h - \delta t = \mu^{\tilde h} - \delta t = \mu$.  

Note that by \eqref{eq:charact_cocycle}, the image of $c(h)$ in $H^k(X;\R)$ coincides with the image of the de Rham cohomology class $[\curv(h)]_\mathrm{dR}$ of $\curv(h)$ under the de Rham isomorphism.

\begin{remark}\label{rem:Hkadditive}
Even though the abelian group $\Ul$ is written multiplicatively, we write $\widehat H^k(X;\Z)$ additively, i.e., for $h,h'\in \widehat H^k(X;\Z)$ and $z\in Z_{k-1}(X;\Z)$ we have
$$
(h+h')(z) = h(z)\cdot h'(z).
$$
The neutral element $0\in\widehat H^k(X;\Z)$ is the constant map
$$
0(z) = 1.
$$
The reason for this convention is that there is an additional multiplicative structure on $\widehat H^*(X;\Z)$ analogous to the cup product turning it into a ring.
The ring structure will be discussed in Section~\ref{subsec:ring}.
\end{remark}

Let $\eta \in \Omega^{k-1}(X)$ be a differential form on $X$.
We define a differential character $\iota(\eta)\in \widehat H^k(X;\Z)$ by setting 
\begin{equation}\label{eq:def_iota}
\iota(\eta)(z) 
:= 
\exp \Big( 2\pi i \int_z \eta \Big) \,.
\end{equation}
\index{+Iota@$\iota:\Omega^{*-1}(X) \to \widehat H^*(X;\Z)$}
Evaluating on boundaries, we see that in this case,
\begin{equation}
\curv(\iota(\eta)) = d \eta .
\label{eq:curv_deta}
\end{equation}
Taking $\widetilde{\iota(\eta)} (z):= \int_z \eta$ as real lift, we have by Stokes's theorem 
$$
\mu^{\widetilde{\iota(\eta)}}(x) 
= \int_x d\eta - \widetilde{\iota(\eta)}(\partial x) 
= \int_x d\eta - \int_{\partial x} \eta 
= 0
$$
so that $h$ is topologically trivial.
If also $d\eta =0$, then $\curv(\iota(\eta)) =0$, thus $h$ is flat.

We thus obtain a homomorphism $\iota: \Omega^{k-1}(X) \to \widehat H^k(X;\Z)$.
If the closed form $\eta$ has integral periods, then $\iota(\eta)(z)=1$ for every $z$, thus $\iota(\eta)=0$.
A form $\eta \in \Omega^{k-1}(X)$ such that $\iota(\eta) = h \in \widehat H^k(X;\Z)$ is called a {\em topological trivialization} of $h$.\index{topological trivialization!of a differential character}%

Let $u \in H^{k-1}(X;\Ul)$. 
\index{+HkXU@$H^*(X;\Ul)$, smooth singular cohomology with $\Ul$-coefficients}
We define a differential character $j(u) \in \widehat H^k(X;\Z)$ by setting 
\begin{equation}\label{eq:def_j}
j(u)(z):= \langle u,[z] \rangle. 
\end{equation}
Thus we obtain an injective map $j:H^{k-1}(X;\Ul) \to \widehat H^k(X;\Z)$.  
\index{+J@$j:H^{*-1}(X;\Ul) \to \widehat H^*(X;\Z)$}

By $\Omega^{k}_{cl}(X)$ we denote the space of closed $k$-forms and by $\Omega_0^k(X)\subset \Omega^{k}_{cl}(X)$ the set of closed $k$-forms with integral periods.
\index{+OmegakclX@$\Omega^{k}_{cl}(X)$, space of closed $k$-forms}
\index{+OmegaokX@$\Omega_0^k(X)$, space of closed $k$-forms with integral periods}
We identify the quotients 
$$
\frac{H^{k}(X;\R)}{H^{k}(X;\Z)_\R} \cong \frac{\Omega^{k}_{cl}(X)}{\Omega^k_0(X)}
$$ 
using the de Rham isomorphism.
Here $H^k(X;\Z)_\R \subset H^k(X;\R)$ denotes the image of $H^k(X;\Z)$ in $H^k(X;\R)$ under the natural map induced by the change of coefficients.
Recall that $\iota:\Omega^{k-1}(X) \to \widehat H^k(X;\Z)$ induces a homomorphism $\tfrac{\Omega^{k-1}(X)}{\Omega^{k-1}_0(X)} \to \widehat H^k(X;\Z)$, again denoted $\iota$.

We obtain the following commutative diagram with exact rows and columns:
\begin{equation}
\xymatrix{
& 0 \ar[d] & 0 \ar[d] & 0 \ar[d] & \\
0 \ar[r] & \frac{H^{k-1}(X;\R)}{H^{k-1}(X;\Z)_\R} \ar[d] \ar[r] 
  & \frac{\Omega^{k-1}(X)}{\Omega^{k-1}_0(X)} \ar[d]^{\iota} \ar[r]^d & d\Omega^{k-1}(X) \ar[d] \ar[r] & 0 \\
0 \ar[r] & H^{k-1}(X;\Ul) \ar[d] \ar[r]^j & \widehat H^k(X;\Z) \ar[d]^c \ar[r]^{\curv} 
  & \Omega^k_0(X) \ar[d] \ar[r] & 0 \\
0 \ar[r] & \Ext(H_{k-1}(X;\Z),\Z) \ar[d] \ar[r] & H^k(X;\Z) \ar[d] \ar[r] 
  & \Hom(H_k(X;\Z),\Z) \ar[d] \ar[r] & 0 \\
& 0 & 0 & 0 &
}
\label{eq:3x3diagram}
\end{equation}
The left column is obtained from the long exact cohomology sequence induced by the coefficient sequence $0 \to \Z \to \R \to \Ul \to 0$ together with the canonical identification of $\Ext(H_{k-1}(X;\Z),\Z)$ with the torsion subgroup of $H^k(X;\Z)$.
The middle column says that a differential character admits a topological trivialization if and only if it is topologically trivial.

For reasons that will become apparent later, we extend the definition of the group $\widehat H^k(X;\Z)$ by setting
\begin{equation}\label{eq:def_Hk_negativ}
\widehat H^k(X;\Z) := H^k(X;\Z) \qquad \mbox{for $k\leq 0$.} 
\end{equation}
This is the only possible choice compatible with the diagram \eqref{eq:3x3diagram}.
In particular, we have $\widehat H^k(X;\Z) = \{0\}$ for $k < 0$.
For $k \leq 0$, we define the characteristic class $c:\widehat H^k(X;\Z) \to H^k(X;\Z)$ to be the identity.

\begin{remark}\label{rem:thin_invariance}
\emph{Thin invariance}.
By construction, the evaluation of differential characters is well defined on $Z_{k-1}(X;\Z)/\partial S_k(X;\Z)$:
If $z \in Z_{k-1}(X;\Z)$ with $z = \partial y$ and $\int_y \eta =0$ for all $\eta \in \Omega^k(X)$, then we find:
$$
h(z)
= h(\partial y) 
= \exp \Big( 2\pi i \underbrace{\int_y \curv(h)}_{=0} \Big) 
= 1 \,. 
$$
We refer to this property of differential characters as \emph{thin invariance}.\index{thin invariance!of differential characters}%

In particular, differential characters are invariant under barycentric subdivision of smooth singular cycles.
This was already observed in \cite[p.~55]{CS83}.
\end{remark}

\begin{remark}\label{rem:def_f*h}
\emph{Naturality}.
If $f:X\to Y$ is a smooth map, then one can pull back differential characters $h\in \widehat H^k(Y;\Z)$ on $Y$ to $X$ by
$$
f^*h := h\circ f_* 
$$
where $f_* :Z_{k-1}(X;\Z) \to Z_{k-1}(Y;\Z)$ is the induced map on cycles.
This defines a homomorphism $f^*:\widehat H^k(Y;\Z) \to \widehat H^k(X;\Z)$.
One easily checks that $\curv(f^*h)=f^*\curv(h)$ and $c(f^*h) = f^*c(h)$.
\index{naturality!of differential characters}%
\index{differential characters!naturality}%
\end{remark}

\begin{remark}\label{rem:torsion_cycles}
\emph{Evaluation on torsion cycles.}
Let $h \in \widehat H^k(X;\Z)$ and let $z \in Z_{k-1}(X;\Z)$ be a cycle that represents a torsion class in $H_{k-1}(X;\Z)$.\index{torsion cycles}%
Hence there exists an $N \in \N$ such that $N \cdot [z] =0 \in H_{k-1}(X;\Z)$.
Choose $x \in C_k(X;\Z)$ such that $N \cdot z = \partial x$.
In particular, $z = \frac{1}{N} \cdot \partial x$ as real cycles.
Then we have:
\begin{align}
h(z)
&= \exp \Big(2\pi i \cdot  \tilde h(z) \Big) \notag \\
&= \exp \Big( 2\pi i \cdot \tilde h \big( \frac{1}{N} \cdot \partial x\big) \Big) \notag \\
&= \exp \Big( \frac{2\pi i}{N} \tilde h(\partial x) \Big) \notag \\
&= \exp \Big( \frac{2\pi i}{N} \delta \tilde h(x) \Big) \notag \\
&= \exp \frac{2\pi i}{N} \Big( \int_x\curv(h) - \mu^{\tilde h}(x) \Big). \notag 
\end{align}
If $\mu \in Z^k(X;\Z)$ is another cocycle representing the characteristic class $c(h)$, then we have $\mu^{\tilde h} - \mu  = \delta t$ for some $t \in C^{k+1}(X;\Z)$.
This yields 
$$
\frac{1}{N} \cdot (\mu^{\tilde h} - \mu)(x) = \delta t(\frac{1}{N} \cdot x) = t(\frac{1}{N} \cdot \partial x) = t(z) \in \Z.
$$
Thus although the evaluation of $c(h)$ on $x$ is not well defined, by abuse of notation we may write
\begin{equation}
h(z) = 
\exp  \frac{2\pi i }{N} \Big(\int_x\curv(h) - \langle c(h),x\rangle \Big). \label{eq:torsion_cycles}
\end{equation}
In particular, if $h$ is topologically trivial and flat, then it vanishes on torsion cycles.

The latter fact can also be deduced from the commutative diagram \eqref{eq:3x3diagram}: if $h$ is in the image of the map $\frac{H^{k-1}(X;\R)}{H^{k-1}(X;\Z)_\R} \to \widehat H^k(X;\Z)$, then the real lift $\tilde h$ can be chosen to be a real cocycle. 
Thus $\tilde h$ vanishes on torsion cycles, and so does $h$. 
\end{remark}

\begin{remark}\label{rem:diff_charact_stratifolds}
Let $h \in \widehat H^k(X;\Z)$ be a differential character on a smooth space $X$, and let $z \in Z_{k-1}(X;\Z)$ be a smooth singular cycle.
According to Lemma~\ref{lem:Liftazeta}, we get a geometric cycle $\zeta(z) = [M\xrightarrow{g}X] \in \ZZ_{k-1}(X)$ and a smooth singular chain $a(z) \in C_k(X;\Z)$ such that $[z - \partial a(z)]_{\partial S_k} = [\zeta(z)]_{\partial S_k}$.
Since differential characters are thin invariant, we have 
$$
h(z)
=
h([\zeta(z)]_{\partial S_k})\cdot h(\partial a(z))
=
h([\zeta(z)]_{\partial S_k}) \cdot \exp \Big( 2\pi i \int_{a(z)} \curv(h) \Big) \,. 
$$
We may also pull back the differential character along the smooth map $g$ to the stratifold $M$.
For dimensional reasons, $g^*h$ is topologically trivial and flat, hence $g^*h = \iota(\varrho)$ for a closed differential form $\varrho \in \Omega^{k-1}_{cl}(M)$.
By definition, the evaluation of $h$ on $[\zeta(z)]_{\partial S_k}$ is the same as the evaluation of $g^*h$ on any representing chain of the fundamental class $[M] \in H_{k-1}(M;\Z)=Z_{k-1}(M;\Z)/\partial S_k(M;\Z)$ of the stratifold $M$. 
So we may write:
\begin{align}
h(z)
&= 
h([\zeta(z)]_{\partial S_k}) \cdot \exp \Big( 2\pi i \int_{a(z)} \curv(h) \Big) \label{eq:evaluation_diff_charact_stratifold_1} \\
&= 
g^*h([M]) \cdot \exp \Big( 2\pi i \int_{a(z)} \curv(h) \Big) \label{eq:evaluation_diff_charact_stratifold_2} \\
&= 
\exp \Big( 2\pi i \int_M \varrho \Big) \cdot \exp \Big( 2\pi i \int_{a(z)} \curv(h) \Big) \,. 
  \label{eq:evaluation_diff_charact_stratifold_3} 
\end{align}

We check that \eqref{eq:evaluation_diff_charact_stratifold_2} is consistent with the property \eqref{eq:def_diff_charact_2} that defines differential characters: for a boundary $z = \partial c \in B_{k-1}(X;\Z)$ we choose $\zeta(c) \in \CC_k(X)$ and $a(c) \in C_k(X;\Z)$ as in Lemma~\ref{lem:Liftazeta}.
This yields:
\begin{align*}
h(\partial c)
&= 
h(\partial [\zeta(c)]_{S_k}) \cdot \exp \Big( 2\pi i \int_{a(\partial c)} \curv(h) \Big) \notag \\
&= 
\exp \Big[ 2\pi i \Big( \int_{[\zeta(c)]_{S_k}} \curv(h) + \int_{a(\partial c)} \curv(h) \Big)\Big] \notag \\
&\ist{\eqref{eq:cadel}}
\exp \Big[ 2 \pi i \Big( \int_{c} \curv(h) - \underbrace{\int_{\partial (a(c+y(c))} \curv(h)}_{=0} \Big)\Big] \\
&= 
\exp \Big( 2 \pi i \int_c \curv(h) \Big). 
\end{align*}
\end{remark}

We identify differential characters in low degrees as mentioned in \cite[p.~54]{CS83}.

\begin{example}\label{ex:H1}
\emph{$\Ul$-valued smooth functions}.
Let $X$ be a differentiable manifold and let $k=1$.
We show $\widehat H^1(X;\Z)=C^\infty(X,\Ul)$.
Any homomorphism $h:Z_{0}(X;\Z)\to\Ul$ corresponds to a map $\bar h:X\to \Ul$.
\index{+Hbar@$\bar h$, smooth $\Ul$-valued map}
For a fixed point $x_0\in M$ we identify a neighborhood of $x_0$ with a ball such that $x_0$ corresponds to its center.
For $x$ in this neighborhood we let $y(x)$ be the straight line from $x_0$ to $x$.
By \eqref{eq:def_diff_charact_2} we have
$$
\bar h(x)=\bar h(x_0)\cdot \exp\Big( 2 \pi i \cdot \int_{y(x)} \curv(h) \Big) .
$$
This shows that $\bar h$ is smooth.

Conversely, given a smooth function $\bar h:X\to\Ul$, we choose a smooth local lift   $\tilde h:U\subset X\to\R$, i.e., $\exp(2\pi i\tilde h(x))=\bar h(x)$, and put $\omega := d\tilde h$.
This form $\omega$ does not depend on the choice of lift and is therefore a globally defined $1$-form on $X$.
Now $h:Z_{0}(X;\Z)\to\Ul$ given by $h(\Sigma_j\alpha_j x_j)=\prod_j\bar h(x_j)^{\alpha_j}$ is a differential character with curvature $\omega$.
Hence $h$ is flat if and only if $\bar h$ is locally constant.
Moreover, $h$ is topologically trivial if and only if $\bar h$ has a \emph{global} lift $\tilde h:X\to\R$.

For the characteristic class one can check that
$$
c(h) = \bar h^*\theta
$$
where $\theta\in H^1(\Ul;\Z)$ is the fundamental class.
From now on we will identify $\widehat H^1(X;\Z)=C^\infty(X,\Ul)$ and not distinguish between $h\in\widehat H^1(X;\Z)$ and $\bar h \in C^\infty(X,\Ul)$.
\index{differential characters!of degree $1$}%
\end{example}

\begin{example}\label{ex:U1Buendel}
\emph{$\Ul$-bundles with connection}.
Let $X$ be a differentiable manifold and let $k=2$.
For a $\Ul$-bundle with connection $(P,\nabla)$ on $X$, the holonomy map associates to each smooth $1$-cycle $z$ an element $h(z)\in\Ul$.
\index{holonomy!of a line bundle with connection}%
\index{+Pnabla@$(P,\nabla)$, $\Ul$-bundle with connection}%
\index{+Nabla@$\nabla$, connection on $\Ul$-bundle or complex line bundle}%
Let $\PP_c^\nabla$ denote parallel transport along an oriented curve $c$ with respect to the connection $\nabla$.
\index{+Pp@$\PP$, parallel transport}% 
\index{parallel transport!for line bundle with connection}%
If $c$ is closed and $z$ is the cycle represented by $c$, then $h(z)$ is characterized by $\PP_c^\nabla (p)=p\cdot h(z)$.
Here $p\in P$ lies in the fiber over the initial point of $c$ and $h(z)$ does not depend on its choice.

This defines a differential character $h\in\widehat H^2(X;\Z)$ whose curvature is $\curv(h)=\frac{-1}{2\pi i}R^\nabla$, where $R^\nabla$ is the curvature of $\nabla$.
\index{+Rnabla@$R^\nabla$, curvature of $\nabla$}%
\index{curvature!of a line bundle with connection}%
The characteristic class $c(h)$ is the first Chern class of $P$.
\index{Chern class}

Conversely, any $h\in\widehat H^2(X;\Z)$ is the holonomy map of a $\Ul$-bundle with connection and determines the bundle up to connection-preserving isomorphism.
Hence differential characters in $\widehat H^2(X;\Z)$ are in 1-1 correspondence with isomorphism classes of $\Ul$-bundles with connection.
\index{differential characters!of degree $2$}%

\emph{Change of connections}.
Given a $\Ul$-bundle with connection $(P,\nabla)$ and a $1$-form $\rho\in\Omega^1(X)$, we get a new connection $\nabla'=\nabla + i\rho$ on $P$.
The differential character corresponding to $(P,\nabla')$ is obtained by adding $\iota(\tfrac{-1}{2\pi}\rho)$ to the character corresponding to $(P,\nabla)$.

\emph{Topological trivializations}.
If the $\Ul$-bundle $P \to X$ is topologically trivial, any trivialization $T:P \to X \times \Ul$ yields a 1-1 correspondence of connections $\nabla$ on $P$ and differential forms $\vartheta(\nabla,T) \in \Omega^1(X)$.\index{topological trivialization!of a line bundle}%
Under this correspondence, the connection $1$-form of $\nabla$ is given as $(T \circ \pr_1)^*( -2\pi i \vartheta(\nabla,T))$.
\index{+ThetanablaT@$\vartheta$, connection $1$-form}%
\index{connection $1$-form}%
Parallel transport along a curve $c$ in $X$ with respect to a connection $\nabla$ on $P$ corresponds to multiplication with $\exp \big(2\pi i \int_c \vartheta(\nabla,T) \big)$.
In particular, the holonomy map of $(P,\nabla)$ is given as $c \mapsto \exp \big(2\pi i \int_c \vartheta(\nabla,T) \big)$, hence $h = \iota(\vartheta)$.

Conversely, given a $1$-form $\varrho \in \Omega^1(X)$ such that $h= \iota(\varrho)$, then the first Chern class of the correponding $\Ul$-bundle $P$ vanishes, hence $P$ is topologically trivial.
One can directly construct global sections and hence trivializations of the bundle $P$ from the $1$-form $\varrho$.
This is explained in detail in Example~\ref{ex:H2rel} below.

\emph{Flat bundles}.
If $P \to X$ is a $\Ul$-bundle which admits a flat connection $\nabla$, then $c_1(P)$ is a torsion class.
The holonomy of $\nabla$ along a closed curve now only depends on the homotopy class of the curve and thus yields an element in $\Hom(\pi_1(X),\Ul) \cong \Hom(H_1(X);\Ul) \cong H^1(X;\Ul)$.
\index{flat line bundle}%

Conversely, for any homomorphism $\chi:\pi_1(X) \to \Ul$, the $\Ul$-bundle $P:= \widetilde X \times_\chi \Ul$ associated to the universal cover via the representation $\chi$ has Chern class $c_1(P) = \chi \in \Hom(\pi_1(X),\Ul) \cong H^1(X;\Ul)$.
The canonical flat connection on the trivial bundle $\widetilde X \times \Ul$ descends to a flat connection on $P$ with holonomy map $\chi$.      

The 1-1 correspondence between isomorphism classes of flat bundles and homomorphisms $\pi_1(X) \to \Ul$ thus obtained corresponds to the isomorphism $j:H^1(X,\Ul) \xrightarrow{\cong} \widehat H^2_\mathrm{flat}(X;\Z)$ of diagram \eqref{eq:3x3diagram}.
\end{example}

\begin{example}
\emph{Hitchin gerbes with connection}.
Let $X$ be a differentiable manifold and let $k=3$.
Similar to the case $k=2$ and $\Ul$-bundles with connection, there is a 1-1 correspondence between differential characters in $\widehat H^3(X;\Z)$ and isomorphism classes of Hitchin gerbes with connection \cite{Hit}.
\index{Hitchin gerbes}%
\end{example}

\section{Differential cohomology}
There are several ways to define differential cohomology axiomatically as a functor $\Ht^*(\cdt;\Z)$ from the category of smooth spaces to the category of $\Z$-graded abelian groups, together with natural transformations $\curvt:\Ht^*(\cdt;\Z) \to \Omega^*_0(\cdt)$ (curvature), $\ct:\Ht^*(\cdt;\Z) \to H^*(\cdt;\Z)$ (characteristic class), $\iotat:\Omega^{*-1}(\cdt)/\Omega^{*-1}_0(\cdt) \to \Ht^*(\cdt;\Z)$ (topological trivialization) and $\jt:H^{*-1}(\cdt;\Ul) \to \Ht^*(\cdt;\Z)$ (inclusion of flat classes).
One difference of our definition from those used in \cite{SS08} and \cite{BS10} is that we require the functor to be defined on a class of spaces also containing stratifolds.
\index{stratifold}

\begin{definition}[Differential cohomology theory]\label{def:diff_cohom} \index{Definition!differential cohomology theory}%
A \emph{differential cohomology} theory \index{differential cohomology theory} is a functor $\Ht^*(\cdt;\Z)$ from the category of smooth spaces to the categoy of $\Z$-graded abelian groups, together with four natural transformations 
\begin{itemize}
\item 
$\curvt:\Ht^*(\cdt;\Z) \to \Omega^*_0(\cdt)$, called \emph{curvature},
\item 
$\ct:\Ht^*(\cdt;\Z) \to H^*(\cdt;\Z)$, called \emph{characteristic class}, 
\item
$\iotat:\Omega^{*-1}(\cdt)/\Omega^{*-1}_0(\cdt) \to \Ht^*(\cdt;\Z)$, called \emph{topological trivialization}, and
\item 
$\jt:H^{*-1}(\cdt;\Ul) \to \Ht^*(\cdt;\Z)$, called \emph{inclusion of flat classes},
\end{itemize}
such that for any smooth space $X$ the following diagram commutes and has exact rows and columns: 
\index{+HtildeZ@$\Ht^*(\cdt;\Z)$, differential cohomology theory}
\begin{equation}
\xymatrix{
& 0 \ar[d] & 0 \ar[d] & 0 \ar[d] & \\
0 \ar[r] & \frac{H^{*-1}(X;\R)}{H^{*-1}(X;\Z)_\R} \ar[d] \ar[r] 
  & \frac{\Omega^{*-1}(X)}{\Omega^{*-1}_0(X)} \ar[d]^{\iotat} \ar[r]^d & d\Omega^{*-1}(X) \ar[d] \ar[r] & 0 \\
0 \ar[r] & H^{*-1}(X;\Ul) \ar[d] \ar[r]^{\jt} & \Ht^*(\cdt;\Z) \ar[d]^{\ct} \ar[r]^{\curvt} 
  & \Omega^*_0(X) \ar[d] \ar[r] & 0 \\
0 \ar[r] & \Ext(H_{*-1}(X;\Z),\Z) \ar[d] \ar[r] & H^*(X;\Z) \ar[d] \ar[r] 
  & \Hom(H_*(X;\Z),\Z) \ar[d] \ar[r] & 0 \\
& 0 & 0 & 0 &
}
\label{eq:3x3diagram_Q}
\end{equation}
\end{definition}

\begin{remark}
Note that the upper and lower rows as well as the left and right columns of \eqref{eq:3x3diagram_Q} are exact sequence, independently of the differential cohomology theory $\Ht^*(\cdt;\Z)$.
Thus the requirement is that the middle row and column are exact sequences and the whole diagram commutes.
Commutativity of the right upper quadrant means that $\curvt \circ \iotat$ is the exterior differential.
Commutativity of the left lower quadrant means that $\ct \circ \jt$ is the connecting homomorphism in cohomology for the coefficient sequence $0 \to \Z \to \R \to \Ul \to 0$.
Hence our definition of differential cohomology coincides with that of \emph{character functors} in \cite[p.~46]{SS08}.\index{character functor}%   
\end{remark}

In this section, we show uniqueness of differential cohomology theories up to unique natural transformations.
More precisely, for any differential cohomo\-logy theory $\Ht^*(\cdt;\Z)$, there exists a unique natural transformation \mbox{$\Xi:\Ht^*(\cdt;\Z) \to \widehat H^*(X;\Z)$} that commutes with the identity on the other functors in diagram~\ref{eq:3x3diagram_Q}. 
Equivalent statements were proved in \cite[Thm.~1.1]{SS08} and in \cite[Thm.~3.1]{BS10}.
Our proof differs from both in that for any fixed smooth space $X$ we obtain an explicit formula for $\Xi:\Ht^*(X;\Z) \to \widehat H^*(X;\Z)$.
However, we rely on \cite[Lemma~1.1]{SS08} to conclude that $\Xi$ commutes with the characteristic class.

The proof of uniqueness of differential cohomology up to unique natural transformation is done in two steps: 
We first show that if there exists a natural transformation, then it is uniquely determined. 

\begin{thm}[Uniqueness of differential cohomology I] \index{Theorem!uniqueness of differential cohomology I}%
Let $\Ht^*(\cdt;\Z)$ be a differential cohomology theory in the sense of Definition~\ref{def:diff_cohom}.
Suppose there exists a natural transformation $\Xi:\Ht^*(\cdt;\Z) \to \widehat H^*(\cdt ;\Z)$ that commutes with curvature and topological trivializations.
\index{+Xi@$\Xi$, natural transformation from a differential cohomology theory to differential characters}
Then $\Xi$ is uniquely determined by these requirements.
\end{thm}

\begin{proof}
Let $X$ be a smooth space.
By assumption, we have a homomorphism $\Xi:\Ht^*(X;\Z) \to \widehat H^*(X;\Z)$ satisfying
\begin{align}
\Xi \circ \iotat 
&= 
\iota \,, \label{eq:Xi1} \\
\curv \circ \Xi
&= \curvt \,. \label{eq:Xi3}
\end{align}
Moreover, naturality means that for any smooth map $f:Y \to X$ and any $x \in \Ht^*(X)$, we have:
\begin{equation}
f^*(\Xi(x))
=
\Xi(f^*x) \,. \label{eq:Xi0} \\  
\end{equation}
Now let $x \in \Ht^k(X;\Z)$, and let $z \in Z_{k-1}(X;\Z)$.
We show that $\Xi(x)(z)$ is uniquely determined:
Choose homomorphisms $\zeta:Z_{k-1}(X;\Z) \to \ZZ_{k-1}(X)$ and $a:Z_{k-1}(X;\Z) \to C_k(X;\Z)$ as in Lemma~\ref{lem:Liftazeta} such that $[z-\partial a(z)]_{\partial S_k} = [\zeta(z)]_{\partial S_k}$.
By Remark~\ref{rem:thin_invariance}, differential characters are thin invariant.
Thus we have
\begin{align*}
\Xi(x)(z)
&=
\Xi(x)([\zeta(z)]_{\partial S_k}) \cdot \Xi(x)(\partial a(z)) \notag \\
&\ist{\eqref{eq:def_diff_charact_2}}\,
\Xi(x)([\zeta(z)]_{\partial S_k}) \cdot \exp \Big( 2\pi i \int_{a(z)} \curv(\Xi(x)) \Big) \,. \notag 
\end{align*}
Write $\zeta(z) = [M\xrightarrow{g}X]$.
For dimensional reasons, we have $\ct(g^*x)=0$. 
Thus by \eqref{eq:3x3diagram_Q}, we find $\varrho \in \Omega^{k-1}(X)$ such that $g^*x = \iotat([\varrho])$.
This yields:
\begin{align}
\Xi(x)(z)
&\ist{\eqref{eq:Xi3}}
g^*\Xi(x)([M]) \cdot \exp \Big( 2\pi i \int_{a(z)} \curvt(x) \Big) \notag \\
&\ist{\eqref{eq:Xi0}}\,
\Xi(g^*x)([M]) \cdot \exp \Big( 2\pi i \int_{a(z)} \curvt(x) \Big) \notag \\
&=
\Xi(\iotat([\varrho]))([M]) \cdot \exp \Big( 2\pi i \int_{a(z)} \curvt(x) \Big) \notag \\
&\ist{\eqref{eq:Xi1}} \,
\iota(\varrho)([M]) \cdot \exp \Big( 2\pi i \int_{a(z)} \curvt(x) \Big)  \notag \\
&\ist{\eqref{eq:def_iota}}\,
\exp \Big[ 2\pi i \Big( \int_M \varrho + \int_{a(z)} \curvt(x) \Big) \Big] \,.   \label{eq:Xi_unique}
\end{align}
We have derived an explicit formula for $\Xi$ and, in particular, proved its uniqueness.
\end{proof}

Now we take \eqref{eq:Xi_unique} to define a natural transformation $\Xi:\Ht^*(\cdt;\Z) \to \widehat H^*(\cdt ;\Z)$:

\begin{definition} 
Let $\Ht^*(\cdt;\Z)$ be a differential cohomology theory.
We define a natural transformation $\Xi: \Ht^*(\cdt;\Z) \to \widehat H^*(\cdt;\Z)$ as follows:
Let $X$ be a smooth space and $x \in \Ht^k(X;\Z)$.
Choose homomorphisms $\zeta:Z_{k-1}(X;\Z) \to \ZZ_{k-1}(X)$ and $a:Z_{k-1}(X;\Z) \to C_k(X;\Z)$ as in Lemma~\ref{lem:Liftazeta} such that $[z-\partial a(z)]_{\partial S_k} = [\zeta(z)]_{\partial S_k}$ for all $z\in Z_{k-1}(X;\Z)$.
Write $\zeta(z) = [M\xrightarrow{g}X]$.
For dimensional reasons, we have $\ct(g^*x)=0$. 
Thus by \eqref{eq:3x3diagram_Q}, we find $\varrho \in \Omega^{k-1}(M)$ such that $g^*x = \iotat([\varrho])$.
Now we set:
\begin{equation}\label{eq:def_XiQ}
\Xi(x)(z)
:=
\exp \Big[ 2\pi i \Big( \int_M \varrho + \int_{a(z)} \curvt(x) \Big) \Big] \,. 
\end{equation}
\end{definition}

The following Lemma shows that $\Xi$ is well defined.
The fact that $\zeta$ and $a$ are homomorphisms will be convenient for the proof of Theorem~\ref{thm:diff_cohom_unique} but for formula \eqref{eq:def_XiQ} this is not relevant.

\begin{lemma}\label{lem:XiQ'}
Let $X$ be a smooth space and $x \in \Ht^k(X;\Z)$.
Let $z \in Z_{k-1}(X;\Z)$.
Let $\zeta'(z) = [M'\xrightarrow{g'}X] \in \ZZ_{k-1}(X)$ and $a'(z) \in C_k(X;\Z)$ be any choice of geometric cycle and singular chain such that $[z-\partial a'(z)]_{\partial S_k}=[\zeta'(z)]_{\partial S_k}$.
Let $\varrho' \in \Omega^{k-1}(M')$ be any differential form such that ${g'}^*x= \iotat([\varrho'])$.
Then we have
\begin{equation}\label{eq:XiQ'}
\Xi(x)(z)
=
\exp \Big[ 2\pi i \Big( \int_{M'} \varrho' + \int_{a'(z)} \curvt(x) \Big) \Big] \,. 
\end{equation}
\end{lemma}

\begin{proof}
Since $\zeta(z)$ and $\zeta'(z)$ both represent the homology class of $z$, we find a geometric boundary $\partial\beta(z) \in \BB_{k-1}(X)$ such that $\partial\beta(z) = \zeta'(z) - \zeta(z)$.
Since 
$$
[\partial a(z) - \partial a'(z)] _{\partial S_k} = [\partial\beta(z)]_{\partial S_k} = \partial [\beta(z)]_{S_k},
$$ 
we find a smooth singular cycle $w(z) \in Z_k(X;\Z)$ such that 
\begin{equation}
[a(z) - a'(z) - w(z)]_{S_k} = [\beta(z)]_{S_k}. 
\label{eq:aawbeta}
\end{equation}
Write $\beta(z) = [N\xrightarrow{G}X]$, where $N$ is a $k$-dimensional oriented compact $p$-stratifold with boundary $\partial N = M' \sqcup \overline{M}$ and $g=G|_{M}$, $g'=G|_{M'}$. 
Since $H^k(N;\Z)=\{0\}$, we have $\ct(G^*x)=0$.
By \eqref{eq:3x3diagram_Q}, we find a differential form $\eta \in \Omega^{k-1}(N)$ such that $G^*x=\iotat([\eta])$.
Then we have
$$
\iotat([\varrho']) - \iotat([\varrho]) 
=
g^*x - {g'}^*x
=
G|_{\partial N}^*x
=
(G^*x)|_{\partial N}
=
\iotat([\eta])|_{\partial N} \,.
$$   
In particular, we have $\eta|_{\partial N} - (\varrho'-\varrho) \in \Omega^{k-1}_0(\partial N)$.
Inserting this into \eqref{eq:def_XiQ} and \eqref{eq:XiQ'}, we find: 
\begin{align*}
\Xi(x)(z) \cdot \exp \Big[2\pi i \Big(&\int_{M'} \varrho' + \int_{a'(z)} \curvt(x) \Big)\Big]^{-1} \\
&=
\exp \Big[2\pi i \Big(\int_{M'} \varrho' - \int_M \varrho + \int_{a'(z)-a(z)} \curvt(x) \Big)\Big] \\
&=
\exp \Big[2\pi i \Big( \int_{\partial N} \eta + \int_{a'(z)-a(z)} \curvt(x) \Big)\Big] \\
&=
\exp \Big[2\pi i \Big( \int_N d\eta + \underbrace{\int_{-w(z)} \curvt(x)}_{\in \Z} + \int_{-[\beta(z)]_{S_k}} \curvt(x) \Big)\Big] \\
&=
\exp \Big[2\pi i \Big( \int_N G^*\curvt(x) + \int_{-[\beta(z)]_{S_k}} \curvt(x) \Big)\Big] \\
&=
\exp \Big(2\pi i \int_{\underbrace{G_*[N]_{S_k} -[\beta(z)]_{S_k}}_{=0}} \curvt(x) \Big) \\
&=
1 \,.
\end{align*}
This yields \eqref{eq:XiQ'}.
\end{proof}

Now we complete the proof of uniqueness of differential cohomology up to unique natural transformation by establishing existence of a natural transformation.

\begin{thm}[Uniqueness of differential cohomology II]\label{thm:diff_cohom_unique} \index{Theorem!uniqueness of differential cohomology II}%
The map $\Xi:\Ht^*(\cdt;\Z) \to \widehat H^*(\cdt;\Z)$ defined in \eqref{eq:def_XiQ} is a natural transformation and commutes with curvature, topological trivializations and inclusion of flat classes.
More explicitly, we have 
\begin{align}
\Xi \circ \iotat 
&= 
\iota \,, \label{eq:XiQ11} \\
\Xi \circ \jt 
&=
j \,, \label{eq:XiQ12} \\
\curv \circ \Xi
&= \curvt \,. \label{eq:XiQ13} 
\intertext{
For any smooth map $f:Y \to X$, and any $x \in \Ht^k(X)$, we have:}
f^*\Xi(x) 
&= \Xi(f^*x) \label{eq:XiQ10} \,.
\end{align}
\end{thm}

\begin{remark}
It follows from \cite[Lemma~1.1]{SS08}, that $\Xi$ also satisfies
\begin{equation*}
c \circ \Xi
=
\ct \,. 
\end{equation*}
\end{remark}

\begin{proof}[Proof of Theorem~\ref{thm:diff_cohom_unique}]
a)
We first show that $\Xi$ takes values in $\widehat H^*(\cdt;\Z)$.
Let $X$ be a fixed smooth space and $x \in \Ht^k(X;\Z)$.
By construction, the maps $\zeta:Z_{k-1}(X;\Z) \to \ZZ_{k-1}(X)$ and $a:Z_{k-1}(X;\Z) \to C_k(X;\Z)$ are group homomorphisms, first defined on basis elements and then extended linearly.
Similarly, the choice of differential forms $\varrho \in \Omega^{k-1}(M)$ for $\zeta(z) = [M\xrightarrow{g}X]$ is made on a basis of $Z_{k-1}(X;\Z)$.
Extending linearly, the map $z \mapsto \exp \Big[ 2\pi i \Big( \int_M \varrho + \int_{a(z)} \curvt(x) \Big) \Big]$ defines a group homomorphism $\Xi(x): Z_{k-1}(X;\Z) \to \Ul$. 

It remains to show that $\Xi(x)$ satisfies condition~\eqref{eq:def_diff_charact_2} for the homomorphism $z \mapsto \Xi(x)(z)$ to be a differential character.
The argument is almost the same as in the proof of Lemma~\ref{lem:XiQ'}. 
Let $z = \partial c$ for some $c \in C_k(X;\Z)$.
By Lemma~\ref{lem:Liftazeta}, the homomorphism $\zeta:C_k(X;\Z) \to \CC_k(X;\Z)$ satisfies $\zeta(\partial c) = \partial \zeta(c)$.
We write $\zeta(c) = [N\xrightarrow{f}X]$, where $M = \partial N$ and $g = f|_M$. 

If $N$ were an oriented smooth manifold with boundary, we would have $H^k(N;\Z) =\{0\}$. 
By the following argument (suggested to us by M.~Kreck), we may also choose the stratifold $N$ such that its top dimensional cohomology vanishes:
Replacing the top dimensional strata of $N$ and $M$ by the connected sum of their components if necessary, we may assume the top dimensional strata of $N$ and $M$ to be connected.
This yields $H^k(N,M;\Z) \cong H^{k-1}(M;\Z) \cong \Z$, the first isomorphism being the boundary map.   
Now the long exact sequence of the pair $(N,M)$ yields $H^k(N;\Z) = \{0\}$.

Since $H^k(N;\Z) = \{0\}$, we have $f^*x = \iotat([\eta])$ for some $\eta \in \Omega^{k-1}(N)$.
Since $\iotat$ is natural with respect to smooth maps, we have: 
$$
\iotat([\varrho])
=
g^*x 
= 
(f^*x)|_{\partial N} 
=
\iotat([\eta])|_{\partial N}
=
\iotat([\eta|_{\partial N}]) \,.
$$
In particular, $\varrho - \eta|_M \in \Omega^{k-1}_0(M)$. 
This yields:
$$
\exp \Big( 2\pi i \int_M \varrho \Big)
=
\exp \Big( 2\pi i \int_{\partial N} \eta \Big) %\\
=
\exp \Big( 2\pi i \int_N d\eta \Big) %\\
=
\exp \Big( 2\pi i \int_N \curvt(f^*x) \Big) \,. 
$$
Inserting this into \eqref{eq:def_XiQ}, we obtain:
\begin{align*}
\Xi(x)(\partial c)
&=
\exp \Big[ 2\pi i \Big( \int_{\partial N} \varrho + \int_{a(\partial c)} \curvt(x) \Big) \Big] \\
&=
\exp \Big[ 2\pi i \Big( \int_N \curvt(f^*x) + \int_{a(\partial c)} \curvt(x) \Big) \Big] \\
&=
\exp \Big[ 2\pi i \Big( \int_{[\zeta(c)]_{S_k}} \curvt(x) + \int_{a(\partial c)} \curvt(x) \Big) \Big] \\
&\ist{\eqref{eq:cadel}}\,
\exp \Big[ 2\pi i \Big( \int_c \curvt(x) + \underbrace{\int_{\partial a(c+y(c))} \curvt(x)}_{=0} \Big) \Big] \\
&= 
\exp \Big( 2\pi i \int_c \curvt(x) \Big) \,. 
\end{align*}
Thus $\Xi(x)$ is a differential character in $\widehat H^k(X;\Z)$ with $\curv(\Xi(x)) = \curvt(x)$. 

b)
For any smooth space $X$, the map $\Xi:\Ht^*(X;\Z) \to \widehat H^*(X;\Z)$ defined by \eqref{eq:def_XiQ} is additive.
Thus $\Xi:\Ht^*(X;\Z) \to \widehat H^*(X;\Z)$ is a degree $0$ homomorphism of graded groups. 

c)
We show that $\Xi$ is natural with respect to smooth maps.
Let $f:Y \to X$ be a smooth map. 
Let $x\in \Ht^k(X)$ and $z \in Z_{k-1}(Y;\Z)$.
We need to show that $\Xi(f^*x)(z) = f^*(\Xi(x))(z)$.
Choose $\zeta(z) \in \ZZ_{k-1}(Y)$ and $a(z) \in C_k(Y;\Z)$ such that \mbox{$[z-\partial a(z)]_{\partial S_k}=[\zeta(z)]_{\partial S_k}$}.

Write $\zeta(z) = [M\xrightarrow{g}Y]$.
Setting $\zeta(f_*z) := f_*\zeta(z) = [M\xrightarrow{f\circ g}X]$ and $a(f_*z) := f_*a(z)$, we obtain
$$
[f_*z- \partial a(f_*z)]_{\partial S_k} 
= 
f_*[z-\partial a(z)]_{\partial S_k}
=
f_*[\zeta(z)]_{\partial S_k}
=
[f_*\zeta(z)]_{\partial S_k}.
$$
Now choose $\varrho \in \Omega^{k-1}(M)$ such that $(f \circ g)^*x = g^*(f^*x) = \iotat([\varrho])$.
By Remark~\ref{rem:def_f*h} and Lemma~\ref{lem:XiQ'}, we find: 
\begin{align*}
f^*(\Xi(x))(z) 
&:= 
\Xi(x)(f_*z) \\
&=
\exp \Big[ 2\pi i \Big( \int_M \varrho + \int_{a(f_*z)} \curvt(x) \Big) \Big] \\
&=
\exp \Big[ 2\pi i \Big( \int_M \varrho + \int_{a(z)} \curvt(f^*x) \Big) \Big] \\
&=
\Xi(f^*x)(z) \,.
\end{align*}

d)
We show that $\Xi$ commutes with inclusions of flat classes.
Let $u \in H^{k-1}(X;\Ul)$ and $z \in Z_{k-1}(X;\Z)$.
We choose $\zeta(z) = [M\xrightarrow{g}X]$ and $a(z) \in C_k(X;\Z)$ as above. 
Note that $\ct(g^*\jt(u))=0$ for dimensional reasons.
Thus $g^*u$ is the reduction mod $\Z$ of a class in $H^{k-1}(M;\R)$.
Let $\varrho \in \Omega^{k-1}(M)$ such that $g^*(\jt(u)) = \iotat([\varrho])$.
Since the upper left quadrant of diagram~\eqref{eq:3x3diagram_Q} commutes, the reduction mod $\Z$ of $[\varrho]_\mathrm{dR} \in H^{k-1}(M;\R)$ coincides with $g^*u$.
Moreover, the diagram \eqref{eq:3x3diagram_Q} yields $\curvt(\jt(u)) =0$.
Thus we have:
\allowdisplaybreaks{
\begin{align*}
\Xi(\jt(u))(z)
&=
\exp \Big[ 2\pi i \Big( \int_M \varrho + \int_{a(z)} \underbrace{{\curvt(\jt(u))}}_{=0} \Big) \Big] \\
&=
\exp \Big( 2\pi i \int_M [\varrho]_\mathrm{dR} \Big) \\
&=
\la g^*u,[M] \ra \\
&=
\la u,g_*[M] \ra \\
&\ist{\eqref{eq:zazeta}}
\la u,[z] \ra \\
&\ist{\eqref{eq:def_j}}
j(u)(z) \,.
\end{align*}
}

e)
We show that $\Xi$ commutes with topological trivializations.
Let $\varrho \in \Omega^{k-1}(X)$.
Then we have:
\begin{align*}
\Xi(\iotat([\varrho]))(z) 
&=
\exp \Big[ 2\pi i \Big( \int_M g^*\varrho + \int_{a(z)} \curvt(\iotat([\varrho])) \Big) \Big] \\
&\ist{\eqref{eq:3x3diagram_Q}}\,
\exp \Big[ 2\pi i \Big( \int_{g_*[M]_{\partial S_k}} \varrho + \int_{a(z)} d\varrho \Big) \Big] \\
&=
\exp \Big[ 2\pi i \Big( \int_{[\zeta(z)]_{\partial S_k}} \varrho + \int_{\partial a(z)} \varrho \Big) \Big] \\
&\ist{\eqref{eq:zazeta}}\,
\exp \Big( 2\pi i \int_z \varrho \Big) \\
&\ist{\eqref{eq:def_iota}}\,
\iota(\varrho)(z) \,. \qedhere 
\end{align*}
\end{proof}

\chapter{The ring structure}\label{subsec:ring}
In this section we discuss the ring structure on differential cohomology.
Existence of a natural ring structure on $\widehat H^*(X;\Z)$ compatible with curvature, characteristic class and topological trivializations was established in \cite[Thm.~1.11]{CS83} by an explicit formula using barycentric subdivision of singular chains and the chain homotopy from the subdivision to the identity.
Simple formulas for the product are obtained for differential characters represented by differential forms with singularities as in \cite{C73} or by de Rham-Federer currents as in \cite[Sec.~3]{HLZ03}. 
\index{de Rham-Federer currents}%

An axiomatic definition of a ring structure on differential cohomology was established in \cite{SS08}, together with a proof that the ring structure is uniquely determined by these axioms (see \cite[Thm.~1.2]{SS08}).
We use an axiomatic definition of the ring structure similar to the one in \cite{SS08}.
The sign convention for topological trivializations differs from the one in \cite[p.~51]{SS08} but coincides with the one in \cite[Def.~1.2]{BS09}.
We give a corresponding axiomatic definition of an external or cross product and prove that this product is uniqely determined by the axioms.
Uniquess of the external product has also been discussed in \cite[Ch.~6]{L06}.
Our proof has the advantage of giving an explicit geometric formula for the product.

\begin{definition}\label{def:int_prod_axioms} \index{Definition!internal product} \index{Definition!ring structure}%
An \emph{internal product} of differential characters yields for any smooth space $X$ and any $(k,l) \in \Z \times \Z$ a  map 
\begin{equation}
*: \widehat H^k(X;\Z) \times \widehat H^l(X;\Z) \to \widehat H^{k+l}(X;\Z) \,, \quad (h,f) \mapsto h*f \,,
\end{equation}
such that the following holds:
\index{+Star@$*$, internal product of differential characters}%
\index{internal product!of differential characters}%
\index{differential characters!internal product}%
\index{differential characters!ring structure}%
\begin{itemize}
\item[1.]{\em{Ring structure.}} 
The product $*$ is associative and $\Z$-bilinear, i.e. $(\widehat H^*(X;\Z),+,*)$ is a ring.
\item[2.]{\em{Graded commutativity.}}
The product $*$ is graded commutative, i.e.~for $h \in \widehat H^k(X;\Z)$ and $f \in \widehat H^l(X;\Z)$, we have $f*h = (-1)^{kl} h*f$.
\item[3.]{\em{Naturality.}}
For any smooth map $g: Y \to X$ and $h,f \in \widehat H^*(X;\Z)$, we have $g^*(h*f) = g^*h * g^*f$.
\item[4.]{\em{Compatibility with curvature.}}
The curvature $\curv:\widehat H^*(X;\Z) \to \Omega^*_0(X)$ is a ring homomorphism, i.e.~for $h,f \in \widehat H^*(X;\Z)$, we have $\curv(h*f)= \curv(h) \wedge \curv(f)$.
\index{compatibility!of internal product with curvature}%
\item[5.]{\em{Compatibility with characteristic class.}}
The characteristic class $c:\widehat H^*(X;\Z) \to H^*(X;\Z)$ is a ring homomorphism, i.e.~for $h,f \in \widehat H^*(X;\Z)$, we have $c(h*f)= c(h) \cup c(f)$.
\index{compatibility!of internal product with characteristic class}%
\item[6.]{\em{Compatibility with topological trivialization.}}
For $\varrho \in \Omega^*(X)$ and $f \in \widehat H^l(X;\Z)$, we have $\iota(\varrho)*f = \iota(\varrho \wedge \curv(f))$. 
\index{compatibility!of internal product with topological trivialization}%
\end{itemize}  
\end{definition}

An internal product on differential cohomology induces an \emph{external product} or \emph{differential cohomology cross product} 
$$
\times: \widehat H^k(X;\Z) \times \widehat H^{k'}(X';\Z) \to \widehat H^{k+l}(X;\Z), \quad h \times h' := \pr_1^*h * \pr_2^*h'\,.
$$
Here $\pr_1, \pr_2$ denotes the projection on the first and second factor of $X \times X'$, respectively.   

We may also define an \emph{external product} or \emph{differential cohomology cross product} axiomatically: 

\begin{definition}\label{def:ext_prod_axioms} \index{Definition!external product} \index{Definition!cross product}%
An \emph{external product} of differential characters yields for any smooth spaces $X$ and $X'$ and any $(k,k') \in \Z \times \Z$ a map 
\begin{equation}
\times: \widehat H^k(X;\Z) \times \widehat H^{k'}(X';\Z) \to \widehat H^{k+k'}(X \times X';\Z) \,, \quad (h,h') \mapsto h \times h'\,, 
\end{equation}
such that the following holds:
\index{+Times@$\times$, external or cross product of differential characters}
\index{differential characters!external product}%
\index{differential characters!cross product}%
\index{cross product!of differential characters}%
\index{cross product!differential cohomology $\sim$}
\index{differential cohomology cross product}%
\begin{itemize}
\item[1.]{\em{Associativity, bilinearity.}} 
The product $\times$ is associative and $\Z$-bilinear.
\item[2.]{\em{Graded commutativity.}}
The product $\times$ is graded commutative, i.e.~for $h \in \widehat H^k(X;\Z)$ and $h' \in \widehat H^{k'}(X';\Z)$, we have:
\begin{equation}\label{eq:ext_comm}
h' \times h 
= 
(-1)^{kk'} h\times h' \,. 
\end{equation}
\item[3.]{\em{Naturality.}}
For any smooth maps $g: Y \to X$ and $g': Y' \to X'$ and for $h \in \widehat H^*(X;\Z)$ and $h' \in \widehat H^*(X';\Z)$, we have:
\begin{equation}\label{eq:ext_nat}
(g \times g')^*(h\times h') 
= 
g^*h \times {g'}^*h' \,.
\end{equation}
\item[4.]{\em{Compatibility with curvature.}}
The curvature \mbox{$\curv:\widehat H^*(X;\Z) \to \Omega^*_0(X)$} commutes with external products, i.e.~for $h \in \widehat H^*(X;\Z)$ and \mbox{$h' \in \widehat H^*(X';\Z)$}, we have:
\begin{equation}\label{eq:ext_curv}
\curv(h\times h')
= 
\curv(h) \times \curv(h') \,. 
\end{equation}
\index{compatibility!of external product with curvature}%
\item[5.]{\em{Compatibility with characteristic class.}}
The characteristic class \mbox{$c:\widehat H^*(X;\Z) \to H^*(X;\Z)$} commutes with external products, i.e.~for $h \in \widehat H^*(X;\Z)$ and $h' \in \widehat H^*(X';\Z)$, we have:
\begin{equation}\label{eq:ext_c}
c(h \times h')
= 
c(h) \times c(h') \,. 
\end{equation}
\index{compatibility!of external product with characteristic class}%
\item[6.]{\em{Compatibility with topological trivialization.}}
For $\varrho \in \Omega^*(X)$ and \mbox{$h' \in \widehat H^k(X';\Z)$}, we have:
\begin{equation}\label{eq:ext_top_triv} 
\iota(\varrho) \times h' 
= 
\iota(\varrho \times \curv(h')) \,.
\end{equation}
\index{compatibility!of external product with topological trivialization}%
\end{itemize}  
\end{definition}
An external product yields an internal product by setting $h*f := \Delta_X^* (h \times f)$ for any $h,f \in \widehat H^*(X;\Z)$.
Here $\Delta_X:X \to X \times X$ denotes the diagonal map.
\index{+Delta@$\Delta$, diagonal map}

Internal and external products are equivalent in the sense that any one determines the other.
Starting with an internal product $*$, the induced external product recovers the original internal product:
for any $h,f \in \widehat H^*(X;\Z)$, we have
\begin{equation}\label{eq:ext_int}
\Delta_X^*(h \times f) 
= 
\Delta_X^* (\pr_1^*h * \pr_2^*f) 
= 
(\pr_1 \circ \Delta_X)^*h * (\pr_2 \circ \Delta_X)^*f 
= 
h*f.
\end{equation}
Conversely, starting with an external product $\times$, the induced internal product recovers the original external product: 
for $h \in \widehat H^*(X;\Z)$ and $h' \in \widehat H^*(X';\Z)$, we have
\begin{align*}\label{eq:int_ext}
\pr_1^*h * \pr_2^*h'
&=
\Delta_{X \times X'} (\pr_1^*h \times \pr_2^*h') \\
&=
\Delta_{X \times X'}^*(\pr_1 \times \pr_2)^*(h \times h') \\
&=
(\underbrace{(\pr_1 \times \pr_2) \circ \Delta_{X \times X'}}_{=\id_{X \times X'}})^*(h \times h') \\
&=
h\times h'.
\end{align*}
Internal products are useful, since they provide differential cohomology with a ring structure.
On the other hand, external products are sometimes more useful for explicit calculations, as we shall see below.

In the following, we show that the ring structure on differential coho\-mology is uniquely determined by the axioms in Definition~\ref{def:int_prod_axioms}.
By the discussion above, this is equivalent to the fact that the induced external product is uniquely determined by the axioms in Definition~\ref{def:ext_prod_axioms}. 
To prove the latter, we start with the following special case:

\begin{lemma}[Evaluation on cartesian products]\label{lem:x_M_M'} \index{Lemma!evaluation on cartesian products}%
Let $M$ and $M'$ be closed oriented $p$-stratifolds.
\index{stratifold}%
Suppose $\dim(M \times M') = k+k'-1$.
Let $\times$ be an external product in the sense of Definition~\ref{def:ext_prod_axioms}.
Then for differential characters $h \in \widehat H^k(M;\Z)$ and $h' \in \widehat H^{k'}(M';\Z)$, we have:
\begin{equation}\label{eq:x_M_M'}
(h \times h')([M \times M'])
= \begin{cases}
  h([M])^{\langle c(h'),[M']\rangle}  & \quad \mbox{for $(\dim(M),\dim(M')) = (k-1,k')$} \\ 
  h'([M'])^{(-1)^{k} \langle c(h),[M]\rangle}  & \quad \mbox{for $(\dim(M),\dim(M')) = (k,k'-1)$} \\ 
  1 & \quad \mbox{otherwise}
  \end{cases} 
\end{equation} 
\end{lemma}

\begin{proof}
If $(\dim(M),\dim(M')) \notin \{(k-1,k'),(k,k'-1)\}$, then either $\dim(M) < k-1$ or $\dim(M') < k'-1$. 
In these cases we have $\widehat H^k(M;\Z) = \{0\}$ or $\widehat H^{k'}(M';\Z) = \{0\}$. 
Since $\times$ is bilinear, we have $h \times h'=0$ in these cases.

Suppose $(\dim(M),\dim(M')) = (k-1,k')$. 
Then $h$ is topologically trivial for dimensional reasons.
Thus we may choose $\varrho \in \Omega^{k-1}(M)$ such that $\iota(\varrho) = h$.
By Definition~\ref{def:ext_prod_axioms}, we then have:
\begin{align*}
(h \times h')([M \times M'])
&=
(\iota(\varrho) \times h')([M \times M']) \\
&\stackrel{\eqref{eq:ext_top_triv}}{=}
(\iota(\varrho \times \curv(h')))([M \times M']) \\
&=
\exp \Big( 2\pi i \int_{M \times M'} \varrho \times \curv(h') \Big) \\
&=
\exp \Big( 2\pi i \Big(\int_M \varrho \cdot \int_{M'} \curv(h') \Big) \\
&=
\exp \Big( 2\pi i \int_M \varrho \Big)^{\la c(h'),[M']\ra} \\
&=
h([M])^{\langle c(h'),[M']\rangle} \,.
\end{align*}
Similarly, for $(\dim(M),\dim(M')) = (k,k'-1)$, we find $\varrho' \in \Omega^{k'-1}(M')$ such that $h' = \iota(\varrho')$.
This yields 
\begin{align*}
h \times h' 
&= 
h \times \iota(\varrho') \\
&= 
(-1)^{kk'} \iota(\varrho') \times h \\
&= 
(-1)^{kk'} \iota(\varrho' \times \curv(h)) \\ 
&= 
(-1)^{kk'} \iota((-1)^{k(k'-1)} \curv(h) \times \varrho) \\
&= 
(-1)^k \iota(\curv(h) \times \varrho')
\end{align*} 
and hence 
$$
(h \times h')([M \times M'])
=
\exp\Big( 2\pi i \int_{M \times M'} (-1)^k\curv(h) \times \varrho'\Big)
=
h'([M'])^{(-1)^k \la c(h),[M] \ra} \,. \qedhere
$$   
\end{proof}

Now we use this special case to show that the differential cohomology cross product is uniquely determined by the axioms in Definition~\ref{def:ext_prod_axioms}.
The main idea of the proof is to use a splitting of the K\"unneth sequence \index{K\"unneth sequence}%
\begin{equation*}
0
\to \big[ H_*(X;\Z) \otimes H_*(X';\Z) \big]_n
\xrightarrow{\times} H_n(X \times X';\Z) 
\to \Tor( H_*(X;\Z) , H_*(X';\Z))_{n-1}
\to 0
\end{equation*}
on the level of cycles.
We use the well-known Alexander-Whitney and Eilenberg-Zilber maps $\xymatrix{C_*(X \times X';\Z) \ar@<2pt>[rr]^{AW} && C_*(X;\Z) \otimes C_*(X';\Z) \ar@<2pt>[ll]^{EZ}}$.
\index{+AW@$AW$, Alexander-Whitney map}
\index{Alexander-Whitney map}%
\index{+EZ@$EZ$, Eilenberg-Zilber map}
\index{Eilenberg-Zilber map}%
These are chain homotopy inverses of each other with $AW \circ EZ = \id_{C_*(X;\Z) \otimes C_*(X';\Z)}$ and $EZ \circ AW$ chain homotopic to the identity on $C_*(X \times X';\Z)$, see \cite[p.~167]{McC01}.
Let $i:Z_*(X;\Z) \to C_*(X;\Z)$ be the inclusion and let $s:C_*(X;\Z) \to Z_*(X;\Z)$ be a splitting as in Remark~\ref{rem:extend}.
Similarly, we have the inclusion $i'$ and a splitting $s'$ on $X'$. 
Set $S := (s \otimes s')\circ AW$ and $K:= EZ \circ (i \otimes i')$.
\index{+S@$S$, splitting of K\"unneth sequence}%
\index{K\"unneth sequence!splitting}%
\index{+K@$K$, splitting of K\"unneth sequence}%
Denoting by $Z(C_*(X;\Z) \otimes C_*(X';\Z))$ the cycles of the tensor product complex, we obtain the following splitting of the K\"unneth sequence on the level of cycles:
\begin{equation*}
\xymatrix{
0 \ar[r] & Z_*(X;\Z) \otimes Z_*(X';\Z) \ar@<4pt>[rr]^{i \otimes i'} \ar@<4pt>[drr]^(0.6)K && Z(C_*(X;\Z) \otimes C_*(X';\Z)) \ar@<4pt>[d]^{EZ} \ar[r] \ar@{-->}[ll]^{s \otimes s'} & \ldots \\
&&& Z_*(X \times X';\Z) \ar@{-->}[u]^{AW} \ar@{-->}[ull]^(0.4)S
} 
\end{equation*}
In particular, we have $S \circ K = (s \otimes s') \circ AW \circ EZ \circ (i \otimes i') = \id_{Z_*(X;\Z) \otimes Z_*(X';\Z)}$.  

Using this splitting, we proceed by carefully choosing the homomorphism $\zeta^{X \times X'}:Z_*(X \times X';\Z) \to \ZZ_*(X \times X')$:
We first construct the homomorphism \mbox{$\zeta^X:Z_*(X;\Z) \to \ZZ_*(X)$} as in Lemma~\ref{lem:Liftazeta}, and similarly for $X'$.
We compose $\zeta^X\otimes\zeta^{X'}$ with the cross product 
$$
\times: \ZZ_*(X) \otimes \ZZ_*(X') \to \ZZ_*(X \times X'), \quad [M\xrightarrow{g}X] \otimes [M'\xrightarrow{g'}X'] \mapsto [M \times M'\xrightarrow{g \times g'}X \times X'],
$$
and obtain a homomorphism
$Z_*(X;\Z)\otimes Z_*(X';\Z)\to \ZZ_*(X \times X')$.
Using the splitting we extend this map to a homomorphism $\zeta^{X \times X'}:Z_*(X \times X';\Z)\to \ZZ_*(X \times X')$.
We thus obtain the commutative diagram:
\begin{equation}\label{eq:diag_zeta_x}
\xymatrix{
\ZZ_*(X) \otimes \ZZ_*(X') \ar[rr]^\times && \ZZ_*(X \times X') \\
Z_*(X;\Z) \otimes Z_*(X';\Z) \ar@<4pt>[rr]^K \ar[u]^{\zeta^X \otimes \zeta^{X'}} && Z_*(X \times X';\Z) \ar@{-->}[ll]^S \ar[u]_{\zeta^{X \times X'}}   
}
\end{equation}

Now let $h \in \widehat H^k(X;\Z)$ and $h' \in \widehat H^{k'}(X';\Z)$ and $z \in Z_{k+k'-1}(X \times X';\Z)$.
We write $z= K \circ S(z) + (z- K \circ S(z))$.
The K\"unneth sequence implies that $(z-K \circ S(z))$ represents a torsion class.
Hence $(h \times h')(z-K\circ S(z))$ may be computed as in Remark~\ref{rem:torsion_cycles}.
We compute $(h \times h')(K \circ S(z))$ as described in Remark~\ref{rem:diff_charact_stratifolds} using geometric chains:

The splitting $S$ decomposes a cycle $z \in Z_*(X \times X';\Z)$ into a sum of tensor products of cycles with degrees adding up to $k+k'-1$.
We write 
\begin{equation}
S(K\circ S(z)) = S(z) = \sum_{i+j=k+k'-1}\sum_m y_i^m\otimes {y'}_j^m ,
\label{eq:Zerl}
\end{equation}
where $y_i^m\in Z_i(X;\Z)$ and ${y'}_j^m\in Z_j(X';\Z)$.
For the geometric cycles, we obtain correspondingly 
\begin{align*}
\zeta^{X \times X'}(K \circ S(z)) 
&= \sum_{i+j=k+k'-1} \sum_m \zeta^X(y_i^m) \times \zeta^{X'}({y'}_j^m).
\end{align*}
Now we are able to compute $(h \times h')(K\circ S(z))$.

\begin{thm}[Uniqueness of cross product]\label{thm:ext_prod_BB} \index{Theorem!uniqueness of cross product}%
The differential cohomology cross product is uniquely determined by the axioms in Definition~\ref{def:ext_prod_axioms}. 

Explicitly, for $h \in \widehat H^k(X;\Z)$ and $h' \in \widehat H^{k'}(X';\Z)$, the evaluation of $h \times h'$ on a cycle $z \in Z_{k+k'-1}(X \times X';\Z)$ can be computed as follows:
Decompose $S(z)$ as in \eqref{eq:Zerl}.
Choose $N \in \N$ and $x \in C_{k+k'}(X;\Z)$ as in Remark~\ref{rem:torsion_cycles} such that $N \cdot (z-K\circ S(z)) = \partial x$.
Then we have:
\begin{align}
(h \times h')(z)
&= \prod_{m}\left[h([\zeta^X(y_{k-1}^m)]_{\partial S_k})^{\langle c(h'),{y'}_{k'}^m\rangle} \cdot h'([\zeta^{X'}({y'}_{k'-1}^m)]_{\partial S_{k'}})^{(-1)^{k}\langle c(h),y_k^m\rangle}\right] \notag \\
& \quad \quad \cdot \exp 2\pi i \cdot \Big[ \int_{a(K \circ S(z))} \curv(h \times h') + \frac{1}{N} \Big(\int_x \curv(h\times h') - \langle c(h\times h'),x\rangle \Big)\Big]. \label{eq:ext_prod_BB}
\end{align}
\end{thm}

\begin{proof}
As above, we write $z = K\circ S(z) + (z-K\circ S(z))$.
We evaluate $(h\times h')$ on the two summands separately.

a)
By Remark~\ref{rem:torsion_cycles}, we have:
\begin{align*}
(h \times h')(1-K \circ S(z))
&\ist{\eqref{eq:torsion_cycles}}\,\,
\exp \frac{2\pi i}{N} \Big( \int_x \curv(h \times h') - \la c(h\times h'), x\ra \Big)\notag \\
&\ist{\eqref{eq:ext_curv},\eqref{eq:ext_c}}\,\,\,\,\,\,\,
\exp \frac{2\pi i}{N} \Big( \int_x \curv(h) \times \curv(h') - \la c(h)\times c(h'), x\ra \Big) \,. 
\end{align*}
which yields the last contribution to \eqref{eq:ext_prod_BB}.
This shows in particular, that the value of $h \times h'$ on torsion cycles is uniquely determined by compatibility with curvature and characteristic class in Definition~\ref{def:ext_prod_axioms}.

b)
We represent the cycle $K \circ S(z)$ by the geometric cycle $\zeta^{X \times X'}(K \circ S(z))$ and a coboundary $\partial a(K \circ S(z))$ as in Lemma~\ref{lem:Liftazeta}.
We compute $(h \times h')(K \circ S(z))$ as in Remark~\ref{rem:diff_charact_stratifolds}:
\begin{align}
(h \times h')(K \circ S(z))
&\ist{\eqref{eq:evaluation_diff_charact_stratifold_1}}
(h \times h')([\zeta^{X \times X'}(K \circ S(z))]_{\partial S_{k+k'}}) \cdot \exp \Big( 2\pi i \int_{a(K \circ S(z))} \curv(h \times h') \Big) \notag\\
&= 
\prod_{i+j=k+k'-1}\prod_m (h \times h') \big([\zeta^X (y_i^m) \times \zeta^{X'}({y'}_j^m)]_{\partial S_{k+k'}} \big) \notag\\
& \qquad \qquad \cdot \exp \Big( 2\pi i \int_{a(K \circ S(z))} \curv(h \times h') \Big)
\label{eq:hkreuzh1}
\end{align}
Now we write $\zeta^X(y_i^m) = [M_i^m\xrightarrow{g_i^m}X]$ and $\zeta^{X'}({y'}_j^m) = [{M'}_j^m\xrightarrow{{g'}_j^m}X]$.
This yields:
\begin{align}
(h \times h') \big([\zeta^X(y_i^m) \times \zeta^{X'}({y'}_j^m)]_{\partial S_{k+k'}} \big)
&= (g_i^m \times {g'}_j^m)^*(h \times h')([M_i^m \times {M'}_j^m]) \notag\\
&\ist{\eqref{eq:ext_nat}} ((g_i^m)^*h \times ({g'}_j^m)^*h')([M_i^m \times {M'}_j^m]) .
\label{eq:hkreuzh2}
\end{align}
By construction of $\zeta^X$ and $\zeta^{X'}$, we have $\dim(M_i^m) = i$ and $\dim({M'}_j^m)=j$.
Using Lemma~\ref{lem:x_M_M'}, we find:
\allowdisplaybreaks{
\begin{align}
((g_i^m)^*h &\times ({g'}_j^m)^*h')([M_i^m \times {M'}_j^m]) \notag\\
&= \begin{cases} 
   (g_{k-1}^m)^*h([M_{k-1}^m])^{\langle ({g'}_{k'}^m)^* c(h'),[{M'}_{k'}^m]\rangle} & \quad \mbox{for $i=k-1$, $j=k'$} \\
   ({g'}_{k'-1}^m)^*h'([{M'}_{k'-1}^m])^{(-1)^{k}\langle (g_k^m)^* c(h),[M_k^m]\rangle} & \quad \mbox{for $i=k$, $j=k'-1$} \\
    1 & \quad \mbox{otherwise} 
   \end{cases}\notag\\
&= \begin{cases} 
   h([\zeta^X(y_{k-1}^m)]_{\partial S_k})^{\langle c(h'),{y'}_{k'}^m\rangle} & \quad \mbox{for $i=k-1$, $j=k'$} \\
   h'([\zeta^{X'}({y'}_{k'-1}^m)]_{\partial S_{k'}})^{(-1)^{k}\langle c(h),y_k^m\rangle} & \quad \mbox{for $i=k$, $j=k'-1$} \\
    1 & \quad \mbox{otherwise} 
   \end{cases}
\label{eq:hkreuzh3}
\end{align}
}
Inserting \eqref{eq:hkreuzh2} and \eqref{eq:hkreuzh3} into \eqref{eq:hkreuzh1} we obtain:
\begin{align*}
(h \times h')(K \circ S(z)) 
&= 
\prod_m \left[ h([\zeta^X(y_{k-1}^m)]_{\partial S_k})^{\langle c(h'),{y'}_{k'}^m\rangle} \cdot h'([\zeta^{X'}({y'}_{k'-1}^m)]_{\partial S_{k'}})^{(-1)^{k}\langle c(h),y_k^m\rangle}\right] \\
&\qquad \cdot \exp \Big( 2\pi i \int_{a(K \circ S(z))} \curv(h \times h') \Big)
\end{align*}
which yields the remaining terms in \eqref{eq:ext_prod_BB}.
In particular, the evaluation of $h \times h'$ on $K \circ S(z)$ is uniquely determined by the axioms in Definition~\ref{def:ext_prod_axioms} (through Lemma~\ref{lem:x_M_M'}).
\end{proof}

\begin{cor}[Uniqueness of ring structure] \index{Corollary!uniqueness of ring structure}
The ring structure on differential cohomology is uniquely determined by the axioms in Definition~\ref{def:int_prod_axioms}. 
\end{cor}

\begin{remark}
We have shown \emph{uniqueness} of the ring structure.
We could take \eqref{eq:ext_prod_BB} as definition of a differential cohomology cross product to prove \emph{existence} of the cross product and ring structure on differential cohomology.
This would require to verify the axioms in Definition~\ref{def:ext_prod_axioms}.
Since this amounts to no more than tedious computation, we take existence of the ring structure and cross product for granted (see \cite[p.~55f]{CS83}).  
\end{remark}

\begin{example}
Let $h_1,h_2 \in \widehat H^1(X;\Z) \cong C^\infty(X;\Ul)$.
As in Example~\ref{ex:H1}, we denote the corresponding smooth functions by $\bar h_1$, $\bar h_2$.  
Now $h_1*h_2\in \widehat H^2(X;\Z)$.
Hence, given two smooth functions $\bar h_j:X\to \Ul$, we obtain a $\Ul$-bundle with connection over $X$ (up to isomorphism).
We now describe this bundle in classical geometric terms.

Let $i \in \widehat H^1(\Ul;\Z)$ be the differential character that corresponds to the smooth function $\bar i=\id_\Ul:\Ul \to \Ul$.
\index{+Ii@$i$, differential character on $\Ul$}
Then we have $\bar h_j = \id_{\Ul} \circ \bar h_j$ and thus $h_j = \bar h_j^*i$.
We put $\bar h=(\bar h_1,\bar h_2):X \to \Ul\times \Ul =:T^2$.
Let $\Delta:\Ul \to T^2$, $t\mapsto (t,t)$,  the diagonal map.
This yields
\begin{align*}
h_1 * h_2
&\stackrel{\phantom{\eqref{eq:ext_nat}}}{=} 
\Delta^*(h_1\times h_2)\\
&\stackrel{\phantom{\eqref{eq:ext_nat}}}{=} 
\Delta^*(\bar h_1^*i\times\bar h_2^*i)\\
&\stackrel{\eqref{eq:ext_nat}}{=} 
\Delta^*(\bar h_1 \times \bar h_2)^*(i \times i)\\
&\stackrel{\phantom{\eqref{eq:ext_nat}}}{=} 
((\bar h_1 \times \bar h_2)\circ\Delta)^*(i \times i)\\
&\stackrel{\phantom{\eqref{eq:ext_nat}}}{=} 
\bar h^*(i \times i).
\end{align*}
The bundle corresponding to $h_1 * h_2$ is thus given by pull-back along $\bar h$ of a universal bundle with connection $(P,\nabla)$ on $T^2$ which represents $i\times i\in\widehat H^2(T^2;\Z)$.
\index{internal product!of degree-$1$ characters}%

The bundle $(P,\nabla)$ was described in algebraic geometric terms in \cite[Sec.~1]{B81} where it leads to the regulator map in algebraic $K$-theory.
The total space is identified with the Heisenberg manifold $H(\R)/H(\Z)$. 
In \cite[p.~60]{Bu12} it is called the \emph{Poincar\'e bundle}.
\index{Poincar\'e bundle}%
We now determine this bundle.

The curvature $\curv(i)$ is a volume form on $\Ul$ with total volume $1$. 
Thus by \eqref{eq:ext_curv}, the curvature $\curv(i \times i)$ is a volume form on $T^2$ with total volume~$1$.
Since $H^2(T^2;\Z)$ has no torsion, the characteristic class $c(i \times i)$ can be identified with the de Rham class of $\curv(i \times i)$.
This class determines the $\Ul$-bundle $P \to T^2$ topologically.
It remains to determine the connection $\nabla$.

Let $\Theta_1,\Theta_2:\R^2 \to \R$ denote the projection on the first and second factor, respectively.  
Let $p: \R^2 \to \R^2/\Z^2 \cong T^2$, $v = (v_1,v_2) \mapsto (\exp(2\pi i v_1),\exp(2\pi i v_2))$, denote the projection.
Let $\nabla$ be any connection on $P$ with curvature $\frac{i}{2\pi}\curv(\nabla) = \curv(i \times i)$.
Fix a trivialization $T:p^*P \to \R^2 \times \Ul$.
As in Example~\ref{ex:U1Buendel}, we denote by $\vartheta(p^*\nabla,T) \in \Omega^1(\R^2)$ the $1$-form that corresponds to the connection $p^*\nabla$. 
The trivialization can be chosen such that 
$$
\vartheta(p^*\nabla,T) = (\Theta_1/2-w_1)d\Theta_2 - (\Theta_2/2-w_2)d\Theta_1=:A_w
$$ 
for some $w = (w_1,w_2)\in \R^2$.
Two forms $A_w$ and $A_{w'}$ describe the same connection $\nabla$ on $P$ if and only if $w-w'\in\Z^2$.

The parameter $w$, and hence the connection $\nabla$, can be determined by the holonomy along two particular curves in $T^2$.
Consider the curves \mbox{$\gamma_1: [0,1] \to T^2$}, \mbox{$t \mapsto (\exp (2\pi it),1)$}, and $\gamma_2: [0,1] \to T^2$, $t \mapsto (1,\exp (2\pi it))$.
Set $\Gamma_1: [0,1] \to \R^2$, \mbox{$t \mapsto (t,0)$}, and $\Gamma_2:[0,1] \mapsto \R^2$, $t \mapsto (0,t)$, so that $\Gamma_j$ lifts $\gamma_j$.  
Then we have:
\begin{align*}
\hol^\nabla(\gamma_1) 
&= \exp \Big(2\pi i\int_{\Gamma_1} A_w \Big) \\
&=\exp \Big(2\pi i\int_\Gamma (-0/2+w_2) d\Theta_1 \Big) \\
&= \exp \Big(2\pi iw_2 \Big)\\ 
\intertext{and similarly}
\hol^\nabla(\gamma_2) 
&= \exp \Big(-2\pi iw_1 \Big). 
\end{align*}
To determine the connection, we evaluate $i\times i$ on the cycles $\gamma_1$ and $\gamma_2$.
Denote the fundamental cycle $[0,1]\to\Ul$, $t\mapsto \exp(2\pi it)$, of $\Ul$ by $y$.
Then the decomposition \eqref{eq:Zerl} of $\gamma_1$ is given by
$$
\gamma_1 = y \times 1 = K(y\otimes 1).
$$
We apply Theorem~\ref{thm:ext_prod_BB} with $z=\gamma_1$ and observe that we can choose $x=0$ because $\gamma_1=K(S(\gamma_1))$.
Since $[\zeta^\Ul(y)]_{\partial S_2}=[\gamma_1]_{\partial S_2}$ we may choose $a(\gamma_1)=0$.
Now \eqref{eq:ext_prod_BB} says 
$$
i\times i(\gamma_1)
=
i(1)^{-1}
=
1.
$$
Similarly, we get $i\times i(\gamma_2)=1$.
Hence our connection $\nabla$ is given by $A_0=\Theta_1/2d\Theta_2 - \Theta_2/2d\Theta_1$.
\end{example}

\begin{remark}
\emph{K\"unneth sequence.}
The exactness of the K\"unneth sequence for singular cohomology 
\begin{equation*}\label{eq:Kuenneth_sing_cohom}
0
\to 
\big[ H^*(X;\Z) \otimes H^*(X';\Z) \big]_n
\xrightarrow{\times} 
H^n(X \times X';\Z) 
\to 
\Tor( H^*(X;\Z) , H^*(X';\Z))_{n+1}
\to 
0
\end{equation*}
implies that the cohomology cross product is injective.
The K\"unneth sequence is usually constructed in two steps: 
the first one is purely algebraic and relates the homology of tensor products of chain complexes with the tensor product of the homologies;
the second one identifies the singular homology of the cartesian product of spaces with the homology of the tensor product of the singular chain complexes.

The question arises whether there is a K\"unneth sequence for differential cohomology.
As to the above mentioned first step, the differential cohomology groups of a space $X$ can be constructed as the homology groups of a chain complex using a modification of the Hopkins-Singer complex, as described in \cite[p.~271]{BT06}.
\index{K\"unneth sequence}%
\index{Hopkins-Singer complex}%
This way one obtains the homological algebraic K\"unneth sequence for that complex.
The middle term of that sequence is the homology of the tensor product complex. 
The relation of this tensor product homology to the differential cohomology of the cartesian product seems to be unknown.
\end{remark}

The following example illustrates that the differential cohomology cross product is in general not injective:
\begin{example}
Let $X,X'$ be closed manifolds of dimensions $k-1$ and $k'$, respectively.
Let $\varrho \in \Omega^{k-1}(X)$ and $\varrho \in \Omega^{k'}(X')$ be volume forms for some Riemannian metrics on $X$ and $X'$ with total volume $1$.
In particular, $\varrho$ and $\varrho'$ are closed with integral periods, and $\frac{1}{2}\varrho$ does not have integral periods.
Choose a differential character $h' \in \widehat H^{k'}(X';\Z)$ with $\curv(h') = 2\varrho'$.
Set $h:= \iota(\tfrac{1}{2}\cdot \varrho) \neq 0 \in \widehat H^{k}(X)$.
Then we have $h \times h' \stackrel{\eqref{eq:ext_top_triv}}{=} \iota(\tfrac{1}{2} \cdot \varrho \times 2\varrho') = \iota(\varrho \times \varrho')$.
This vanishes since $\int_{X \times X'} \varrho \times \varrho' =1$ and $\varrho \times \varrho'$ thus has integral periods. 
\end{example}

%%%%%%%%%%%%%%%%%%%%%%%%%%%%%%%%%%%%%%%%%%%%%%%%%%%%%%%%%%%%%%%%%%%%%%%%%
\chapter{Fiber integration}
%%%%%%%%%%%%%%%%%%%%%%%%%%%%%%%%%%%%%%%%%%%%%%%%%%%%%%%%%%%%%%%%%%%%%%%%%

In this section we construct the fiber integration map for differential characters.
Fiber integration has been described in some of the various models for differential cohomology.
The construction of Hopkins and Singer in \cite{HS05} is based on their own model and uses embeddings into high-dimensional Euclidean spaces.
\index{Hopkins-Singer complex}%
% Ljungmann discusses the existence and uniqueness of fiber integration for smooth Deligne cohomology in \cite{L06} and gives a geometric construction where the combinatorial complications are taken care of by the calculus of simplicial forms.
% Using this construction, Dupont and Ljungmann prove the up-down formula in \cite{}.
In \cite{DL05} and \cite{L06} Dupont and Ljungmann give a geometric construction of fiber integration for smooth Deligne cohomology where the combinatorial complications are taken care of by the calculus of simplicial forms.
Uniqueness of fiber integration is discussed in \cite[Ch.~6]{L06}.
\index{Deligne cohomology}%
A model for differential characters involving stratifolds is described in \cite{BKS10} where fiber integration is also discussed.
\index{stratifold}%
The fiber integration or Gysin map for de Rham-Federer currents is described in \cite[Sec.~10]{HLZ03}.
\index{de Rham-Federer currents}%
\index{Gysin map}%

We use the original definition of differential characters due to Cheeger and Simons.
Our construction of the fiber integration map works for fiber bundles (with compact oriented fibers) on all smooth spaces in the sense of Section~\ref{sec:smoothspaces}.
The approaches in \cite{BKS10} and \cite{DL05,L06} seem to be limited to fiber bundles over finite dimensional bases.
However, allowing infinite-dimensional manifolds is important.
For example, the transgression map from equivalence classes of gerbes on $X$ to equivalence classes of line bundles with connection on the free loop space $\LL(X)$ is constructed using fiber integration in the trivial bundle $S^1 \times \LL(X) \to \LL(X)$, compare Section~\ref{sec:applications}.

We show that fiber integration (for fiber bundles whose fibers are closed oriented manifolds) is uniquely determined by certain naturality conditions. 
This yields an explicit formula for the fiber integration map which we then use for its definition.
We show that this yields a well-defined fiber integration map that has the required properties.
Finally, we discuss fiber integration in the case where the fiber has a boundary.

Similar approaches to our construction of the fiber integration map have been sketched briefly in \cite[Prop.~2.1]{F02}, in \cite[Sec.~3.6]{CJM04} and in \cite[Thm.~3.135]{Bu12}.

\section{Fiber integration for closed fibers}\label{subsec:fiber_int}

\begin{definition}\label{def:fiber_int_diff_charact_axioms} \index{Definition!fiber integration}%
Let $F \hookrightarrow E \stackrel{\pi}{\twoheadrightarrow} X$ be a fiber bundle over a smooth space $X$ whose fibers are closed (i.e., finite-dimensional, compact and boundaryless) oriented manifolds.
Fiber integration for differential characters associates to each such bundle a group homomorphism 
$\widehat\pi_!: \widehat H^*(E;\Z) \to \widehat H^{*-\dim F}(X;\Z)$ such that the following holds:
\index{+Pihat@$\widehat\pi_\ausruf$, fiber integration for differential characters}\index{fiber integration!for differential characters}%
\begin{itemize}
\item[1.]{\em{Naturality.}}
For any smooth map $g: Y \to X$ the fiber integration map commutes with the maps in the pull-back diagram
\begin{equation*}
\xymatrix{
g^*E \ar^{\pi}[d] \ar^{G}[r] & E \ar^{\pi}[d] \\
Y \ar^{g}[r] & X \,.
}
\end{equation*}
This means that for any $h \in \widehat H^k(E;\Z)$, we have 
\begin{equation}
\widehat\pi_!(G^*h) 
= g^*\widehat\pi_!(h) \,. \label{eq:fiber_int_natural}
\end{equation}
In other words, the following diagram is commutative for all $k$:
\begin{equation}\label{eq:fiber_int_natural_diagram}
\xymatrix{
\widehat H^k(E;\Z) \ar^{G^*}[r] \ar^{\widehat\pi_!}[d] & \widehat H^k(g^*E;\Z) \ar^{\widehat\pi_!}[d] \\
\widehat H^{k-\dim F}(X;\Z) \ar^{g^*}[r] & \widehat H^{k-\dim F}(Y;\Z)  
}
\end{equation}
\index{naturality!of fiber integration}%

\item[2.]{\em{Compatibility with curvature.}}
Let $\fint_F: \Omega^*(E) \to \Omega^{*-\dim F}(X)$ be the usual fiber integration of differential forms, see \cite[Ch.~VII]{GHV-I}.
We require that the fiber integration of differential characters is compatible with the fiber integration of the curvature form, i.e.,
\begin{equation}\label{eq:fiber_int_curvature}
\xymatrix{
\widehat H^k(E;\Z) \ar[d]^{\widehat\pi_!} \ar[r]^{\curv} 
& \Omega_0^k(E) \ar[d]^{\fint} \\
\widehat H^{k-\dim F}(X;\Z)  \ar[r]^{\curv} 
& \Omega_0^{k-\dim F}(X) 
}
\end{equation}
commutes.
\index{compatibility!of fiber integartion with curvature}%

\item[3.]{\em Compatibility with topological trivializations of flat characters.}
We demand that the following diagram commutes:
\begin{equation}\label{eq:fiber_int_iota}
\xymatrix{
\Omega^{k-1}_{cl}(E) \ar[d]^{\fint} \ar[r]^{\iota}
& \widehat H^k(E;\Z) \ar[d]^{\widehat\pi_!} \\
\Omega^{k-1-\dim F}_{cl}(X) \ar[r]^{\iota}
& \widehat H^{k-\dim F}(X;\Z)  \,.
}
\end{equation}
\index{compatibility!of fiber integration with topological trivialization}%
\end{itemize}
\end{definition}
Before we construct fiber integration for differential characters using geometric chains, we first show that it is uniquely determined by the above conditions:

\begin{thm}[Uniqueness of fiber integration]\label{thm:fiber_int_unique} \index{Theorem!uniqueness of fiber integration}%
If fiber integration for differential characters exists, then it is uniquely determined by the conditions of naturality and compatibility in Definition~\ref{def:fiber_int_diff_charact_axioms}.
\end{thm}

\begin{proof}
Let $F \hookrightarrow E \stackrel{\pi}{\twoheadrightarrow} X$ be a fiber bundle with closed oriented fibers over a smooth space $X$.
Let $\widehat\pi_!: \widehat H^k(E;\Z) \to \widehat H^{k-\dim F}(X;\Z)$ be a fiber integration map as in Definition~\ref{def:fiber_int_diff_charact_axioms}.

For $k < \dim F$ the map $\widehat\pi_!$ is uniquely determined, since in this case $\widehat H^{k-\dim F}(X;\Z) = \{0\}$ by \eqref{eq:def_Hk_negativ}.
For $k = \dim F$, the compatibility with curvature implies that $\curv(\widehat\pi_!h) = \fint_F \curv(h) \in \Omega^0_0(X)$.
For degree $0$, the diagram \eqref{eq:3x3diagram} yields the isomorphisms
\begin{equation}\label{eq:0-3x3}
\xymatrix{\widehat H^0(X;\Z) \ar[d]_{c=\id}^\cong \ar[rr]^(0.45){\curv}_(0.45)\cong && \Omega^0_0(X) \ar[d]_\cong \\
H^0(X;\Z) \ar[rr]_(0.45)\cong && \Hom(H_0(X;\Z),\Z) \\
}
\end{equation}
Thus $\widehat\pi_!h$ is uniquely determined by its curvature. 

Now let $k > \dim F$.
Let $h \in \widehat H^k(E;\Z)$ be a differential character on the total $E$ and \mbox{$z \in Z_{k-1-\dim F}(X;\Z)$} a smooth singular cycle in the base $X$.
We show that the value of $\widehat\pi_! h$ on $z$ is uniquely determined by the conditions in Definition~\ref{def:fiber_int_diff_charact_axioms}.

As in Lemma~\ref{lem:Liftazeta} we choose a geometric cycle $\zeta(z) = [M\xrightarrow{g}X] \in \ZZ_{k-1-\dim F}(X)$ and a smooth singular chain $a(z) \in C_{k-\dim F}(X;\Z)$ such that $[z-\partial a(z)]_{\partial S_{k-\dim F}} = [\zeta(z)]_{\partial S_{k-\dim F}}$.
We then have:
\begin{align*}
(\widehat \pi_!h)(z)
&\stackrel{\eqref{eq:evaluation_diff_charact_stratifold_2}}{=} 
(g^* \widehat\pi_! h)[M] \cdot \exp \Big( 2\pi i \int_{a(z)} \curv(\widehat \pi_! h) \Big) \\
&\stackrel{\eqref{eq:fiber_int_curvature}}{=} 
(g^* \widehat\pi_! h)[M] \cdot \exp \Big( 2\pi i \int_{a(z)} \fint_F \curv(h) \Big) \\ 
&\stackrel{\eqref{eq:fiber_int_natural}}{=}
(\widehat\pi_! G^*h)[M] \cdot \exp \Big( 2\pi i \int_{a(z)} \fint_F \curv(h) \Big) \,.
\end{align*}
The differential character $G^*h \in \widehat H^k(g^*E;\Z)$ is topologically trivial and flat for dimensional reasons (note that $\dim(g^*E) = k-1$).
Hence $G^*h = \iota(\chi)$ for some closed differential form $\chi \in \Omega^{k-1}(g^*E)$.
From the commutative diagram~\eqref{eq:fiber_int_iota} we then have
\begin{align}
(\widehat \pi_!h)(z)
&= 
(\widehat\pi_! G^*h)[M] \cdot \exp \Big( 2\pi i \int_{a(z)} \fint_F \curv(h) \Big) \nonumber \\
&= 
(\widehat\pi_! \iota(\chi))[M] \cdot \exp \Big( 2\pi i \int_{a(z)} \fint_F \curv(h) \Big) \notag \\
&\ist{\eqref{eq:fiber_int_iota}}\,\,
\iota(\fint_F\chi)[M] \cdot \exp \Big( 2\pi i \int_{a(z)} \fint_F \curv(h) \Big) \nonumber \\ 
&=  
\exp \Big( 2\pi i \int_M \fint_F \chi \Big) \cdot \exp \Big( 2\pi i \int_{a(z)} \fint_F \curv(h) \Big) \,. \label{eq:fiber_int_unique}
\end{align}
We thus obtained an expression for the value of $\widehat \pi_!h$ on $z$, which is uniquely determined by the conditions of naturality and compatibility. 
\end{proof}

We can rewrite formula \eqref{eq:fiber_int_unique} more elegantly in terms of the pull-back operation $\PB_\bullet$ from Section~\ref{sec:geomchains}:
As above, let $h \in \widehat H^k(E;\Z)$ be a differential character on the total space and $z \in Z_{k-1-\dim F}(X;\Z)$ a smooth singular cycle in the base.
As above we get the geometric cycle $\zeta(z) = [M\xrightarrow{g}X] \in \ZZ_{k-1-\dim F}(X)$ and the smooth singular chain $a(z) \in C_{k-\dim F}(X;\Z)$ such that $[z - \partial a(z)]_{\partial S_{k-\dim F}} = [\zeta(z)]_{\partial S_{k-\dim F}}$.
We then have:
\allowdisplaybreaks{
\begin{align*}
(\widehat\pi_! h)(z)
&\ist{\eqref{eq:fiber_int_unique}}\,\,
\exp \Big( 2\pi i \int_M \fint_F \chi \Big) \cdot \exp \Big( 2\pi i \int_{a(z)} \fint_F \curv(h) \Big) \\
&= 
\exp \Big( 2\pi i \int_{g^*E} \chi \Big) \cdot \exp \Big( 2\pi i \int_{a(z)} \fint_F \curv(h) \Big) \\
&= 
\iota(\chi)([g^*E]) \cdot \exp \Big( 2\pi i \int_{a(z)} \fint_F \curv(h) \Big) \\
&= 
G^*h([g^*E]) \cdot \exp \Big( 2\pi i \int_{a(z)} \fint_F \curv(h) \Big) \\
&= 
h([g^*E\xrightarrow{G}E]_{\partial S_k}) \cdot \exp \Big( 2\pi i \int_{a(z)} \fint_F \curv(h) \Big) \\
&= 
h([\PB_E([M\xrightarrow{g}X])]_{\partial S_k}) \cdot \exp \Big( 2\pi i \int_{a(z)} \fint_F \curv(h) \Big) \\
&= 
h([\PB_E \zeta(z)]_{\partial S_k}) \cdot \exp \Big( 2\pi i \int_{a(z)} \fint_F \curv(h) \Big) \,.
\end{align*}
}
Hence we obtain the following constructive definition for the fiber integration map on differential characters:

\begin{definition}\label{def:fiber_int_diff_charact_construction} \index{Definition!fiber integration map}%
Let $F \hookrightarrow E \twoheadrightarrow X$ be a fiber bundle with closed oriented fibers over a smooth space $X$.
For $k < \dim F$, the fiber integration map $\widehat\pi_!:\widehat H^k(E;\Z) \to \widehat H^{k-\dim F}(X;\Z) = \{0\}$ is trivial.
For $k = \dim F$, the fiber integration map $\widehat\pi_!: \widehat H^{\dim F}(E;\Z) \to \widehat H^0(X;\Z) = H^0(X;\Z)$ is defined as:
$$
\widehat\pi_!h
:= \pi_!c(h) \,.
$$
For $k > \dim F$, the fiber integration map $\widehat\pi_!: \widehat H^k(E;\Z) \to \widehat H^{k-\dim F}(X;\Z)$ is defined as:
\begin{align}
(\widehat\pi_!h)(z) 
:=&\, h([\PB_E \zeta(z)]_{\partial S_k}) \cdot \exp \Big( 2\pi i \int_{a(z)} \fint_F \curv(h)  \Big) \label{eq:def_fiber_int} \\
=&\, (G^*h)([g^*E]) \cdot \exp \Big( 2\pi i \int_{a(z)} \fint_F \curv(h) \Big) \,.\label{eq:def_fiber_int_2}
\end{align}
\end{definition}

Using the transfer map $\lambda$ constructed in Remark~\ref{rem:lambda}, we obtain the following expression for fiber integration:
\begin{lemma}[Fiber integration via transfer map] \index{Lemma!fiber integration via transfer map}%
Let $k > \dim F$.
Let $h \in \widehat H^k(E;\Z)$ and let $\lambda:C_{k-1-\dim F}(X;\Z) \to C_{k-1}(E;\Z)$ as defined in Remark~\ref{rem:lambda}.
Then we have for any $z \in Z_{k-1-\dim F}(X;\Z)$:
\begin{equation}\label{eq:pilambda}
(\widehat\pi_!h)(z)
= h(\lambda(z)) \cdot \exp \Big( 2\pi i \int_{a(z)} \fint_F \curv(h)  \Big).  
\end{equation}
\end{lemma}

\begin{proof}
By \eqref{eq:def_fiber_int} and the construction of $\lambda$, we find:
\begin{align*}
(\widehat\pi_!h)(z)
&:= 
h([\PBE\zeta(z)]_{\partial S_k}) \cdot \exp \Big( 2\pi i \int_{a(z)} \fint_F \curv(h)  \Big) \\
&\ist{\eqref{eq:lambdazzeta}}\,
h([\lambda(z)]_{\partial S_k}) \cdot \exp \Big( 2\pi i \int_{a(z)} \fint_F \curv(h)  \Big) \\
&= 
h(\lambda(z)) \cdot \exp \Big( 2\pi i \int_{a(z)} \fint_F \curv(h)  \Big) . 
\end{align*}
In the last equation we have used thin invariance of differential characters.
\end{proof}

\begin{lemma}
The fiber integration $\widehat \pi_!$ as defined in \eqref{eq:def_fiber_int} is a group homomorphism $\widehat H^k(E;\Z) \to \widehat H^{k-\dim F}(X;\Z)$.
In particular, for $k > \dim F$ the map $z \mapsto \widehat\pi_!h(z)$ is indeed a differential character. 
\end{lemma}

\begin{proof}
For $k < \dim F$, we obtain the trivial map $\widehat H^k(E;\Z) \to \widehat H^{k-\dim F}(X;\Z) = \{0\}$, which is a group homomorphism.  
For $k=\dim F$, the fiber integration map $\widehat\pi_! = \pi_! \circ c: \widehat H^{\dim F}(E;\Z) \to \widehat H^0(X;\Z)$ is the composition of the group homomorphisms $c: \widehat H^{\dim F}(E;\Z) \to H^{\dim F}(E;\Z)$ and $\pi_!:H^{\dim F}(E;\Z) \to H^0(X;\Z)$.

Now let $k > \dim F$.
We show that $\widehat\pi_!h$ is indeed a differential character.
The map $z\mapsto\widehat\pi_!h(z)$ is a group homomorphism $Z_{k-1-\dim F}(X;\Z)\to \Ul$ because all ingredients of the right hand side of \eqref{eq:def_fiber_int} are homomorphisms.

We check that the evaluation of $\widehat\pi_!h$ on a boundary is given by the integral of a differential form.
Let $z = \partial c \in B_{k-1-\dim F}(X;\Z)$ be a smooth singular boundary on $X$.
As in Lemma~\ref{lem:Liftazeta}, we choose geometric chains $\zeta(\partial c)\in \BB_{k-\dim F-1}(X)$ and $\zeta(c) \in \CC_{k-\dim F}(X)$, and smooth singular chains $a(z) \in C_{k-\dim F}(X;\Z)$ and $y(c)\in Z_{k-\dim F}(X;\Z)$ such that $\partial \zeta(c) = \zeta(\partial c)$ and $[c - a(\partial c) - \partial a(c+y(c))]_{S_{k-\dim F}} = [\zeta(c)]_{S_{k-\dim F}}$.
Using \eqref{eq:fint_lambda_closed} for the transfer map $\lambda$, we obtain:
{\allowdisplaybreaks
\begin{align}
(\widehat\pi_!h)(\partial c) 
&=
h(\lambda(\partial c)) \cdot \exp \Big( 2\pi i \int_{a(\partial c)} \fint_F \curv(h)  \Big) \nonumber \\
&\ist{\eqref{eq:dellambdalambdadel}}\,
h(\partial \lambda(c)) \cdot \exp \Big( 2\pi i \int_{a(\partial c)} \fint_F \curv(h)  \Big) \nonumber \\
&=
\exp \Big( 2\pi i \cdot \Big(\int_{\lambda(c)} \curv(h) + \int_{a(\partial c)} \fint_F \curv(h) \Big)\Big) \nonumber \\
&\ist{\eqref{eq:fint_lambda_closed}}\,
\exp \Big( 2\pi i \cdot \Big(\int_{c-a(\partial c)} \fint_F \curv(h) + \int_{a(\partial c)} \fint_F \curv(h) \Big)\Big) \nonumber \\
&=
\exp \Big( 2\pi i \int_c \fint_F \curv(h) \Big) \,.
\label{eq:fiber_int_diff_charact} 
\end{align}
}%
Thus $\widehat\pi_!h$ is indeed a differential character.
From \eqref{eq:def_fiber_int} it is now clear that $h\mapsto \widehat\pi_!h$ is a homomorphism $\widehat H^k(E;\Z) \to \widehat H^{k-\dim F}(X;\Z)$.
\end{proof}

We show that the definition of $\widehat \pi_!h$ in \eqref{eq:def_fiber_int} is independent of the choices.

\begin{lemma}\label{lem:independent}
Let $k > \dim F$.
Let $\zeta':Z_{k-1-\dim F}(X;\Z)\to \ZZ_{k-1-\dim F}(X)$ and \mbox{$a':Z_{k-1-\dim F}(X;\Z)\to C_{k-\dim F}(X;\Z)$} be any maps (not necessarily homomorphisms) such that \eqref{eq:zazeta} in Lemma~\ref{lem:Liftazeta} holds, i.e.,
$$
[z-\partial a'(z)]_{\partial S_{k-\dim F}}=[\zeta'(z)]_{k-1-\dim F}
$$
is true for all $z\in Z_{k-1-\dim F}(X;\Z)$.
Then \eqref{eq:def_fiber_int} remains valid, i.e.,
$$
(\widehat\pi_!h)(z) 
:= h([\PB_E \zeta'(z)]_{\partial S_k}) \cdot \exp \Big( 2\pi i \int_{a'(z)} \fint_F \curv(h)  \Big)
$$
holds for all $z\in Z_{k-1-\dim F}(X;\Z)$ and all $h\in \widehat H^k(E;\Z)$.
\end{lemma}

\begin{proof}
Let $z\in Z_{k-1-\dim F}(X;\Z)$ be a cycle.
Then we find a geometric boundary $\partial\beta(z) \in \BB_{k-1-\dim F}(X)$ such that $\zeta'(z) - \zeta(z) = \partial \beta(z)$.
Since 
$$
[\partial a(z) - \partial a'(z)] _{\partial S_{k-\dim F}} = [\partial\beta(z)]_{\partial S_{k-\dim F}} = \partial [\beta(z)]_{S_{k-\dim F}},
$$ 
we find a smooth singular cycle $w(z) \in Z_{k-\dim F}(X;\Z)$ such that 
\begin{equation}
[a(z) - a'(z) - w(z)]_{S_{k-\dim F}} = [\beta(z)]_{S_{k-\dim F}}. 
\label{eq:aay}
\end{equation}
We then have: 
\begin{align*}
h([\PB_E\zeta'&(z)]_{\partial S_k}) \cdot  h([\PB_E\zeta(z)]_{\partial S_k})^{-1} \\
&= 
h([\PB_E \partial \beta(z)]_{\partial S_k}) \\
&\ist{\eqref{eq:partial_iota_commute}}\,
h([\partial\PB_E\beta(z)]_{\partial S_k}) \\
&\ist{\eqref{eq:geomchainsboundary}}
h(\partial [\PB_E(\beta(z)]_{S_k}) \\
&\ist{\eqref{eq:def_diff_charact_2}}\,
\exp \Big( 2\pi i \int_{[\PB_E\beta(z)]_{_k}} \curv(h) \Big) \\
&\ist{\eqref{axiomD:fiber_int_bound_forms}}\,
\exp \Big( 2\pi i \int_{[\beta(z)]_{S_{k-\dim F}}} \fint_F \curv(h) \Big) \\
&\ist{\eqref{eq:aay}}\,
\exp \Big( 2\pi i \int_{a(z)-a'(z)-w(z)} \fint_F \curv(h) \Big) \\
&=
\exp \Big( 2\pi i \int_{a(z)-a'(z)} \fint_F \curv(h) \Big)\cdot\exp \Big( 2\pi i \underbrace{\int_{w(z)} \fint_F \curv(h)}_{\in\Z} \Big)^{-1} \\
&=
\exp \Big( 2\pi i \int_{a(z)} \fint_F \curv(h) \Big) \cdot \exp \Big( 2\pi i \int_{a'(z)} \fint_F \curv(h) \Big)^{-1} \,.
\end{align*}
This proves the lemma.
\end{proof}

\begin{thm}[Existence of fiber integration]\label{prop:Axiomegelten} \index{Theorem!existence of fiber integration}%
Fiber integration $\widehat \pi_!$ as defined in \eqref{eq:def_fiber_int} satisfies the axioms in Definition~\ref{def:fiber_int_diff_charact_axioms}.
\end{thm}

\begin{proof}
For $k < \dim F$, the trivial map $\widehat\pi_!:\widehat H^k(X;\Z) \to \widehat H^{k-\dim F}(X;\Z) = \{0\}$ obviously satisfies the axioms in Definition~\ref{def:fiber_int_diff_charact_axioms}. 
For $k = \dim F$, the fiber integration $\widehat \pi_!=\pi_! \circ c$ is natural since $\pi_!:H^k(E;\Z) \to H^{k-\dim F}(X;\Z)$ is natural with respect to bundle maps and \mbox{$c:\widehat H^k(E;\Z) \to H^k(E;\Z)$} is natural with respect to any smooth maps.
Compatibility with curvature follows from the commutative diagram \eqref{eq:0-3x3}.
To show compatibility with topological trivializations, let $h = \iota(\varrho)$ for some $\varrho \in \Omega^{\dim F-1}(E)$.
Then we have $c(\iota(\varrho)) =0$. 
For dimensional reasons, we have $\fint_F \varrho =0$.
Thus $\widehat \pi_!\iota(\varrho) = \pi_!c(\iota(\varrho)) = 0 = \iota(\fint_F\varrho)$.

Now let $k> \dim F$.
Equation \eqref{eq:fiber_int_diff_charact} yields for the curvature of $\widehat\pi_!h$:
\begin{equation}
\curv(\widehat\pi_!h) = \fint_F \curv(h) \,.
\label{eq:fiber_int_curv}
\end{equation}
This is compatibility with curvature \eqref{eq:fiber_int_curvature}.

Now let $h = \iota(\eta)$ for some $ \eta \in \Omega^{k-1}(E)$.
Let $z \in Z_{k-1-\dim F}(X;\Z)$.
Using Stokes's theorem we find:
\begin{align*}
(\widehat\pi_! h)(z)
&= 
h(\lambda(z)) \cdot \exp \Big( 2 \pi i \int_{a(z)} \fint_F \curv(h) \Big) \\
&= 
\exp \Big( 2 \pi i \int_{\lambda(z)} \eta \Big) \cdot \exp \Big( 2 \pi i \int_{a(z)} \fint_F \curv(h) \Big) \\
&\ist{\eqref{eq:fint_lambda}}\,
\exp \Big( 2 \pi i \int_{[\zeta(z)]_{\partial S_{k-\dim F}}} \fint_F \eta \Big) \cdot \exp \Big( 2 \pi i \int_{a(z)} \fint_F d\eta \Big) \\
&\ist{\eqref{eq:zazeta}}\,
\exp \Big( 2 \pi i \cdot \Big( \int_{z - \partial a(z)} \fint_F \eta  + \int_{a(z)} \fint_F d\eta \Big) \Big) \\ 
&= 
\exp \Big( 2 \pi i \int_z \fint_F \eta \Big) \,.
\end{align*}
Hence $\widehat\pi_! h = \iota(\fint_F \eta)$, as claimed in \eqref{eq:fiber_int_iota}.

It remains to prove naturality.
Let $g:Y\to X$ be a smooth map.
We have the pull-back diagram
\begin{equation*}
\xymatrix{
g^*E \ar^{\pi}[d] \ar^{G}[r] & E \ar^{\pi}[d] \\
Y \ar^{g}[r] & X \,.
}
\end{equation*}
Let $z\in Z_{k-1-\dim F}(Y;\Z)$.
As in Lemma~\ref{lem:Liftazeta} we choose $\zeta(z)\in \ZZ_{k-1-\dim F}(Y)$ and $a(z)\in C_{k-\dim F}(Y;\Z)$ such that
$[z-\partial a(z)]_{\partial S_{k-\dim F}}=[\zeta(z)]_{\partial S_{k-\dim F}}$.
Hence
$$
[g_*\zeta(z)]_{\partial S_{k-\dim F}}
= g_*[\zeta(z)]_{\partial S_{k-\dim F}}
= g_*[z-\partial a(z)]_{\partial S_{k-\dim F}}
= [g_*z-\partial g_*a(z)]_{\partial S_{k-\dim F}}.
$$
Now let $h\in \widehat H^k(E;\Z)$.
By Lemma~\ref{lem:independent}, we may choose $\zeta(g_*z) = g_*\zeta(z)$ and $a(g_*z)=g_*a(z)$ to compute $g^*(\widehat\pi_!h)(z)$.
This yields:
{\allowdisplaybreaks
\begin{align*}
\widehat\pi_!G^*h(z)
&=
G^*h([\PB_{g^*E}\zeta(z)]_{\partial S_k}) \cdot \exp\Big(2\pi i \int_{a(z)}\fint_F \curv(G^*h)\Big) \\
&=
h(G_*[\PB_{g^*E}\zeta(z)]_{\partial S_k}) \cdot \exp\Big(2\pi i \int_{a(z)}\fint_F G^*\curv(h)\Big) \\
&=
h([G_*\PB_{g^*E}\zeta(z)]_{\partial S_k}) \cdot \exp\Big(2\pi i \int_{a(z)}g^*\fint_F \curv(h)\Big) \\
&\ist{\eqref{eq:PBnatuerlich}}
h([\PBE (g_*\zeta(z))]_{\partial S_k}) \cdot \exp\Big(2\pi i \int_{g_*a(z)}\fint_F \curv(h)\Big) \\
&=
h([\PBE \zeta(g_*z)]_{\partial S_k}) \cdot \exp\Big(2\pi i \int_{a(g_*z)}\fint_F \curv(h)\Big) \\
&=
\widehat\pi_!h(g_*z)\\
&=
g^*(\widehat\pi_!h)(z) .
\end{align*}
}%
For the third equality we use compatibility of fiber integration and pull-back of differential forms, see \cite[Ch.~VII, Prop.~VIII]{GHV-I}.
This proves \eqref{eq:fiber_int_natural}.
\end{proof}

\begin{cor}[Existence and uniqueness of fiber integration] \index{Corollary!existence and uniquess of fiber integration}
There is a unique fiber integration of differential characters satisfying the axioms in Definition~\ref{def:fiber_int_diff_charact_axioms}.
\end{cor}

\begin{remark}\label{rem:shows_more}
The proof of Theorem~\ref{prop:Axiomegelten} shows more than compatibility with topological trivializations of flat characters.
Namely, \eqref{eq:fiber_int_iota} commutes for all $\eta\in\Omega^{k-1}(E)$, not necessarily closed.
In other words, we have shown compatibility with topological trivializations of characters, not necessarily flat.
\end{remark}

\begin{prop}[Compatibility of fiber integration with characteristic class] \index{Proposition!compatibility of fiber integration with characteristic class}%
Fiber integration of differential characters is compatible with the characteristic class, i.e., the diagram 
\begin{equation}\label{eq:cpi_commute}
\xymatrix{
\widehat H^k(E;\Z) \ar[d]^{\widehat\pi_!} \ar[r]^c &
H^k(E;\Z) \ar[d]^{\pi_!} \\
\widehat H^{k-\dim F}(X;\Z) \ar[r]^c & H^{k-\dim F}(X;\Z)
}
\end{equation}
commutes.
\index{compatibility!of fiber integration with characteristic class}%
\end{prop}

\begin{proof}
For $k < \dim F$, there is nothing to show.
For $k = \dim F$, this follows from the commutative diagram \eqref{eq:0-3x3}.

Thus let $k > \dim F$.
We compute the characteristic class $c(\widehat\pi_! h)$.
Let $\tilde h \in \Hom(Z_{k-1}(E;\Z),\R)$ be a real lift of the differential character $h \in \widehat H^k(E;\Z)$ and denote by 
$$
\mu^{\tilde h}: c \mapsto \int_c \curv(h) - \tilde h(\partial c)
$$
the corresponding cocycle representing the characteristic class $c(h)$.

Let $z \in Z_{k-1-\dim F}(X;\Z)$ be a smooth singular cycle in the base $X$.
As in Lemma~\ref{lem:Liftazeta}, we get the geometric cycle $\zeta(z) \in \ZZ_{k-1-\dim F}(X)$ and the smooth singular chain $a(z) \in C_{k-\dim F}(X;\Z)$ such that $[z - \partial a(z)]_{\partial S_{k-\dim F}} = [\zeta(z)]_{\partial S_{k-\dim F}}$.
By definition, the $(k-1)$-chain $\lambda(z)$ represents the fundamental class of the pull-back $\PBE\zeta(z)$, i.e., $[\lambda(z)]_{\partial S_k} = [\PBE\zeta(z)]_{\partial S_k}$, where $\lambda:Z_{k-1-\dim F}(X;\Z) \to Z_{k-1}(E;\Z)$ is the transfer map constructed in Remark~\ref{rem:lambda}.
We obtain a real lift $\widetilde{\widehat\pi_!h}$ of the differential character $\widehat\pi_!h$ by setting
\begin{equation}\label{eq:real_lift}
\widetilde{\widehat\pi_!h}(z)
:= \tilde h(\lambda(z)) + \int_{a(z)} \fint_F \curv(h) \,. 
\end{equation}
Hence the characteristic class of the differential character $\widehat\pi_!h$ is represented by the cocycle
{\allowdisplaybreaks
\begin{align}
c
&\mapsto 
\int_c \curv(\widehat\pi_!h) - \widetilde{\widehat\pi_!h}(\partial c) \notag \\
&= 
\int_c \fint_F \curv(h) - \tilde h(\lambda(\partial c)) - \int_{a(\partial c)} \fint_F \curv(h) \notag \\
&\ist{\eqref{eq:dellambdalambdadel}}
\int_{c-a(\partial c)} \fint_F \curv(h) - \tilde h(\partial \lambda(c)) \notag \\
&\ist{\eqref{eq:fint_lambda_closed}}
\int_{\lambda(c)} \curv(h) - \tilde h(\partial\lambda(c)) \notag \\
&=
(\mu^{\tilde h}\circ \lambda)(c) \,. \label{eq:cocycle_pi_c}
\end{align}
}%
By Remark~\ref{rem:fiber_int_cohomology} this cocycle represents the cohomology class $\pi_!(c(h))$, hence
\begin{equation*}
c(\widehat\pi_!(h))
= \pi_!(c(h)) \,. 
\qedhere
\end{equation*}
\end{proof}

\begin{prop}
Fiber integration of differential characters is compatible with the inclusion of cohomology classes with coefficients in $\Ul$, i.e.,~the diagram 
$$
\xymatrix{
H^{k-1}(E;\Ul) \ar[rr]^j \ar[d]^{\pi_!} && \widehat H^k(E;\Z) \ar[d]^{\widehat\pi_!} \\
H^{k-1-\dim F}(X;\Ul) \ar[rr]^j && \widehat H^{k-\dim F}(X;\Z)  
}
$$
commutes.
\end{prop}

\begin{proof}
For $k < \dim F$ there is nothing to show since both fiber integration maps are trivial for dimensional reasons.
Let $k = \dim F$ and $u \in H^{\dim F-1}(E;\Ul)$.
Diagram~\eqref{eq:3x3diagram} shows that $c(j(u))$ is a torsion class. 
Thus $\widehat\pi_!j(u)=\pi_!c(j(u))=0$ since $\widehat H^0(X;\Z)=H^0(X;\Z)$ is torsion free.
On the other hand, we have \mbox{$\pi_!u \in H^{-1}(X,\Ul) = \{0\}$} and hence $j(\pi_!u)=0$.

Now let $k> \dim F$ and $u \in H^{k-1}(E;\Ul)$. 
Diagram \eqref{eq:3x3diagram} shows \mbox{$\curv(j(u)) =0$}. 
As explained in Remark~\ref{rem:fiber_int_cohomology}, fiber integration $\pi_!$ for singular cohomo\-logy is induced by pre-composition of cocycles with the transfer map \mbox{$\lambda:C_{k-1-\dim F}(X;\Z) \to C_{k-1}(E;\Z)$} constructed in Remark~\ref{rem:lambda}.  
Thus for any $z \in Z_{k-1-\dim F}(X;\Z)$ we have:
\begin{align*}
\widehat\pi_! j(u)(z) 
&\stackrel{\eqref{eq:pilambda}}{=}
j(u)(\lambda(z)) \cdot \exp \Big( 2 \pi i \int_{a(z)} \fint_F \underbrace{\curv(j(u))}_{=0} \Big) \\
&\stackrel{\eqref{eq:def_j}}{=} 
\langle u,[\lambda(z)] \rangle  \\
&\stackrel{\eqref{eq:fiber_int_singular_cohomology}}{=} 
\langle \pi_!u,[z] \rangle \\
&\stackrel{\eqref{eq:def_j}}{=} 
(j(\pi_!u))(z) . \qedhere
\end{align*}
\end{proof}

\begin{prop}[Orientation reversal]\label{prop:orient_reversal} \index{Proposition!orientation reversal}%
Let $E \xrightarrow{\pi} X$ be a fiber bundle with closed oriented fibers over a smooth space $X$.
Let $\overline{\pi}:\overline{E}\to X$ denote the bundle with fiber orientation reversed and $\widehat{\overline{\pi}}_!$ the corresponding fiber integration.
For every $h \in \widehat H^k(E;\Z)$ we have 
\[
\widehat{\overline{\pi}}_!h = -\widehat\pi_!h .
\]
\end{prop}

\begin{proof}
There is nothing to show in case $k < \dim F$. 
For $k = \dim F$, we have \mbox{$\widehat{\overline{\pi}}_! = \overline{\pi}_! \circ c = -\pi_! \circ c = - \widehat\pi_!$}.

Now let $k > \dim F$.
Let $h \in \widehat H^k(E;\Z)$ and $z \in Z_{k-1-\dim F}(X;\Z)$.
Choose $\zeta(z) \in \ZZ_{k-1-\dim F}(X)$ and $a(z) \in C_{k-\dim F}(X;\Z)$ as in Definition~\ref{def:fiber_int_diff_charact_construction}. 
We have $\PB_{\overline{E}}(\zeta(z)) = \overline{\PBE(\zeta(z))}$ and $\fint_{\overline{F}} = -\fint_F$.
This yields
\begin{align*}
(\widehat{\overline{\pi}}_!h)(z)
&= h([\PB_{\overline{E}} \zeta(z)]_{\partial S_k}) \cdot \exp \Big( 2\pi i \int_{a(z)} \fint_{\overline{F}} \curv(h)  \Big) \\
&= h([\overline{\PBE \zeta(z)}]_{\partial S_k}) \cdot \exp \Big( -2\pi i \int_{a(z)} \fint_F \curv(h)  \Big) \\
&= h(-[\PBE \zeta(z)]_{\partial S_k}) \cdot \exp \Big( -2\pi i \int_{a(z)} \fint_F \curv(h)  \Big) \\
&= (\widehat\pi_!h(z))^{-1} \\
&= (-\widehat\pi_!h)(z) \,. \qedhere
\end{align*}
\end{proof}

\begin{example}
We consider the case $k=1$ and $\dim F=0$.
Then $h\in \widehat H^1(E;\Z)=C^\infty(E,\Ul)$ is a smooth $\Ul$-valued function on $E$ and $\pi:E\to X$ is a finite covering.
We orient the fibers such that each point is positively oriented.
It is easy to see that the function $\widehat \pi_!h \in \widehat H^1(X;\Z)=C^\infty(X,\Ul)$ is given by 
$$
\widehat \pi_!h (x) = \prod_{e\in\pi^{-1}(x)}h(e) .
$$
\index{fiber integration!for differential characters!$0$-dimensional fiber}%
\end{example}

\begin{example}
Again consider a finite covering $\pi:E\to X$, i.e., $\dim F=0$, but now $k=2$.
Let $P\to E$ be a $\Ul$-bundle with connection whose isomorphism class corresponds to a differential character $h\in\widehat H^2(X;\Z)$.
Here it is convenient to take for $P$ the Hermitian line bundle rather than the $\Ul$-principal bundle.
Then $\widehat \pi_!h$ is given by the bundle whose fibers over $x\in X$ is
$$
(\widehat\pi_!P)_x = \bigotimes_{e\in\pi^{-1}(x)}P_e .
$$
This bundle inherits a natural tensor product connection from $P$.
\index{fiber integration!for differential characters!$0$-dimensional fiber}%
\end{example}

\begin{example}
Now let $\pi:E\to X$ be a circle bundle with oriented fibers, hence $\dim F=1$.
The fiber integration map $\widehat \pi_!: \widehat H^2(E;\Z) \to \widehat H^1(X;\Z)$ can be described as follows:
Let $P\to E$ be a $\Ul$-bundle with connection.
For any $x\in X$ the holonomy of $P$ along the oriented fiber $E_x$ yields an element in $\Ul$.
In this way, we obtain a smooth function $X\to\Ul$.
\index{fiber integration!for differential characters!$1$-dimensional fiber}%
\end{example}

We show that fiber integration is functorial with respect to composition of fiber bundle projections, compare~\cite[p.~12]{BKS10}.

\begin{prop}[Functoriality of fiber integration] \index{Proposition!functoriality of fiber integration}%
Let $\kappa: N \to E$ and $\pi:E \to X$ be fiber bundles with compact oriented fibers $L$ and $F$, respectively.
Let $\pi \circ \kappa:N \to X$ be the composite fiber bundle with the composite orientation.
\index{orientation!composite $\sim$}%
Then we have
\begin{equation}\label{eq:fiber_int_funct}
\widehat{(\pi \circ \kappa)}_! = \widehat{\pi}_! \circ \widehat{\kappa}_! \,.  
\end{equation}
\index{functoriality!of fiber integration}%
\index{fiber integration!for differential characters!functoriality}%
\end{prop}

\begin{proof}
Denote the fibers of $\pi \circ \kappa$ by $Q$.
The bundle projection $\kappa$ restricts to a fiber bundle $\kappa|_{Q}:Q \to F$ with fibers $L$.
 
For $k < \dim F + \dim L$, we have $\widehat{(\pi \circ \kappa)}_! \equiv 0$. 
Now let $h \in \widehat H^k(N;\Z)$. 
Then we have $\widehat\kappa_!h \in H^{k-\dim L}(E;\Z)$ and $k-\dim L < \dim F$.
Thus $\widehat\pi_!(\widehat\kappa_!h)=0$.

For $k = \dim F + \dim L$, we have:
$$
\widehat{(\pi\circ\kappa)}_!
=
(\pi \circ \kappa)_! \circ c 
=
\pi_! \circ \kappa_! \circ c
\stackrel{\eqref{eq:cpi_commute}}{=}
\pi \circ c \circ \widehat\kappa_!
=
\widehat\pi_! \circ \widehat\kappa_!.
$$
In the second equality we have used the functoriality of fiber integration for singular cohomology, see \cite[p.~484]{BH53}. 

Now let $k > \dim F + \dim L$. 
Let $h \in \widehat H^k(N;\Z)$ and $z \in Z_{k-\dim(F)-\dim(L)-1}(X;\Z)$.
Choose $\zeta(z) = [M\xrightarrow{g}X] \in \ZZ_{k-\dim(F)-\dim(L)-1}(X)$ and $a(z) \in C_{k-\dim(F)-\dim(L)}(X;\Z)$ as in Lemma~\ref{lem:Liftazeta}.
Then we have the pull-back bundles
\begin{equation*}
\xymatrix{
G^*N \ar[r]^{\mathbf G} \ar[d] & N \ar[d]^\kappa \\
g^*E \ar[r]^G \ar[d] & E \ar[d]^\pi \\
M \ar[r]^g & X
} 
\end{equation*}
which define the geometric cycles $\PB_{\pi}(\zeta(z)) = [g^*E\xrightarrow{G}E]$ and $\PB_{\kappa}(\PB_{\pi}(\zeta(z))) = \PB_{\pi \circ \kappa}(\zeta(z)) = [G^*N\xrightarrow{\mathbf G}N]$.
We pull back $h$ to the stratifold $G^*N$, where it is topologically trivial for dimensional reasons.
Thus we find a differential form $\chi \in \Omega^{k-1}(G^*N)$ such that $\mathbf G^*h = \iota(\chi)$.
By the compatibility conditions~\eqref{eq:fiber_int_natural} and \eqref{eq:fiber_int_iota}, we have $\widehat{\kappa}_! (\mathbf G^*h) = G^*(\widehat{\kappa}_!h) = \iota(\fint_L \chi)$.  
In particular, $(\mathbf G^*h)([G^*N]) = \iota(\chi)([G^*N]) = \iota(\fint_L \chi)([g^*E]) = (\widehat\kappa_!(\mathbf G^*h))([g^*E])$.

This yields:
\begin{align*}
(\widehat{\pi \circ \kappa}_! h)(z)
&\ist{\eqref{eq:def_fiber_int_2}}
(\mathbf G^*h)([G^*N]) \cdot \exp \Big( 2 \pi i \int_{a(z)} \fint_{Q} \curv(h) \Big) \\
&= 
(G^*(\widehat\kappa_!h))([g^*E]) \cdot \exp \Big( 2 \pi i \int_{a(z)} \fint_F \curv(\widehat{\kappa}_! h) \Big) \\
&\ist{\eqref{eq:def_fiber_int_2}}
\big(\widehat\pi_!(\widehat\kappa_!h)\big)(z) . \qedhere
\end{align*}
\end{proof}

\section{Fiber integration for fibers with boundary} \label{subsec:fiber_int_bound}
Let 
$(F,\partial F) \hookrightarrow (E,\partial E) 
\xrightarrow{(\pi^E,\pi^{\partial E})} X$ 
be a fiber bundle bundle whose fibers are compact oriented manifolds with boundary.
For any differential form $\omega \in \Omega^*(E)$ on the total space $E$ we have the fiberwise Stokes theorem \cite[p.~311]{GHV-I}:
\begin{equation}\label{eq:stokes_fiber}
\fint_F d \omega = d \fint_F \omega + (-1)^{\deg\omega + \dim \partial F} \fint_{\partial F} \omega \,.
\end{equation}
\index{fiber integration!for differential forms!for fibers with boundary}%
In particular, if $\omega\in \Omega^k(E)$ is a closed form, then $\fint_{\partial F}\omega \in \Omega^{k-\im \partial F}(X)$ is exact.
Thus fiber integration of differential forms in the bundle $\pi^{\partial E}:\partial E \to X$ induces the trivial map on de Rham cohomology.
The same holds true for fiber integration on singular cohomology.

Denote by $\widehat\pi^{\partial E}_!: \widehat H^k(\partial E;\Z) \to \widehat H^{k-\dim\partial F}(X;\Z)$\index{+PidelE@$\widehat\pi^{\partial E}_\ausruf$, fiber integration for fibers with boundary} the fiber integration map for the bundle 
$\partial F \hookrightarrow \partial E \xrightarrow{\pi^{\partial E}} X$ as constructed in the previous section.
In the following, we do not distinguish in notation between a differential character $h \in \widehat H^k(E;\Z)$ and its pull-back to $\partial E$.
Applying the fiber integration map $\widehat\pi^{\partial E}_!$ to $h \in \widehat H^k(E;\Z)$ yields the following (compare also \cite[p.~363]{HS05} and \cite[p.~305]{F02}):

\begin{prop}[Fiber integration for fibers that bound]\label{prop:fiber_int_rand} \index{Proposition!fiber integration for fibers that bound}%
Let $(F,\partial F) \hookrightarrow (E,\partial E) 
\xrightarrow{(\pi^E,\pi^{\partial E})} X$ 
be a fiber bundle with compact oriented fibers with boundary over $X$ and let $h \in \widehat H^k(E;\Z)$ be a differential character.
\index{fiber integration!for differential characters!fibers that bound}%
Then $\widehat\pi^{\partial E}_!h \in \widehat H^{k-\dim\partial F}(X;\Z)$ is topologically trivial. 
A topological trivialization is given by:
\begin{equation}\label{eq:fiber_int_bound2}
\widehat\pi^{\partial E}_!h
=
\iota\Big( (-1)^{k-\dim F}\fint_F \curv(h) \Big) \,.
\end{equation}
In particular, for $k = \dim \partial F$, we have $\widehat\pi^{\partial E}_!h=0 \in \widehat H^0(X;\Z)$.
\end{prop}

\begin{proof}
As explained in Remark~\ref{rem:lambda}, we construct transfer maps $\lambda^E$ and $\lambda^{\partial E}$ for the bundles $\pi^E:E \to X$ and $\pi^{\partial E}:\partial E \to X$, respectively.
By \eqref{axiomD:fiber_int_bound}, we have 
$$
\PB_{\partial E}(\zeta(z)) 
= \begin{cases}
  \partial(\PB_E\zeta(z)) & \; \mbox{for $z \in Z_n(X;\Z)$, $n$ even,}   \\
  \overline{\partial(\PB_E\zeta(z))} & \; \mbox{for $z \in Z_n(X;\Z)$, $n$ odd.}   
  \end{cases}
$$ 
Thus we can arrange the choices in the construction of the transfer maps $\lambda^E$ and $\lambda^{\partial E}$ in such a way that we have: 
\begin{equation}\label{eq:lambdadelE}
\lambda^{\partial E} = (-1)^n \cdot \partial \circ \lambda^E:Z_n(X;\Z) \to B_{n+\dim F}(E;\Z) \,.
\end{equation}
Now we prove the claim:

For $k < \dim \partial F$, there is nothing to show.
Let $k= \dim \partial F$ and $h \in H^{\dim\partial F}(E;\Z)$.
Let $\tilde h$ be a real lift of $h$ and $\mu^{\tilde h} \in C^{\dim F}(E;\Z)$ the corresponding cocycle representing $c(h)$. 
Since $Z_0(X;\Z) = C_0(X;\Z)$, we may use \eqref{eq:lambdadelE} and \eqref{eq:fiber_int_singular_cohomology} to conclude:
$$
\widehat\pi^{\partial E}h
=
\pi^{\partial E} c(h)
\stackrel{\eqref{eq:fiber_int_singular_cohomology}}{=}
[\mu^{\tilde h} \circ \lambda^{\partial E}]
\stackrel{\eqref{eq:lambdadelE}}{=}
[\mu^{\tilde h} \circ \partial \circ \lambda^E]
=
[0].
$$

Now let $k > \dim \partial F$.
Let $h\in\widehat H^{k}(E;\Z)$ and $z \in Z_{k-1-\dim\partial F}(X;\Z)$.
Choose $\zeta(z) \in \CC_{k-1-\dim\partial F}(X)$ and $a(z) \in C_{k-\dim\partial F}(X;\Z)$ such that $[z - \partial a(z)]_{\partial S_{k-\dim\partial F}} = [\zeta(z)]_{\partial S_{k-\dim\partial F}}$.
Then we compute:
\begin{align*}
\widehat\pi^{\partial E}_!h(z)
&\ist{\eqref{eq:pilambda}}\,
h(\lambda^{\partial E}(z)) \cdot \exp \Big( 2\pi i \int_{a(z)} \fint_{\partial F} \curv(h) \Big) \\
&\ist{\eqref{eq:lambdadelE},\eqref{eq:stokes_fiber}}\,\,\,\,\,\,\,
h((-1)^{\deg(z)} \partial \lambda^E(z)) \cdot \exp \Big( 2\pi i \int_{a(z)} d\fint_{F} \curv(h) \Big)  \\ 
&\ist{\eqref{eq:stokes_fiber}}\,
\exp \Big[ 2\pi i (-1)^{k-\dim F} \Big( \int_{\lambda^E(z)} \curv(h) + \int_{\partial a(z)} \fint_F \curv(h) \Big) \Big] \\
&\ist{\eqref{eq:fint_lambda_cycle}}\,
\exp \Big[ 2\pi i (-1)^{k-\dim F} \Big( \int_{[\zeta(z)]_{\partial S_{k-\dim \partial F}}}\fint_F  \curv(h) + \int_{\partial a(z)} \fint_F \curv(h) \Big) \Big] \\
&\ist{\eqref{eq:zazeta}}\,
\exp \Big( 2\pi i \int_z (-1)^{k-\dim F} \fint_F \curv(h) \Big) \\
&=
\iota\Big( (-1)^{k-\dim F} \fint_F \curv(h) \Big)(z) \,. \qedhere
\end{align*}
\end{proof}

\begin{remark}
Proposition~\ref{prop:fiber_int_rand} says that $\widehat\pi^{\partial E}_!(h)$ is topologically trivial.
However, $\widehat\pi^{\partial E}_!(h)$ is in general not flat, since 
$$
\curv(\widehat \pi_!^{\partial E} h) 
= 
\fint_{\partial F} \curv(h) 
\stackrel{\eqref{eq:stokes_fiber}}{=} 
(-1)^{k-\dim F} d \fint_F\curv (h) \,,
$$
is an exact form, but need not be $0$.
\end{remark}

As a special case of fiber integration for fibers with boundary, we obtain the well-known homotopy formula:

\begin{example}
Differential cohomology is not a generalized cohomology theory, in particular, it is not homotopy invariant.
Let $f:[0,1] \times X \to Y$ be a homotopy between smooth maps $f_0,f_1: X \to Y$ and $h \in \widehat H^k(Y;\Z)$ a differential character.
Then we have the well-known homotopy formula \cite[Prop.~3.28]{Bu12}:
\index{differential characters!homotopy formula}%
\index{homotopy formula}%
$$
f_1^*h - f_0^*h 
=
\iota \Big( \int_0^1 f_s^*\curv(h) ds \,\Big) \,.
$$
This is a special case of \eqref{eq:fiber_int_bound2} for the trivial bundle $X \times [0,1] \to X$: for the left hand side we have $f_1^*h - f_0^*h = \widehat \pi^{\partial E}_! f^*h$.
By the orientation conventions, we obtain for the right hand side
$\fint_F f^*\curv(h) = (-1)^{k-1} \int_0^1 f_s^*\curv(h) ds$ with $k=1$.  
\end{example}

\begin{example}\label{ex:FIk1f1}
Let the fibers of $\pi:E\to X$ be diffeomorphic to compact intervals and carry an orientation.
Hence $\dim F=1$.
\index{fiber integration!for differential characters!$1$-dimensional fiber}%
The boundary of $E$ decomposes as $\partial E = \partial^+E \sqcup \partial^-E$ where $\partial^+E$ consists of the endpoints of the oriented fibers and $\partial^-E$ of the initial points.
The restriction of $\pi$ to $\partial^+E$ is a diffeomorphism whose inverse we denote by $j^+:=(\pi|_{\partial^+E})^{-1}:X\to \partial^+E$, and similarly for $j^-$.

We consider the case $k=1$.
Then for any $h\in \widehat H^1(E;\Z)=C^\infty(E,\Ul)$ we have $\widehat \pi_!^{\partial E}h \in C^\infty(X,\Ul)$ where
$$
\widehat \pi_!^{\partial E}h = 
(h \circ j^+)\cdot(h \circ j^-)^{-1} .
$$
The exponent $-1$ in this formula is due to the fact that the points in $\partial^-E$ inherit a negative orientation.

Recall from Example~\ref{ex:H1} that $\curv(h)=d\tilde h$ where $\tilde h$ is a local lift of $h$.
Integration along the fiber $E_x$ over $x\in X$ yields $\rho(x)=\tilde h(j^+(x))-\tilde h(j^-(x))$.
The ambiguity in the choice of $\tilde h$ cancels and we obtain a global smooth function $\rho:X\to\R$.
Obviously, $\rho$ is a lift of $\widehat \pi_!^{\partial E}h$.
\end{example}

\begin{example}\label{ex:FIk2f1}
Let $\pi:E\to X$ be as in Example~\ref{ex:FIk1f1}.
\index{fiber integration!for differential characters!$0$-dimensional fiber}%
Now we consider the case $k=2$.
Let $P\to E$ be a $\Ul$-bundle with connection $\nabla$ corresponding to $h\in\widehat H^2(X;\Z)$.
Fiber integration along $\partial F$ yields the $\Ul$-bundle with connection over $X$ whose fiber over $x$ is 
$$
(\widehat\pi_!^{\partial E} P)_x = P_{j^+(x)}\otimes P_{j^-(x)}^* \,.
$$
Fiber integration of $\curv(h)$ yields the $1$-form $\rho$ on $X$.
Integrating $\rho$ along a closed curve $c$ in $X$ yields
$$
\exp\Big(2\pi i\int_c\rho\Big)
= \exp\Big(2\pi i\int_{\pi^{-1}(c)}\curv(h)\Big)
= h(j^-_*c-j^+_*c)
= \widehat\pi_!h(c)^{-1} \,.
$$
As explained in Example~\ref{ex:U1Buendel}, the $1$-form $\varrho$ corresponds to the parallel transport in $(P,\nabla) \to E$ along $F$.
\end{example}

\section{Fiber products and the up-down formula}\label{subsec:up-down}
In this section we prove that the fiber integration in a fiber product is the external product of the fiber integrations.
The up-down formula is an immediate consequence. 

Let $E \to X$ and $E' \to X'$ be fiber bundles over smooth spaces $X$ and $X'$ with compact oriented fibers $F$ and $F'$, respectively.
We consider the fiber product \mbox{$E \times E' \xrightarrow{\pi^E \times \pi^{E'}} X \times X'$} as the composition of fiber bundles \mbox{$E \times E' \xrightarrow{\id \times \pi^{E'}} E \times X' \xrightarrow{\pi^E \times \id} X \times X'$}.
Fiber integration on singular cohomology commutes up to sign with the external product. 
Explicitly, for singular cohomology classes $u \in H^k(E;\Z)$ and $u'\in H^{k'}(E';\Z)$, we have:
\begin{align}
\pi_!^{E\times E'}(u\times u')
&= (\pi^E \times \id)_!((\id \times \pi^{E'})_!(u \times u')) \notag \\
&= (\pi^E \times \id)_!(u \times \pi^{E'}_!u') \notag \\
&= (-1)^{(k'-\dim F')\dim F}\pi^E_!u \times \pi^{E'}_!u' \, .\label{eq:pi_prod} 
\end{align}
This follows from \cite[p.~585]{C53} and the functoriality of fiber integration for singular cohomology, proved in \cite[p.~484]{BH53}.

Similarly, for differential forms $\omega \in \Omega^k(E)$ and $\omega' \in \Omega^{k'}(E')$, we have:
\begin{equation}
\fint_{F \times F'} \omega \times \omega' 
=(-1)^{(k'-\dim F')\dim F} \Big(\fint_F  \omega\Big) \times \Big( \fint_{F'} \omega' \Big) \, .
\label{eq:fint_prod}
\end{equation}
The analogous result for differential characters is the following:

\begin{thm}[Fiber integration on fiber products] \index{Theorem!fiber integration on fiber products}%
Let $E \to X$ and $E' \to X'$ be fiber bundles over smooth spaces $X$ and $X'$ with closed oriented fibers $F$ and $F'$, respectively.
Let $h \in \widehat H^k(E;\Z)$ and $h' \in \widehat H^{k'}(E';\Z)$.
Then we have:
\index{fiber integration!for differential characters!fiber products}%
\begin{equation}\label{eq:ext_prod_commute}
\widehat\pi_!^{E \times E'}(h \times h') 
= 
(-1)^{(k'-\dim F')\dim F}\cdot\widehat\pi_!^E h \times \widehat\pi_!^{E'} h' . 
\end{equation}
\end{thm}

\begin{proof}
Conceptually, the proof is just a computation using the explicit formulas we derived for fiber integration and external product. 
The crucial point is the construction transfer maps commuting with external products.

a) 
We compute the curvature of the differential characters $\widehat\pi_!^{E \times E'}(h \times h')$ and $(\widehat\pi_!^Eh \times \widehat\pi_!^{E'}h')$:
\begin{align}
\curv(\widehat\pi^{E \times E'}_!(h \times h')) 
&\stackrel{\eqref{eq:fiber_int_curvature}}{=} 
\fint_{F \times F'}\curv(h \times h') \notag \\
&\stackrel{\eqref{eq:ext_curv}}{=} 
\fint_{F \times F'} \curv(h) \times \curv(h') \notag \\
&\stackrel{\eqref{eq:fint_prod}}{=} 
(-1)^{(k'-\dim F')\dim F}\cdot\Big(\fint_F \curv(h)\Big) \times \Big(\fint_{F'} \curv(h')\Big) \notag \\
&\stackrel{\eqref{eq:fiber_int_curvature}}{=}
 (-1)^{(k'-\dim F')\dim F}\cdot\curv(\widehat\pi^E_!h) \times \curv(\widehat\pi^{E'}_!h') \notag \\
&\stackrel{\eqref{eq:ext_curv}}{=} 
(-1)^{(k'-\dim F')\dim F}\cdot\curv(\widehat\pi^E_!h \times \widehat\pi^{E'}_!h') \, .  \label{eq:curv_pi_x}
\end{align}
Similarly, we find for the characteristic class:
\begin{align}
c(\widehat\pi^{E \times E'}_!(h \times h')) 
&\stackrel{\eqref{eq:cpi_commute}}{=} 
\pi^{E \times E'}_!(c(h \times h')) \notag \\
&\stackrel{\eqref{eq:ext_c}}{=} 
\pi^{E \times E'}_!(c(h) \times c(h')) \notag \\
&\stackrel{\eqref{eq:pi_prod}}{=} 
(-1)^{(k'-\dim F')\dim F}\cdot\pi^E_!(c(h)) \times \pi^{E'}_!(c(h')) \notag \\
&\stackrel{\eqref{eq:cpi_commute}}{=}
(-1)^{(k'-\dim F')\dim F}\cdot c(\widehat\pi^E_!h) \times c(\widehat\pi^{E'}_!h') \notag \\
&\stackrel{\eqref{eq:ext_c}}{=} 
(-1)^{(k'-\dim F')\dim F}\cdot c(\widehat\pi^E_!h \times \widehat\pi^{E'}_!h'). \label{eq:c_pi_x}
\end{align}
Thus the differential characters $\widehat\pi_!^{E \times E'}(h \times h')$ and $(\widehat\pi_!^Eh \times \widehat\pi_!^{E'}h')$ have the same curvature and characteristic class.
By Remark~\ref{rem:torsion_cycles} this implies that they coincide on cycles $z \in Z_{k+k'-\dim F\times F'-1}(X \times X';\Z)$ that represent torsion classes.

b)
Let $z \in Z_{k+k'-\dim F\times F'-1}(X \times X';\Z)$ be a cycle.
As in Section~\ref{subsec:ring}, we choose a splitting $S$ of the cycles in the K\"unneth sequence.
Composing the homomorphism $\zeta^X$ with the pull-back operation $\PBE$, we construct a transfer map $\lambda^X:Z_{*-\dim F}(X;\Z) \to Z_*(E;\Z)$ as in Remark~\ref{rem:lambda}, and similarly for $\lambda^{X'}$.
We use the splitting to extend $\lambda^{X} \otimes \lambda^{X'}$ to a transfer map $\lambda^{X \times X'}:Z_{*-\dim F\times F'}(X\times X';\Z) \to Z_*(E \times E';\Z)$ such that the following diagram is graded commutative:
\begin{equation}\label{eq:diag_PB_xa}
\xymatrix{
Z_*(E;\Z) \otimes Z_*(E';\Z) \ar[r]^(0.58)\times \ar[d] & Z_*(E \times E';\Z) \ar[d] \\
\frac{Z_*(E;\Z)}{\partial S_{*+1}(E;\Z)} \otimes \frac{Z_*(E';\Z)}{\partial S_{*+1}(E';\Z)} \ar[r]^(0.58)\times & \frac{Z_*(E\times E';\Z)}{\partial S_{*+1}(E\times E';\Z)} \\
\ZZ_*(E) \otimes \ZZ_*(E') \ar[r]^(0.58)\times \ar[u]^{\psi^E_* \otimes \psi^{E'}_*} & \ZZ_*(E \times E') \ar[u]_{\psi^{E\times E'}_*} \\
\ZZ_{*-\dim F}(X) \otimes \ZZ_{*-\dim F'}(X') \ar[u]^{\PBE \otimes \PB_{E'}} \ar[r]^(0.58)\times & \ZZ_{*-\dim F \times F'}(X \times X') \ar[u]_{\PB_{E \times E'}} \\
Z_{*-\dim F}(X;\Z) \otimes Z_{*-\dim F'}(X';\Z) \ar@<3pt>[r]^(0.58)K \ar[u]^{\zeta^X \otimes \zeta^{X'}} \ar`l`[uuuu]^{\lambda^X \otimes \lambda^{X'}}[uuuu]
& Z_{*-\dim F \times F'}(X \times X';\Z) \ar@{-->}[l]^(0.42)S \ar[u]_{\zeta^{X \times X'}} \ar`r`[uuuu]_{\lambda^{X\times X'}}[uuuu]
}
\end{equation}
The graded commutativity is caused by the orientation conventions.
As in \eqref{eq:PB_prod}, we have $\PB_{E \times E'}(\zeta_i \times \zeta'_j) = (-1)^{j \cdot \dim F}\PB_E (\zeta_i) \times \PB_{E'}(\zeta'_j)$ for $\zeta_i \in \ZZ_i(X)$ and $\zeta'_j \in \ZZ_j(X')$.
Consequently, $\lambda^{X \times X'}(y_i \times y'_j) = (-1)^{j \cdot \dim F} \lambda^X(y_i) \times \lambda^{X'}(y'_j)$.

Now write $z = K \circ S(z) + (z-K \circ S(z))$.
By the K\"unneth sequence, the cycle $z-K\circ S(z)$ represents a torsion class.
Thus by part a) the differential characters $\widehat\pi_!^{E \times E'}(h \times h')$ and $(\widehat\pi_!^Eh \times \widehat\pi_!^{E'}h')$ coincide on $z-K\circ S(z)$.
Hence it suffices to evaluate them on $K\circ S(z)$.
By \eqref{eq:pilambda}, we have:
\begin{align}
\widehat\pi^{E \times E'}_!(h \times h')(K \circ S(z))
&= 
(h \times h')(\lambda^{X\times X'} (K\circ S(z))) \notag \\
& 
\qquad \cdot \exp \Big( 2\pi i \int_{a(K\circ S(z))} \fint_{F\times F'} \curv(h \times h') \Big) \,.
\label{eq:updown1a}
\end{align}
As in \eqref{eq:Zerl}, we write $S(z) = \sum_{i+j=k+k'-\dim F\times F' -1} \sum_m y_i^m \otimes {y'}_j^m$.
As in the proof of Theorem~\ref{thm:ext_prod_BB}, we write $\zeta^X(y^m_i) = [M^m_i\xrightarrow{g^m_i}X]$ and $\zeta^{X'}({y'}^m_j) = [{M'}^m_j\xrightarrow{{g'}^m_j}X]$.
Then we have $\PBE(\zeta^X(y^m_i)) = [(g^m_i)^*E\xrightarrow{G^m_i}E]$ and $\PB_{E'}(\zeta^{X'}({y'}^m_j)) = [({g'}^m_j)^*E'\xrightarrow{{G'}^m_i}E']$.
Lemma~\ref{lem:x_M_M'} applied to the product stratifolds $({g_i^m})^*E \times ({g'}^m_j)^*E'$ yields 
\begin{align}
(h \times h')(\lambda^{X\times X'}(y^m_i \times {y'}^m_j))
&\ist{\eqref{eq:diag_PB_xa}}
(h \times h')((-1)^{j \cdot \dim F}\lambda^X(y^m_i) \times \lambda^{X'}({y'}^m_j)) \notag \\
&=
\begin{cases}
h(\lambda^X(y^m_{k-\dim F -1}))^{(-1)^{(k'-\dim F')\dim F}\langle c(h'),\lambda^{X'}({y'}^m_{k'-\dim F'}\rangle} & \\
 \qquad \qquad \quad \mbox{for $(i,j)=(k-\dim F -1,k'-\dim F')$} & \\
h(\lambda^{X'}({y'}^m_{k'-\dim F' -1}))^{(-1)^{(k'-\dim F'-1)\dim F}\langle c(h),\lambda^X(y^m_{k-\dim F}\rangle} & \\ 
 \qquad \qquad \quad \mbox{for $(i,j)=(k-\dim F,k'-\dim F'-1)$} & \\
1 \qquad \quad \quad \;\; \mbox{otherwise} & \,. 
\end{cases}
\label{eq:updown3}
\end{align}
Inserting this into \eqref{eq:updown1a} we obtain:
\allowdisplaybreaks{
\begin{align*}
\widehat\pi_!^{E \times E'}(h &\times h')(K\circ S(z)) 
\cdot \exp \Big( -2\pi i \int_{a(K\circ S(z))} \fint_{F\times F'} \curv(h \times h') \Big) \\
&=
(h \times h')(\lambda^{X \times X'}(K \circ S(z))) \\
&=
(h \times h')(\lambda^{X\times X'}(\sum_{i+j=k+k'-\dim F\times F' -1} \sum_m y^m_i \times {y'}^m_j)) \\
&\ist{\eqref{eq:updown3}}\,\,
\prod_m \left[ h(\lambda^X(y^m_{k-\dim F -1}))^{(-1)^{(k'-\dim F')\dim F}\langle c(h'),\lambda^{X'}({y'}^m_{k'-\dim F'})\rangle}  \right. \\
&
\qquad \qquad \left. h(\lambda^{X'}({y'}^m_{k'-\dim F' -1}))^{(-1)^{(k'-\dim F'-1)\dim F}\langle c(h),\lambda^X(y^m_{k-\dim F})\rangle} \right] \\
&\ist{\eqref{eq:pilambda},\eqref{eq:cocycle_pi_c}}\,\,\,\,\,\,\,\,
\prod_m \left[ \widehat\pi_!^E h([\zeta^X(y^m_{k-\dim F -1})]_{\partial S_{k-\dim F}})^{(-1)^{(k'-\dim F')\dim F}\langle \widehat\pi^{E'}_! c(h'),{y'}^m_{k'-\dim F'}\rangle} \right. \\
&
\qquad \left. \widehat\pi_!^{E'} h'([\zeta^{X'}({y'}^m_{k'-\dim F' -1})]_{\partial S_{k'-\dim F'}})^{(-1)^{(k'-\dim F'-1)\dim F}\langle \widehat\pi^E_! c(h),y^m_{k-\dim F}\rangle} \right] \\
&\ist{\eqref{eq:ext_prod_BB}} \,
\Big[
(\widehat\pi_!^E h \times \widehat\pi_!^{E'} h')(K \circ S(z))   \\
&
\qquad \cdot \exp \Big(- 2\pi i \int_{a(K\circ S(z))} \curv(\widehat\pi^E_!h \times \widehat\pi^{E'}_!h') \Big) \Big]^{(-1)^{(k'-\dim F')\dim F}} \,.
\end{align*}
}%
Using \eqref{eq:curv_pi_x}, we conclude 
$$
\widehat\pi_!^{E \times E'}(h \times h')(K \circ S(z))
=
(-1)^{(k'-\dim F')\dim F}\cdot (\widehat\pi_!^Eh \times \widehat\pi_!^{E'}h')(K \circ S(z))
$$ 
which completes the proof.
\end{proof}

Fiber integration for differential forms satisfies the following up-down formula: for any $\eta \in \Omega^k(X)$ and $\omega \in \Omega^l(E)$, we have:%
\begin{equation}\label{eq:up-down_forms}
\fint_F \pi^*\eta \wedge \omega 
= \eta \wedge \fint_F \omega.   
\end{equation}
\index{up-down formula!for differential forms}%
Likewise, fiber integration on singular cohomology satisfies the corresponding up-down formula: for any $u \in H^k(X;\Z)$ and $w \in H^l(E;\Z)$, we have:
\begin{equation*}
\pi_!(\pi^*u \cup w) 
= u \cup \pi_!w. 
\end{equation*}
\index{up-down formula!for singular cohomology}%
For a proof, see~\cite[p.~585]{C53} or \cite[p.~483]{BH53}. 

Now we prove the corresponding up-down formula for fiber integration of differential characters.
The idea of the proof is due to Chern who proved the up-down formula for singular cohomology in \cite[p.~585]{C53}.
The same idea has been used in \cite{BKS10} along the lines of a representation of differential cohomology by cohomology stratifolds. 
 
\begin{thm}[Up-down formula] \index{Theorem!up-down formula}%
Let $E \to X$ be a fiber bundle over a smooth space $X$ with closed oriented fibers $F$.
Let $h \in \widehat H^k(X;\Z)$ and $f \in \widehat H^l(E;\Z)$.
Then  we have
\begin{equation}\label{eq:up_down}
\widehat \pi_!(\pi^*h * f) 
= h * (\widehat \pi_!f) \in \widehat H^{k+l-\dim F}(X;\Z).
\end{equation}
\index{up-down formula!for differential characters}%
\index{differential characters!up-down formula}%
\index{fiber integration!for differential characters!up-down formula}%
\end{thm}

\begin{proof}
We decompose the fiber product $E \times E' \to X \times X'$ as the composite fiber bundle $E\times E \xrightarrow{\pi \times \id_E} X \times E \xrightarrow{\id_X \times \pi} X \times X$.
Let $\Delta_E:E \to E \times E$ and $\Delta_X:X \to X \times X$ denote the diagonal maps.
Then we have the bundle map 
\begin{equation*}
\xymatrix{
E \ar[rrr]^{(\pi \times \id_E) \circ \Delta_E} \ar[d]_\pi &&& X \times E \ar[d]^{\id_X \times \pi} \\
X \ar[rrr]_{\Delta_X} &&& X \times X 
} 
\end{equation*}
The up-down formula now follows from the product formula \eqref{eq:ext_prod_commute} and the naturality \eqref{eq:fiber_int_natural} of fiber integration:
\begin{align*}
\widehat\pi^E_!(\pi^*h * f) 
&\stackrel{\eqref{eq:ext_int}}{=} \widehat\pi^E_!(\Delta_E^*(\pi^*h \times f)) \\  
&\stackrel{\eqref{eq:ext_nat}}{=} \widehat\pi^E_!(\Delta_E^*(\pi \times \id_E)^*(h \times f)) \\
&\stackrel{\eqref{eq:fiber_int_natural}}{=} \Delta_X^* (\widehat\pi^{X\times E}_!(h\times f)) \\
&\stackrel{\eqref{eq:ext_prod_commute}}{=} \Delta_X^* (h \times \widehat\pi^E_!f) \\
&\stackrel{\eqref{eq:ext_int}}{=} h*\widehat\pi^E_!f. 
\end{align*}
There is no sign in the second last equation because the fiber over the first factor is zero-dimensional.
\end{proof}

%%%%%%%%%%%%%%%%%%%%%%%%%%%%%%%%%%%%%%%%%%%%%%%%%%%%%%%%%%%%%%%%%%%%%%%%%
\chapter{Relative differential characters}\label{sec:rel_diff_charact}
%%%%%%%%%%%%%%%%%%%%%%%%%%%%%%%%%%%%%%%%%%%%%%%%%%%%%%%%%%%%%%%%%%%%%%%%%

In this chapter, we discuss several aspects of relative differential characters, defined in \cite{BT06}.
From a geometric point of view, relative differential characters are to be considered as topological trivializations or global sections of differential characters. 
We explain this point of view in Section~\ref{subsec:Rel_1}.

From a topological point of view, the group of relative differential characters should be considered as a relative version of differential cohomology.
However, differential cohomology is not a (generalized) cohomology theory in the sense of Eilenberg and Steenrod. 
In particular, one cannot expect to obtain the usual long exact sequence relating the groups of relative and absolute differential characters.
In Section~\ref{subsec:Rel_2}, we derive an exact sequence that relates the groups of relative and absolute differential characters.
This sequence characterizes in particular the existence and uniqueness of global sections.

\section{Definition and examples}\label{subsec:Rel_1}

Let $k\ge1$ and $\varphi: A \to X$ a smooth map.
Relative differential characters in $\widehat H^k_\varphi(X,A;\Z)$ may be considered as differential characters on $X$ with sections along the map $\varphi$.
We briefly recall the construction of $\widehat H^k_\varphi(X,A;\Z)$ from \cite{BT06}.
Then we construct an exact sequence which characterizes those differential characters in $\widehat H^k(X;\Z)$ which admit sections along the map $\varphi$, i.e., which are in the image of the natural map $\widehat H^k_\varphi(X,A;\Z) \to \widehat H^k(X;\Z)$.

The {\em mapping cone complex} of a smooth map $\varphi: A \to X$ is the complex $C_k^\varphi(X,A;\Z) := C_k(X;\Z) \times C_{k-1}(A;\Z)$ of pairs of smooth singular chains with the differential $\partial_\phi (s,t) := (\partial s + \varphi_* t,-\partial t)$.
\index{mapping cone complex!for singular chains}%
\index{+CkphiXAZ@$C_k^\varphi(X,A;\Z)$, group of relative chains}% 
\index{relative chains}%
\index{+Delphi@$\partial_\phi (s,t)$, relative boundary map}%
The homology $H_k^\phi(X,A;\Z)$ of this complex coincides with the homology of the mapping cone of $\varphi$ in the topological sense.
\index{+HkphiXAZ@$H_k^\phi(X,A;\Z)$, homology of the mapping cone complex}
For the special case of an embedding $A \subset X$ it coincides with the relative homology $H_k(X,A;\Z)$.

Similarly, we consider the complex $\Omega^k_\varphi(X,A) := \Omega^k(X) \times \Omega^{k-1}(A)$ of pairs of differential forms with the differential 
$d_\phi(\omega,\vartheta) := (d\omega,\varphi^*\omega - d\vartheta)$.
\index{mapping cone complex!for differential forms}%
\index{+OmegakphiXA@$\Omega^k_\varphi(X,A)$, space of relative differential forms}%
\index{relative differential forms}%
\index{+Dphi@$d_\phi$, relative exterior differential}%
The homology $H^k_{\mathrm{dR},\phi}(X,A)$ of this complex is the relative de Rham cohomology for the map $\varphi$, as explained in \cite[p.~78]{BT82}.
\index{+HkdRphiXA@$H^k_{\mathrm{dR},\phi}(X,A)$, relative de Rham cohomology}
\index{relative de Rham cohomology}%

We denote by $Z_k^\varphi(X,A;\Z)$ the group of cycles of the mapping cone complex and by $B_k^\phi(X,A;\Z)$ the space of boundaries. 
\index{relative cycles}%
\index{+ZkphiXAZ@$Z_k^\varphi(X,A;\Z)$, group of cycles of the mapping cone complex}
The group of relative differential characters is defined as: 
\begin{equation*}
\widehat H^k_\varphi(X,A;\Z) 
:= 
\big\{\, f \in \Hom(Z^\varphi_{k-1}(X,A;\Z),\Ul) \,\big|\, f \circ \partial_\phi \in \Omega^k_\varphi(X,A) \,\big\} \,.  
\end{equation*}
\index{+HhatkphiXAZ@$\widehat H^k_\varphi(X,A;\Z)$, group of relative differential characters}%
\index{relative differential characters}%
\index{differential characters!relative $\sim$}%
The notation $f \circ \partial_\phi \in \Omega^k_\varphi(X,A)$ means that there exists a pair of differential forms $(\omega,\vartheta) \in \Omega^k_\varphi(X,A)$ such that for every pair of smooth singular chains $(x,y) \in C_k^\varphi(X,A)$ we have
\begin{equation}
f(\partial_\phi(x,y))
= 
\exp \Big[ 2 \pi i \cdot \Big( \int_x \omega + \int_y \vartheta \Big) \Big] \,. \label{eq:def_rel_diff_charact_2}
\end{equation}
The form $\omega=: \curv(f)$ in the definition is called the {\em curvature} of the relative differential character $f$ and the form $\vartheta =: \cov(f)$ is called its {\em covariant derivative}.
\index{curvature!of a relative differential character}%
\index{covariant derivative!of a relative differential character}%
\index{+Cov@$\cov$, covariant derivative of a relative differential character}%
As in the absolute case, the curvature is uniquely determined by the differential character.
For $k\ge2$, this is also true for the covariant derivative.
For $k=1$, the function $\vartheta$ is unique only up to addition of a locally constant integer valued function, see Example~\ref{ex:H1rel}.

It is shown in \cite[p.~273f]{BT06} that relative differential characters $f\in\widehat H^k_\varphi(X,A;\Z)$ have characteristic classes $c(f) \in H^k_\phi(X,A;\Z)$, the $k$-th cohomology of the mapping cone complex.
\index{+Cd@$c$, characteristic class of a relative differential character}%
\index{characteristic class!of a relative differential character}
By \cite[Thm.~2.4]{BT06}, the group $\widehat H^k_\varphi(X,A;\Z)$ fits into short exact sequences similar to the ones in \eqref{eq:3x3diagram}:
\begin{equation*}
\xymatrix@R=5mm{
0 \ar[r] &
\frac{\Omega^{k-1}_\varphi(X,A)}{\Omega^{k-1}_{\varphi,0}(X,A)} \ar[r]^\iota &
\widehat H^k_\varphi(X,A;\Z) \ar^c[rr] &&
H^k_\varphi(X,A;\Z) \ar[r] \ar[r] &
0 ,\\
0 \ar[r] &
H^{k-1}_\varphi(X,A;\Ul) \ar[r] &
\widehat H^k_\varphi(X,A;\Z) \ar^{(\curv,\cov)}[rr] &&
\Omega^k_{\varphi,0}(X,A) \ar[r] &
0 .}
\end{equation*}
Here $\Omega^k_{\varphi,0}(X,A)$ denotes the space of all $d_\phi$-closed pairs $(\omega,\vartheta) \in \Omega^k_\phi(X,A)$ with integral periods, i.e., $\int_{(s,t)} (\omega,\vartheta) \in \Z$ for all relative cycles $(s,t) \in Z^k_\phi(X,A;\Z)$.
\index{+OmegakphioXA@$\Omega^k_{\varphi,0}(X,A)$, space of closed relative differential forms with integral periods}

Furthermore, we have the obvious maps \index{+I@$\ti:\widehat H^{k-1}(A;\Z)\to \widehat H^k_\phi(X,A;\Z)$}\index{+P@$\vds:\widehat H^k_\phi(X,A;\Z)\to\widehat H^k(X;\Z)$}
\begin{equation}\label{eq:natural_maps_rel_diff_charact}
\xymatrix{
\widehat H^{k-1}(A;\Z) 
\ar^{\ti}[r] &
\widehat H^k_\varphi(X,A;\Z)
\ar^{\vds}[r] &
\widehat H^k(X;\Z) 
}
\end{equation}
which map a differential character $g \in \widehat H^{k-1}(A;\Z)$ to $\ti(g):(s,t) \mapsto g(t)$ and a relative differential character $f \in \widehat H^k_\varphi(X,A;\Z)$ to $\vds(f):z \mapsto f(z,0)$.
\index{+Ii@$\ti:\widehat H^{*-1}(A;\Z)\to \widehat H^*_\varphi(X,A;\Z)$}
\index{+Pp@$\vds:\widehat H^*_\varphi(X,A;\Z) \to \widehat H^*(X;\Z)$}
One easily checks that $\curv(\ti(g))=0$, $\cov(\ti(g))=-\curv(g)$, and $\curv(\vds(f))=\curv(f)$.

\begin{remark}
Note that $\vds$ is defined also for $k=1$.
As in the absolute case we set $\widehat H^0_\phi(X,A;\Z):=H^0_\phi(X,A;\Z)$ and $\widehat H^k_\phi(X,A;\Z):=0$ for $k<0$.
Moreover, $\ti:\widehat H^0_\phi(A;\Z) \to \widehat H^1_\phi(X,A;\Z)$ is defined to be zero while $\vds:\widehat H^0_\varphi(X,A;\Z) \to
\widehat H^0(X;\Z)$ is defined to coincide with the usual map in the long exact sequence
\[
\xymatrix{
0
\ar[r] &
H^0_\phi(X,A;\Z)
\ar[r] &
H^0(X;\Z) 
\ar[r] & \cdots
}
\]
\end{remark}

\begin{definition} \index{Definition!section}%
Let $\varphi:A \to X$ be a smooth map of differentiable manifolds.
A differential character $h \in \widehat H^k(X;\Z)$ is said to {\em admit sections} along $\varphi$ if it lies in the image of the map $\vds:\widehat H^k_\varphi(X,A;\Z) \to \widehat H^k(X;\Z)$.
\index{differential characters!admitting sections}%

Let $h \in \im(\vds) \subset \widehat H^k(X;\Z)$ be a differential character that admits sections along the map $\varphi$.
Then any relative differential character $f \in \widehat H^k_\varphi(X,A;\Z)$ with $\vds(f) = h$ is called a {\em section} of $h$ along $\phi$.
\index{section!of a differential character}%
A section $f \in \widehat H^k_\phi(X,A;\Z)$ of $\vds(f) \in \widehat H^k(X;\Z)$ along $\phi$ is called {\em parallel} if $\cov(f)=0$.
\index{section!parallel $\sim$} 
\end{definition}

\begin{example}\label{ex:H1rel}
Let $k=1$.
Since $Z^\phi_0(X,A;\Z) = Z_0(X;\Z)$, any relative differential character of degree $1$ corresponds to a function $\bar f:X \to \Ul$ as in the absolute case.
Using \eqref{eq:def_rel_diff_charact_2} with $y=0$ one sees that $\bar f$ is smooth and $\curv(f) = d\tilde f$ where $\tilde f$ is a local lift of $\bar f$ as in Example~\ref{ex:H1}.
Equation \eqref{eq:def_rel_diff_charact_2} with $x=0$ shows that $\vartheta$ is a lift of $\bar f\circ \phi$ on $A$.
Such a lift is unique only up to addition of a locally constant integer valued function.

To summarize, $\widehat H^1_\phi(X,A;\Z)$ is the subgroup of $\widehat H^1(X;\Z)=C^\infty(X,\Ul)$ containing those functions $\bar f$ which are trivial along $\phi$ in the sense that $\bar f\circ \phi$ has a lift.
\index{relative differential characters!of degree $1$}%
\end{example}

\begin{example}\label{ex:H2rel}
Let $k=2$.
Given $f\in \widehat H^2_\phi(X,A;\Z)$ we have $\vds(f)\in \widehat H^2(X;\Z)$ which by Example~\ref{ex:U1Buendel} corresponds to a $\Ul$-principal bundle $P\to X$ with connection $\nabla$.
We pull back $P$ and $\nabla$ along $\phi$ and we obtain a $\Ul$-principal bundle $\phi^*P\to A$ with connection $\phi^*\nabla$.
\index{relative differential characters!of degree $2$}%

\emph{Relative characters determine sections}.
Fix $x_0\in A$.
For any two curves $c$ and $c'$ emanating from $x_0$ and ending at the same point $x\in A$, we look at the cycle $c-c'\in Z_1(A;\Z)$.
Using \eqref{eq:def_rel_diff_charact_2} we compute:
\begin{align*}
\phi^*(\vds f)(c-c')
&=
(\vds f)(\phi_*(c-c')) \\
&=
f(\phi_*(c-c'),0) \\
&=
f(\partial_\phi(0,c-c')) \\
&=
\exp \Big( 2 \pi i \int_{c-c'} \vartheta \Big) \\
&=
\exp \Big( 2 \pi i \int_{c} \vartheta \Big) \cdot
\exp \Big( 2 \pi i \int_{c'} \vartheta \Big)^{-1} .
\end{align*}
We recall from Example~\ref{ex:U1Buendel} that for any $p_0\in\phi^*P$ over
$x_0$ we have
$$
(\PP^{\phi^*\nabla}_{c'})^{-1}\circ \PP^{\phi^*\nabla}_c (p_0) 
= \PP^{\phi^*\nabla}_{c-c'}(p_0)
= p_0\cdot\phi^*(\vds f)(c-c').
$$
Therefore
$$
\PP^{\phi^*\nabla}_c (p_0) \cdot \exp \Big( 2 \pi i \int_{c} \vartheta
\Big)^{-1}
=
\PP^{\phi^*\nabla}_{c'} (p_0) \cdot \exp \Big( 2 \pi i \int_{c'} \vartheta
\Big)^{-1} .
$$
Hence the expression 
$$
\PP^{\phi^*\nabla}_c (p_0) \cdot \exp \Big( 2 \pi i \int_{c} \vartheta
\Big)^{-1}
$$
depends on $x$ but not on the choice of curve connecting $x_0$ and $x$.
Fixing $x_0$ and $p_0$ we can define a smooth section of $\phi^*P$ over the
connected component containing $x_0$ by 
\begin{equation}
\sigma(x) := 
\PP^{\phi^*\nabla}_c (p_0) \cdot \exp \Big( 2 \pi i \int_{c} \vartheta
\Big)^{-1} .
\label{eq:Schnitt}
\end{equation}
Choosing $x_0$ and $p_0$ in each connected component of $A$ we obtain a smooth
section of $\phi^*P$ over all of $A$.
If $\sigma'$ is a section obtained by different choices of the $x_0$'s and
$p_0$'s, then $\sigma'=\sigma\cdot u$ where $u:A\to \Ul$ is a locally constant
function.
\index{+Sigma@$\sigma$, section of $\Ul$-bundle or complex line bundle}
\index{section!of a line bundle}%

\emph{Isomorphism classes of sections}.
We further restrict the freedom in the choices of the $p_0$'s.
Consider the pull-back diagram
$$
\xymatrix{
\phi^*P \ar[r]^\Phi \ar[d] & P \ar[d] \\
A \ar[r]^\phi & X
}
$$
Equation \eqref{eq:Schnitt} yields for any closed curve $c$ in $A$ starting and
ending at $x_0$ that
\begin{align*}
\PP^\nabla_{\phi_*c}(\Phi(p_0))
&=
\Phi(p_0)\cdot \exp \Big( 2 \pi i
\int_{c} \vartheta \Big) \\
&=
\Phi(p_0)\cdot f(\partial_\phi(0,c))  .
\end{align*}
For a closed curve $s$ in $X$ (instead of $A$) starting and ending at
$\phi(x_0)$ we have more generally
$$
\PP^\nabla_{s}(\Phi(p_0)) = \Phi(p_0)\cdot f(s,0).
$$
Now, if $x_0$ and $x_0'$ lie in different connected components of $A$ but
$\phi(x_0)$  and $\phi(x_0')$ lie in the same connected component of $X$, then
we demand for any curve $s$ in $X$ starting at $\phi(x_0)$ and ending at
$\phi(x_0')$ that 
$$
\PP^\nabla_{s}(\Phi(p_0)) = \Phi(p_0')\cdot f(s,x_0-x_0') .
$$
In this way, the choice of $p_0'$ is determined by the choice of $p_0$.
Moreover, this relation does not depend on the choice of $s$.
With this additional requirement the freedom to choose the $p_0$'s reduces to
one choice for each maximal set of $x_0$'s which are mapped to the same
connected component of $X$.
Hence two sections $\sigma$ and $\sigma'$ constructed in this way are related by
$\sigma' = \sigma\cdot (u\circ\phi)$ where $u:X\to \Ul$ is a locally
constant function.

Said differently, a relative differential character $f \in \widehat H^2_\phi(X,A;\Z)$ determines an isomorphism class $[P,\nabla,\sigma]$ of $\Ul$-bundles with connection $(P,\nabla)$ and section $\sigma$ along the map $\phi$.
\index{+Pnablasigma@$[P,\nabla,\sigma]$, isomorphism class of complex line bundle with connection and section}
Here $(P,\nabla,\sigma)$ is identified with $(P',\nabla',\sigma')$ if and only if there is a bundle isomorphism $\Psi:P \to P'$ such that $\nabla = \Psi^*\nabla'$ and $\Phi' \circ \sigma' = \Psi \circ \Phi \circ \sigma$. 
In particular, sections of the pull-back bundle are identified by bundle isomorphisms of $P$ (and not of the pull-back bundle $\phi^*P$).

\emph{Sections determine relative characters}.
Conversely, let $P \to X$ be a $\Ul$-bundle with connection $\nabla$ and $\sigma$ a section of $\phi^*P$ over $A$.
For any relative cycle of the form $(s,x-x')$ we define $f(s,x-x')$ by
$$
\PP^\nabla_{s}(\Phi(\sigma(x))) = \Phi(\sigma(x'))\cdot f(s,x-x') .
$$
Since $Z^\phi_1(X,A;\Z)$ is generated by cycles of this form, the differential
character $f$ is uniquely determined.
The definition of $f$ is invariant under bundle isomorphisms as defined above.  

\emph{Curvature and connection form}.
To summarize, we have a 1-1 correspondence between relative differential characters $f \in \widehat H^2_\phi(X,A;\Z)$ and isomorphism classes $[P,\nabla,\sigma]$ of $\Ul$-bundles with connection and section along $\phi$.   
Under this correspondence, $-2\pi i \cdot \curv(f)$ is the curvature form of $(P,\nabla)$ and $-2\pi i\cdot \cov(f)$ is the connection $1$-form of $\phi^*\nabla$ with respect to $\sigma$.
\end{example}

\begin{remark}[Relative differential cohomology]\index{relative differential cohomology}
The group $\widehat H^k_\varphi(X,A;\Z)$ of relative differential characters may be considered as a relative differential cohomology group.
There have appeared different versions of relative differential cohomology in the literature:
In \cite{HL01}, de Rham-Federer currents on manifolds $X$ with boundary are used to describe differential cohomology relative to $A = \partial X$.  
\index{de Rham-Federer currents}%
In \cite{U11}, relative differential cohomology is defined for the case of a submanifold $A \subset X$.
In both these models, the curvature of a relative cohomology class vanishes upon restriction to the subset $A$. 

However, the covariant derivatives of relative differential characters need not be closed.
In this sense, the relative differential cohomology group defined by relative differential characters is more general (or is a larger group) than the ones the relative differential cohomology groups described in \cite{HL01} and \cite{U11}.
\end{remark}

\section{Existence of sections}\label{subsec:Rel_2}
Since differential cohomology is not a (generalized) cohomology theory, the question arises whether there are long exact sequences that relate the absolute and relative differential cohomology groups.
Here we fit the maps from \eqref{eq:natural_maps_rel_diff_charact} into an exact sequence that characterizes those differential characters in $\widehat H^k(X;\Z)$ that admit sections along $\varphi$.

\begin{thm}[Exact sequence]\label{thm:rel_diff_charact_exact_sequence} \index{Theorem!exact sequence}%
Let $\varphi:A \to X$ be a smooth map.
Then the following sequences are exact:
\begin{equation}\label{eq:rel_diff_charact_exact_sequence}
0 \to
\phi^* \widehat H^{k-1}_\mathrm{flat}(X;\Z) \to
\widehat H^{k-1}(A;\Z)  \xrightarrow{\ti}
\widehat H^k_\varphi(X,A;\Z) \xrightarrow{\vds}
\widehat H^k(X;\Z) \xrightarrow{\varphi^* \circ c}
\varphi^*H^k(X;\Z) \to
0 \, .
\end{equation}
if $k\ge2$ and
\begin{equation}\label{eq:rel_diff_charact_exact_sequence01}
0 \xrightarrow{}
\widehat H^k_\varphi(X,A;\Z) \xrightarrow{\vds}
\widehat H^k(X;\Z) \xrightarrow{\varphi^* \circ c}
\varphi^*H^k(X;\Z) \to
0 \, .
\end{equation}
if $k=0$ or $k=1$.
\end{thm}

\begin{remark}
Sequences~\eqref{eq:rel_diff_charact_exact_sequence} and \eqref{eq:rel_diff_charact_exact_sequence01} can be derived by homological algebraic methods.
There are several ways to obtain differential cohomology as the cohomology of a chain complex. 
The smooth Deligne complex and the Hopkins-Singer complex that compute degree-$k$ differential cohomology both depend on $k$.
\index{Hopkins-Singer complex}%
\index{Deligne complex}%
Thus the cohomology groups in the long exact sequence obtained from the corresponding mapping cone complexes coincide with differential cohomology only in degree $k$.

The Hopkins-Singer complex can be modified so that all its cohomologies realize differential cohomology, see \cite[p.~271]{BT06}.
The mapping cone construction then yields a long exact sequence where the absolute cohomology groups coincide with differential cohomology. 
But the corresponding relative groups for this modified complex are only subquotients of the groups of relative differential characters, \cite[p.~278ff.]{BT06}.

Another long exact sequence for relative and absolute differential (generalized) cohomology is constructed in \cite[Thm.~2.7]{U11}. 
Another way to define global trivializations of differential cohomology, based on the Hopkins-Singer complex, is described in \cite{R12}.

Here we do not use any of these identifications of the groups of differential characters with the cohomologies of a chain complex, but give a direct proof.   
\end{remark}

\begin{proof}[Proof of Theorem~\ref{thm:rel_diff_charact_exact_sequence}]
We only consider the case $k\ge2$ because the case $k=0$ is obvious and the case $k=1$ is similar to but simpler than the case $k\ge2$.

At several steps in the proof we use the fact that the group $\Ul$ is divisible, hence that for every injective group homomorphism $G' \to G$ the induced homomorphism $\Hom(G,\Ul) \to \Hom(G',\Ul)$ is surjective, see e.g.\ \cite[pp.~32 and 372]{Pr}.
In other words, any homomorphism from a subgroup of $G$ to $\Ul$ can be extended to a homomorphism from $G$ to $\Ul$.

a) 
The map $\phi^*\widehat H^{k-1}_\mathrm{flat}(X;\Z) \to \widehat H^{k-1}(A;\Z)$ is the inclusion of a subgroup and hence injective.

b)
We prove exactness at $\widehat H^{k-1}(A;\Z)$.
For $\bar g \in \widehat H^{k-1}_\mathrm{flat}(X;\Z)$ and any $(s,t) \in Z_{k-1}^\varphi(X,A;\Z)$ we have:
\begin{align*}
\ti(\varphi^*\bar g)(s,t)
&= (\varphi^*\bar g)(t) \\
&= \bar g(\varphi_*t) \\
&= \bar g(-\partial s) \\
&= \exp \Big( 2\pi i \int_s -\curv(\bar g) \Big) \\
&= 1 \,,  
\end{align*}
because $\curv(\bar g)=0$.
Hence $\phi^*\widehat H^{k-1}_\mathrm{flat}(X;\Z) \subset \ker(\ti)$.

To show the converse inclusion we pick $g\in \widehat H^{k-1}(A;\Z)$ with $\ti(g)=1$.
Let $Q \subset Z_{k-2}(A;\Z)$ be the subgroup of those $t \in Z_{k-2}(A;\Z)$ for which there exists an $s \in C_{k-1}(X;\Z)$ such that $(s,t) \in Z_{k-1}^\varphi(X,A;\Z)$.
The condition $\ti(g)=1$ is equivalent to $g$ being trivial on $Q$.
We construct $\bar g \in \widehat H^{k-1}_\mathrm{flat}(X;\Z)$ such that $\phi^* \bar g = g$.
Let $Q'\subset Z_{k-2}(X;\Z)$ be the subgroup generated by $\phi_*Z_{k-2}(A;\Z)$ and $B_{k-2}(X;\Z)$.
We define a group homomorphism $\bar g:Q' \to \Ul$ by setting
\begin{align}
\bar g(\varphi_* x)
&:=  g(x) \label{eq:include_H_flat_2} \,,\\
\bar g(\partial y)
&:= 1 \label{eq:include_H_flat_1} \,.
\end{align}
The conditions are consistent since $\phi_*Z_{k-2}(A;\Z) \cap B_{k-2}(X;\Z) = \phi_*Q$.
By \eqref{eq:include_H_flat_1}, any extension to a group homomorphism $\bar g:Z_{k-2}(X;\Z) \to \Ul$ yields a flat differential character on $X$.
By \eqref{eq:include_H_flat_2}, we have $\phi^*\bar g=g$.

c)
We prove exactness at $\widehat H^k_\varphi(X,A;\Z)$.
For every $z \in Z_{k-1}(X;\Z)$, we have $\vds(\ti(g))(z) = \ti(g)(z,0) = g(0) = 1$.
Hence $\im(\ti) \subset \ker(\vds)$.

Conversely, let $f \in \ker(\vds)$.
Thus $f(z,0) = 1$ for every $z \in Z_{k-1}(X;\Z)$.
For cycles $(s,t), (s',t) \in Z_{k-1}^\varphi(X,A;\Z)$, we have $\partial(s-s') = -\phi_*t + \phi_*t =0$.
Hence $f(s-s',0)=1$ and thus $f(s,t) = f(s',t)$.
Let $Q \subset Z_{k-2}(A;\Z)$ be the subgroup defined in b).
We define a group homomorphism $g:Q \to \Ul$ by setting $g(t) := f(s,t)$. 

Now $B_{k-2}(A;\Z) \subset Q$, since for $t = \partial y$, we have $(-\varphi_* y,t) = \partial_\phi(0,-y) \in B_{k-1}^\phi(X,A;\Z) \subset Z_{k-1}^\varphi(X,A;\Z)$.
We can extend $g$ as a group homomorphism \mbox{$g: Z_{k-2}(A;\Z) \to \Ul$.}
On $B_{k-2}(A;\Z)$, we have
$$
g(\partial y)
= f(\partial_\phi(0,-y))
= \exp \Big( -2\pi i \int_y \cov(f) \Big) \,.
$$
Hence $g: Z_{k-2}(A;\Z) \to \Ul$ is a differential character $g \in \widehat H^{k-1}(A;\Z)$ with curvature $\curv(g) = -\cov(f)$.
Since $f(s,t) = g(t)$ for every $(s,t) \in Z_k^\phi(X,A;\Z)$, we have $f=\ti(g)$.

d)
For the exactness at $\widehat H^k(X;\Z)$ consider the following commutative diagram with exact columns:
\begin{equation*}
\xymatrix{
0 \ar[d] & 0 \ar[d] & 0 \ar[d]  \\
\frac{\Omega^{k-1}_\varphi(X,A)}{\Omega^{k-1}_{\varphi,0}(X,A)} \ar[r] \ar^\iota[d] 
  & \frac{\Omega^{k-1}(X)}{\Omega^{k-1}_0(X)} \ar^{\varphi^*}[r] \ar^\iota[d]
  & \frac{\Omega^{k-1}(A)}{\Omega^{k-1}_0(A)} \ar^\iota[d] \\
\widehat H^k_\varphi(X,A;\Z) \ar^c[d] \ar^{\vds}[r] 
  & \widehat H^k(X;\Z) \ar^c[d] \ar^{\varphi^*}[r] 
  & \widehat H^k(A;\Z) \ar^c[d] \\
H^k_\varphi(X,A;\Z) \ar[r] \ar[d]
  & H^k(X;\Z) \ar^{\varphi^*}[r]  \ar[d] 
  & H^k(A;\Z)  \ar[d] \\
0 & 0 & 0 
}
\end{equation*}
The bottom row is part of the long exact cohomology sequence obtained from the short exact sequence of chain complexes 
$$
0 \to C_*(X;\Z) \to C_*^\phi(X,A;\Z) \to C_{*-1}(A;\Z) \to 0\,.
$$
Let $f \in \widehat H^k_\phi(X,A;\Z)$.
From the commutativity of the diagram and the exactness of the bottom row we conclude $c(\phi^*\vds(f))=0$.
Hence $\im(\vds) \subset \mbox{$\ker(\varphi^* \circ c)$}$.

Conversely, let $h \in \ker(\varphi^* \circ c)$.
We construct a section along $\phi$.
From the diagram we conclude that there exists a differential form $\chi \in \Omega^{k-1}(A)$ such that $\varphi^*h = \iota(\chi)$.
Hence $\varphi^*\curv(h) = \curv(\varphi^*h) = \curv(\iota(\chi))= d\chi$.
Let $W \subset Z_{k-1}^\varphi(X,A;\Z)$ be the subgroup generated by $B_{k-1}^\phi(X,A;\Z)$ and by pairs of the form $(s,t) = (z,0)$ with $z \in Z_{k-1}(X;\Z)$.
We define a group homomorphism $f:W \to \Ul$ by setting:
\begin{align}
f(\partial_\phi(x,y))
&:= \exp \Big[ 2 \pi i \cdot \Big( \int_x \curv(h) + \int_y \chi  \Big) \Big] \, , \label{eq:constr_sections_1} \\
f((z,0))
&:= h(z) \label{eq:constr_sections_2} \,.
\end{align}
This definition is consistent, since for $(z,0) = \partial_\phi(x,y)=(\partial x + \varphi_* y,-\partial y)$, we have 
\begin{align*}
f((z,0))\,\,
&\ist{\eqref{eq:constr_sections_2}}\,\, 
h( \partial x + \varphi_* y) \\
&= 
\exp \Big( 2\pi i \int_x \curv(h) \Big) \cdot \varphi^*h(y) \\
&= 
\exp \Big( 2\pi i \int_x \curv(h) \Big) \cdot \iota(\chi)(y) \\
&= 
\exp \Big[ 2 \pi i \cdot \Big( \int_x \curv(h) + \int_y \chi  \Big) \Big] \\
&\ist{\eqref{eq:constr_sections_1}}\,\, 
f(\partial_\phi(x,y)) \,.
\end{align*}
We extend $f$ to a $\Ul$-valued group homomorphism on $Z_{k-1}^\varphi(X,A;\Z)$.
By equation \eqref{eq:constr_sections_1}, this homomorphism satisfies \eqref{eq:def_rel_diff_charact_2}.
Thus $f \in \widehat H^k_\varphi(X,A;\Z)$.
Equation~\eqref{eq:constr_sections_2} implies that $\vds(f) = h$.

e)
Finally, the map $\varphi^* \circ c: \widehat H^k(X;\Z) \to \varphi^*H^k(X;\Z)$ is surjective since $c$ is surjective by \eqref{eq:3x3diagram} and pull-back along $\varphi$ is surjective onto its image.
\end{proof}

\begin{cor}[Long exact sequence] \index{Corollary!long exact sequence}%
For $k \geq 2$ we have the following long exact sequence:
\begin{equation}\label{eq:rel_diff_charact_exact_sequence2}
\ldots \to
H^{k-2}(X;\Ul) \xrightarrow{j \circ\varphi^*}
\widehat H^{k-1}(A;\Z)  \xrightarrow{\ti}
\widehat H^k_\varphi(X,A;\Z) \xrightarrow{\vds}
\widehat H^k(X;\Z) \xrightarrow{c \circ \varphi^*}
H^k(A;\Z) \to
\cdots \, .
\end{equation}
The sequence extends on the left and on the right as the mapping cone sequence for singular cohomology with coefficients $\Ul$ and $\Z$, respectively. 
\index{differential characters!long exact sequence}%
\index{relative differential characters!long exact sequence}%
\end{cor}

\begin{proof}
We use the identification $H^{k-2}(X;\Ul) \xrightarrow{\cong} \widehat H^{k-1}_\mathrm{flat}(X;\Z)$ from diagram~\eqref{eq:3x3diagram}.
In particular, the map $j: H^{k-2}(X;\Ul) \to \widehat H^{k-1}(X;\Z)$ is injective. 

Exactness at the three middle terms is clear from Theorem~\ref{thm:rel_diff_charact_exact_sequence}.
From the mapping cone sequence for cohomology with $\Ul$-coefficients, we conclude:
\begin{align*}
\ker\big[\, j \circ \varphi^* : H^{k-2}(X;\Ul) \to \widehat H^{k-1}(A;\Z) \,\big] 
&= 
\ker\big[\,\varphi^*:H^{k-2}(X;\Ul) \to H^{k-2}(A;\Ul)\,\big] \\
&=
\im\big[\,H^{k-2}_\varphi(X,A;\Ul) \to H^{k-2}(X;\Ul)\,\big] \,.
\end{align*}
This proves exactness at $H^{k-2}(X;\Ul)$.

From the mapping cone sequence for cohomology with integral coefficients and surjectivity of $c$, we conclude:
\begin{align*}
\ker\big[\,H^k(A;\Z) \to H^{k+1}_\varphi(X,A;\Z)\,\big]
&=
\im\big[\,\varphi^*:H^k(X;\Z) \to H^k(A;\Z)\,\big] \\
&=
\im\big[\,\varphi^* \circ c:\widehat H^k(X;\Z) \to H^k(A;\Z)\,\big] \,.
\end{align*}
This proves exactness at $H^k(A;\Z)$.
\end{proof}

A differential character $h \in \widehat H^k(X;\Z)$ is called {\em topologically trivial along $\phi$} if $\varphi^*c(h)=0$.
\index{differential characters!topologically trivial along $\phi$}%
\index{topologically trivial along $\phi$}%
A differential character $h \in \widehat H^k(X;\Z)$ is called {\em flat along $\phi$} if $\varphi^*\curv(h)=0$.
\index{differential characters!flat along $\phi$}%
\index{flat along $\phi$}%

As is well-known, a $\Ul$-bundle is topologically trivial if and only if it admits sections. 
Topological triviality is detected by the first Chern class.
Thus the pull-back bundle along a smooth map $\varphi$ is topologically trivial if and only if the original bundle admits sections along $\varphi$.
A similar statement holds for differential characters of any degree: 

\begin{cor}[Properties of sections]\label{cor:existence_sections} \index{Corollary!properties of sections}%
A differential character $h \in \widehat H^k(X;\Z)$ admits sections along a smooth map $\phi:A \to X$ if and only if it is topologically trivial along $\varphi$.

If $h$ admits parallel sections along $\phi$, then $h$ is also flat along $\phi$.
Conversely, if $(\curv(h),0) \in \Omega^k_{\phi,0}(X,A)$ and $h$ is topologically trivial along $\phi$, then it admits a parallel section. 

Sections along $\phi$ are uniquely determined by their covariant derivatives if $\phi_*:H_{k-2}(A;\Z) \to H_{k-2}(X;\Z)$ is injective. 
Explicitly, if $f_1,f_2 \in \widehat H^k_\phi(X,A;\Z)$ with $\vds(f_1)=\vds(f_2)$ and $\cov(f_1) = \cov(f_2)$, then $f_1 = f_2$. 
\end{cor}

\begin{proof}
The first statement follows immediately from Theorem~\ref{thm:rel_diff_charact_exact_sequence}.

For the second, let $f \in \widehat H^k_\phi(X,A;\Z)$ with $\vds(f)=h$ and $\cov(f) =0$.
Then $d_\phi(\curv(f),\cov(f))=0$ implies $0=\phi^*\curv(f) - d\cov(f) = \phi^*\curv(h)$.   
Conversely, by surjectivity of the map $(\curv,\cov):\widehat H^k(X,A;\Z) \to \Omega^k_{\phi,0}(X,A)$, we find a parallel section if $(\curv(h),0) \in \Omega^k_{\varphi,0}(X,A)$.
A necessary condition is $\phi^*\curv(h) =0$, but this might not be sufficient. 

To show the last assertion, observe that $\phi_*H_{k-2}(A;\Z) = H_{k-2}(X;\Z)$ implies 
\begin{align*}
\phi^*\widehat H^{k-1}_\mathrm{flat}(X;\Z)
&\cong
\phi^*H^{k-2}(X;\Ul) \\
&=
\phi^*\Hom(H_{k-2}(X;\Z),\Ul) \\
&= 
\Hom(\phi_*H_{k-2}(A;\Z),\Ul) \\
&= 
\Hom(H_{k-2}(A;\Z),\Ul) \\
&= 
H^{k-2}(A;\Ul) \\
&\cong
\widehat H^{k-1}_\mathrm{flat}(A;\Z) \,.
\end{align*}
Now let $f_1,f_2 \in \widehat H^k_\phi(X,A;\Z)$ be sections of $h \in \widehat H^k(X;\Z)$ with $\cov(f_1) = \cov(f_2)$.
By Theorem~\ref{thm:rel_diff_charact_exact_sequence}, we have $f_1 - f_2 = \ti(g)$ for some $g \in \widehat H^{k-1}(A;\Z)$.
Since $\curv(g) = - \cov(f_1-f_2) =0$, we have  $g \in \widehat H^{k-1}_\mathrm{flat}(A;\Z) = \phi^*\widehat H^{k-1}_\mathrm{flat}(X;\Z)$.
Hence $f_1 - f_2 = \ti(g) = 0$ by Theorem~\ref{thm:rel_diff_charact_exact_sequence}.  
\end{proof}

\begin{remark}
Any differential character $h\in \widehat H^k(X;\Z)$ has local sections in the
following sense:
\index{section!local $\sim$}%
\index{differential characters!local sections}%
If $\phi:A\to X$ is smooth where $A$ is contractible, then $H^k(A;\Z)=0$.
Hence $h$ is topologically trivial along $\phi$ and therefore admits sections
along $\phi$.
\end{remark}

\begin{example}\label{ex:CCS}
Let $G$ be a compact Lie group with Lie algebra $\g$.
An invariant polynomial, homogeneous of degree $k$, is a symmetric $\Ad_G$-invariant multilinear map $q:\g^{\otimes k}\to \R$.
The Chern-Weil construction associates to any principal $G$-bundle with connection $(P,\nabla)$ a closed differential form $CW(q) = q(R^\nabla) \in \Omega^{2k}(X)$ by applying the polynomial $q$ to the curvature $2$-form $R^\nabla$ of the connection $\nabla$.
Consider those polynomials $q$ for which the Chern-Weil form $CW(q)$ has integral periods. 
Let $u \in H^{2k}(X;\Z)$ be a universal characteristic class for principal $G$-bundles that coincides in $H^{2k}(X;\R)$ with the de Rham class of $CW(q)$. 
The Cheeger-Simons construction~\cite[Thm~2.2]{CS83} associates to this setting a differential character $\widehat{CW}(q,u) \in \widehat H^{2k}(X;\Z)$ with curvature $\curv(\widehat{CW}(q,u)) = CW(\nabla,q)$, the Chern-Weil form, and characteristic class $c(\widehat{CW}(q,u))=u$, the fixed universal characteristic class.
\index{Cheeger-Simons construction}%
\index{differential characters!Cheeger-Simons construction}%
The construction is natural with respect to bundle maps.

Since the total space $EG$ of the universal principal $G$-bundle is contractible, universal characteristic classes vanish upon pull-back to the total space.
By Theorem~\ref{thm:rel_diff_charact_exact_sequence} the Cheeger-Simons character $\widehat{CW}(q,u)$ thus admits sections along the bundle projection $\pi:P \to X$.
The so-called Cheeger-Chern-Simons construction of \cite{B12} yields a canonical section $\widehat{CCS}(q,u) \in \widehat H^{2k}_\pi(X,P;\Z)$ with covariant derivative $\cov(\widehat{CCS}(q,u)) = CS(q) \in \Omega^{2k-1}(P)$, the Chern-Simons form.
\index{Cheeger-Chern-Simons construction}%
\index{relative differential characters!Cheeger-Chern-Simons construction}%
The construction is natural with respect to bundle maps.
\end{example}

\section{Relative differential characters and fiber integration}

Throughout this section, we consider the case that $A\subset X$ is a smooth subspace and $\phi:A \to X$ the inclusion.
We drop $\phi$ in the notation and write $\widehat H^k(X,A;\Z)$ instead of $\widehat H^k_\phi(X,A;\Z)$.

Let us consider the space $\widehat H^k(X,X;\Z)$ of differential characters with global sections.
Let $(x,y) \in Z_k(X,X;\Z)$.
Then we have $x = -\partial y$, hence $(x,y) = \partial (0,-y)$ and $Z_k(X,X;\Z)=B_k(X,X;\Z)$.
Therefore any relative differential character $f\in\widehat H^k(X,X;\Z)$ is of the form
$$
f(c,-\partial c) = f(\partial(0,c)) = \exp\Big(2\pi i \int_c\cov(f)\Big) .
$$
Conversely, each $(k-1)$-form $\vartheta$ defines a relative differential character in $\widehat H^k(X,X;\Z)$ by
$$
f(c,-\partial c) = \exp\Big(2\pi i \int_c\vartheta\Big) .
$$
Thus $\cov:\widehat H^k(X,X;\Z) \to \Omega^{k-1}(X)$ is an isomorphism.
Moreover, the diagram
$$
\xymatrix{
\widehat H^k(X,X;\Z) \ar[rr]^{\cov}_{\cong} \ar[dr]_{\vds} & & \Omega^{k-1}(X)\ar[ld]^\iota
\\
& \widehat H^k(X;\Z) &\\
}
$$
commutes.

We may now reinterprete fiber integration for fibers $F$ with boundary as follows:
Given $h\in H^k(E;\Z)$, fiber integration along $F$ yields a form \mbox{$\rho = (-1)^{k-\dim F} \fint_F \curv(h) \in\Omega^{k-\dim F}(X)$} in the notation of Proposition~\ref{prop:fiber_int_rand}.
Applying the isomorphism $\cov^{-1}$, we obtain a relative differential character $\widehat \pi_!^Eh\in\widehat H^{k-\dim(F)+1}(X,X;\Z)$ with the property, that $\vds(\widehat\pi_!h)=\iota(\rho)=\widehat\pi_!^{\partial E}h$.
Hence we have constructed a fiber integration map 
\begin{equation}\label{eq:def_fiber_int_relative}
\widehat \pi_!^E = \cov^{-1} \circ \fint_F \circ \, (-1)^{k-\dim F}\curv: 
\quad H^k(E;\Z) \to H^{k-\dim(F)+1}(X,X;\Z)
\end{equation} 
\index{+PihatE@$\widehat \pi_\ausruf^E$, fiber integration for fibers with boundary}%
\index{relative differential characters!obtained from fiber integration}%
\index{fiber integration!for differential characters!for fibers with boundary}%
such that the diagram
\begin{equation}
\xymatrix{
& \widehat H^k(E;\Z) \ar[ld]_{\widehat \pi_!^E} \ar[rd]^{\widehat\pi_!^{\partial E}} & \\
\widehat H^{k-\dim F+1}(X,X;\Z) \ar[rr]_{\vds}&& \widehat H^{k-\dim\partial F}(X;\Z)
}
\label{eq:FI_bound_relative}
\end{equation}
commutes.

\begin{example}\label{ex:partrans}
Let the fibers of $\pi:E\to X$ be diffeomorphic to compact intervals and carry an orientation.
Let $k=2$ and let $P$ be a $\Ul$-bundle with connection $\nabla$ corresponding to $h\in \widehat H^2(E;\Z)$.
In the notation of Examples~\ref{ex:FIk1f1} and \ref{ex:FIk2f1} we have
$$
\widehat\pi_!^{\partial E} P 
= (j^+)^*P\otimes (j^-)^*P^* 
= \Hom((j^-)^*P,(j^+)^*P)
$$
where $\Hom$ stands for (unitary)\footnote{If we regard $\widehat\pi_!^{\partial E} P$ as a $\Ul$-principal bundle, we have to take unitary homomorphisms. If we regard it as a complex line bundle, we have to take all $\C$-linear homomorphisms.} homomorphisms.
Now $\widehat \pi_!^EP\in \widehat H^2(X,X;\Z)$ yields a global section $\sigma$ of $\widehat\pi_!^{\partial E} P= \Hom((j^-)^*P,(j^+)^*P)$, uniquely determined up to multiplication by an element in $\Ul$ over each connected component of $X$.
In the construction of $\sigma$ we choose $\sigma(x_0)\in(\widehat\pi_!^{\partial E} P)_{x_0}= \Hom(P_{j^-(x_0)},P_{j^+(x_0)})$ as the parallel transport in $P$ along the fiber $E_{x_0}$ for some fixed $x_0$.
Then one can check that $\sigma(x)$ is parallel transport in $P$ along the fiber $E_x$ for all $x$ in the connected component of $X$ containing $x_0$.
\index{fiber integration!for differential characters!$1$-dimensional fiber}
\end{example}

%%%%%%%%%%%%%%%%%%%%%%%%%%%%%%%%%%%%%%%%%%%%%%%%%%%%%%%%%%%%%%%%%%%%%%%%%
\chapter{Applications}\label{sec:applications}
%%%%%%%%%%%%%%%%%%%%%%%%%%%%%%%%%%%%%%%%%%%%%%%%%%%%%%%%%%%%%%%%%%%%%%%%%

We will now see how various constructions occurring in different contexts in the literature, such as higher-dimensional holonomy, parallel transport and transgression as well as chain field theories, can be described using the general calculus of absolute and relative differential characters developed in the preceding sections.

\section{Higher dimensional holonomy and parallel transport}\label{subsec:Hol_PT}
In this section, we discuss holonomy and parallel transport of differential characters along compact oriented smooth manifolds.  
Holonomy of smooth Deligne classes has been discussed in \cite[Sec.~3]{CJM04}.
Surface holonomy was considered as classical action for a quantum field theory in \cite{G88,FNSW10}.
An approach to holonomy along surfaces with boundary using $D$-branes is described in \cite[Sec.~6]{FNSW10}. 

For a $\Ul$-bundle with connection $(P,\nabla)$ on $X$, holonomy around a closed loop is defined geometrically by parallel transport.
Parallel transport along a path \mbox{$\gamma:[0,1] \to X$} in the associated complex line bundle takes values in the line $L_{\gamma(0)}^* \otimes L_{\gamma(1)}$.
Holonomy along a closed path $\gamma:[0,1] \to X$ is the element in $\Ul$ that corresponds to the value of the parallel transport in $L_{\gamma(0)}^* \otimes L_{\gamma(0)} \cong \C$.
\index{holonomy!of a line bundle with connection}%
\index{parallel transport!for line bundle with connection}%

\emph{Higher dimensional holonomy.}
In abritrary degree $k$, let $h \in \widehat H^k(X;\Z)$ be a differential character.
In view of Example~\ref{ex:U1Buendel}, we may think of the map $h$ as defining holonomy around orientable closed manifolds of dimension $k-1$.
More explicitly, for a smooth map $\phi:\Sigma \to X$ from an oriented closed $(k-1)$-manifold $\Sigma$, we set
\begin{equation}\label{eq:def_hol}
\hol^h(\phi) := \phi^*h([\Sigma]) = \phi^*h([\Sigma]_{\partial S_k}) = h(\phi_*[\Sigma]_{\partial S_k}). 
\end{equation}
\index{+Holh@$\hol^h$, higher holonomy}%
\index{higher holonomy}%
\index{holonomy!higher $\sim$}%
\index{differential characters!holonomy}
Holonomy is invariant under thin cobordism in the sense of \cite{BTW04}: 
for a cobordism $\Phi:W\to X$ from $\phi:\Sigma \to X$ to $\phi':\Sigma'\to X$, we have $\phi'_*[\Sigma']_{\partial S_k} - \phi_*[\Sigma]_{\partial S_k} = \partial \Phi_*[W]_{S_k}$.
If the cobordism is thin, then $\Phi_* c \in S_k(X;\Z)$ for any fundamental cycle $c$ of $W$.
Thus $\hol^h(\phi') = h(\phi'_*[\Sigma']_{\partial S_k}) = h(\phi_*[\Sigma]_{\partial S_k}) = \hol^h(\phi)$.
\index{holonomy!thin invariance}

Higher dimensional parallel transport will be defined analogously by evaluating differential characters along oriented smooth $(k-1)$-manifolds with boundary.
The result will be an element in a complex line attached to the boundary.
For surfaces such constructions are well known from Chern-Simons theory, see e.g.\ \cite[Sec.~2]{RSW89} and \cite[Sec.~2]{F95}.

\emph{Construction of the line bundle $\LL^h$.}
Let $h \in \widehat H^k(X;\Z)$ and let $W$ be a compact oriented $(k-1)$-manifold $W$ with boundary $\partial W = \Sigma$.
Let $\phi:\Sigma \to X$ be a smooth map which extends to a map defined on $W$.
In other words, it lies in the image of the restriction map $r:C^\infty(W,X) \to C^\infty(\partial W,X)$, $\Phi \mapsto \Phi|_{\partial W}$.   
For a smooth map $\Phi:W \to X$ we set \mbox{$-\Phi:\overline{W} \to X$} for the same map from the manifold with reversed orientation. 
On the set $C^\infty(W,X) \times \C$, we consider the equivalence relation
\begin{equation}\label{eq:eq_rel}
(\Phi,c) \sim (\Phi',c')
:\Leftrightarrow \; \mbox{$r(\Phi) = r(\Phi')$ and $c = \hol^h(\Phi' \cup_\phi -\Phi) \cdot c'$} .
\footnote{In general, $\Phi' \cup_\phi -\Phi$ is not smooth as a map defined on the \emph{manifold} $W\cup_{\partial W}\overline W$ but it defines a smooth singular cycle if the fundamental cycle of $W\cup_{\partial W}\overline W$ is chosen appropriately.}
\end{equation}
For $\phi \in r(C^\infty(W,X))$, this defines a complex line 
\begin{equation}\label{eq:def_LL}
\LL^h_\phi 
:= \{(\Phi,c) \,|\, \Phi \in r^{-1}(\phi), c \in \C\} / \sim \,.
\end{equation}\index{+Lh@$\LL^h$, line bundle for higher parallel transport}%
Varying the map $\phi$, we obtain a complex line bundle $\LL^h \to r(C^\infty(W,X))$.
Holonomy defines a Hermitian metric on $\LL^h$ by
\begin{equation}\label{eq:def_metr}
\langle [\Phi_1,c_1],[\Phi_2,c_2] \rangle := \hol^h(\Phi_1 \cup_\varphi -\Phi_2) \cdot c_1 \cdot \overline{c_2} \,. 
\end{equation}
This is well defined, since for $(\Phi_1,c_1) \sim (\Phi'_1,c'_1)$ we have $c'_1 = \hol^h(\Phi_1 \cup_\phi -\Phi'_1) \cdot c_1$ and thus 
\begin{align*}
\langle [\Phi'_1,c'_1],[\Phi_2,c_2]\rangle 
&= 
\hol^h(\Phi'_1 \cup_\varphi -\Phi_2) \cdot c'_1 \cdot \overline{c_2} \\
&=
\hol^h(\Phi_1 \cup_\phi -\Phi'_1) \cdot \hol^h(\Phi'_1 \cup_\varphi -\Phi_2) \cdot c_1 \cdot \overline{c_2} \\
&=
\hol^h(\Phi_1 \cup_\varphi -\Phi_2) \cdot c_1 \cdot \overline{c_2} \\
&=
\langle [\Phi_1,c_1],[\Phi_2,c_2]\rangle \,. 
\end{align*}

\emph{Higher dimensional parallel transport.}
Parallel transport along $\Phi:W \to X$ is defined by 
\begin{equation}\label{eq:def_PT}
\PT^h(\Phi) 
:= 
[\Phi,1] \in \LL^h_{r(\Phi)} . 
\end{equation}
\index{+PTh@$\PT^h$, higher parallel transport}%
\index{higher parallel transport}%
\index{parallel transport!higher $\sim$}
\index{differential characters!parallel transport}
The map $\PT^h:C^\infty(W,X) \to \LL^h$, $\Phi \mapsto [\Phi,1]$, is a section of $\LL^h$ along the restriction map $r$.
Moreover, $[\Phi,1]$ has unit length.
Thus parallel transport yields a section of the $\Ul$-bundle associated with the Hermitian line bundle $\LL^h$. 

\emph{The connection $\nabla^h$ on $\LL^h$.}
We construct a connection $\nabla^h$ on the bundle $\LL^h$ by describing its parallel transport (not to be confused with the higher dimensional parallel transport constructed above):
Choose a path $\gamma:[0,1] \to r(C^\infty(W,X))$ and a lift $\Gamma:[0,1] \to C^\infty(W,X)$ with $r \circ \Gamma = \gamma$.
Define $F: [0,1] \times W \to X$ by $F(t,w) := \Gamma(t)(w)$.

The path $\Gamma$ yields a lift of the path $\gamma$ to the total space $\LL^h$, defined by 
$$
\overline\Gamma:[0,1] \to \LL^h, \quad\overline\Gamma(t):= [\Gamma(t),1].
$$
We define parallel transport along the path $\gamma$ to be the homomorphism
\begin{equation}\label{eq:def_nabla_PT}
\PP^{\nabla^h}_\gamma: \LL^h_{\gamma(0)} \to \LL^h_{\gamma(t)},\quad
 \overline\Gamma(0) = [\Gamma(0),1] \mapsto \exp \Big( - 2\pi i \int_{[0,t] \times W} F^*\curv(h) \Big) \cdot [\Gamma(t),1] .   
\end{equation}

\emph{Identification of the holonomy of $\nabla^h$.}
Now we compute the holonomy of this connection.
Let $\gamma:[0,1] \to r(C^\infty(W,X))$ be a closed curve, i.e.~$\gamma(0) = \gamma(1)=\phi \in r(C^\infty(W,X))$.
Then the lift $\Gamma:[0,1] \to C^\infty(W,X)$ need not be closed.
But for any $w \in \partial W$, we have: 
$$
F(0,w) 
= 
\Gamma(0)(w) 
= 
\gamma(0)(w) 
=
\phi(w)
=
\gamma(1)(w)
=
F(1,w).
$$
Hence $F|_{[0,1] \times \partial W}$ descends to a map $f:S^1 \times \partial W \to X$.
 
By definition, holonomy along $\gamma$ in the bundle $(\LL^h,\nabla^h)$ is the complex number $\hol^{\nabla^h}(\gamma) \in \C^*$ defined by 
\begin{align*}
\PP^{\nabla^h}_\gamma (\overline\Gamma(0)) 
&= \hol^{\nabla^h}(\gamma) \cdot \overline\Gamma(0) \,. 
\end{align*}%
By \eqref{eq:eq_rel}, we may write 
$$
\overline\Gamma(0) 
= 
[\Gamma(0),1] 
= 
\big( \hol^h(\Gamma(1) \cup_\phi -\Gamma(0)) \big)^{-1} \cdot [\Gamma(1),1] 
= 
\big( \hol^h(\Gamma(1) \cup_\phi -\Gamma(0)) \big)^{-1} \cdot \overline\Gamma(1).
$$
Thus we obtain for the parallel transport along the closed curve $\gamma$:
\begin{align*}
\PP^{\nabla^h}_\gamma(\overline\Gamma(0))
&= 
\exp \Big( -2\pi i \int_{[0,1] \times W} F^*\curv(h) \Big) \cdot \overline\Gamma(1) \\
&= 
\exp \Big( -2\pi i \int_{F_*([0,1]\times W)} \curv(h) \Big) \cdot \overline\Gamma(1) \\
&= 
h \big(-\partial F_*([0,1]\times W)\big) \cdot \overline\Gamma(1) \\
&= 
h \big( F_*([0,1] \times \partial W) - F_*(\{1\} \times W \sqcup \{0\} \times \overline W) \big) \cdot \overline\Gamma(1) \\
&= 
h \big( F_*([0,1] \times \partial W)\big) \cdot h \big(\Gamma(1)_*W - \Gamma(0)_* W)\big)^{-1} \cdot \overline\Gamma(1) \\
&= 
f^*h([S^1 \times \partial W]) \cdot \hol^h(\Gamma(1) \cup_\phi -\Gamma(0))^{-1} \cdot \overline\Gamma(1) \\
&=
f^*h([S^1 \times \partial W]) \cdot \hol^h(\Gamma(1) \cup_\phi -\Gamma(0))^{-1} \cdot \overline\Gamma(1) \\
&=
\hol^h(f) \cdot \overline\Gamma(0) \,.
\end{align*}
Consequently, 
\begin{equation}\label{eq:hol_LL}
\hol^{\nabla^h}(\gamma) = \hol^h(f) \in \Ul \subset \C^* \,. 
\end{equation}
Thus the holonomy of $\nabla^h$ along the path $\gamma$ coincides with the higher dimensional holonomy along the map $f:S^1 \times \partial W \to X$.

In particular, we have defined a unitary connection $\nabla^h$ on the line bundle \mbox{$\LL^h \to r(C^\infty(W,X))$} with holonomy given by the holonomy of the differential character $h \in \widehat H^k(X;\Z)$. 

\emph{Computation of the connection $1$-form.}
The bundle $\LL^h \to r(C^\infty(W,X))$ with connection $\nabla^h$ and section $\PT^h$ along the restriction map \mbox{$r:C^\infty(W,X) \to C^\infty(\partial W,X)$} yields a relative differential character $[\LL^h,\nabla^h,\PT^h] \in \widehat H^2_r(r(C^\infty(W,X)),C^\infty(W,X);\Z)$.
To complete the picture of the equivalence class $[\LL^h,\nabla^h,\PT^h]$ as a relative differential character, it remains to 
compute the $1$-form $\cov([\LL^h,\nabla^h,\PT^h]) \in \Omega^1(C^\infty(W,X))$.
By Example~\ref{ex:H2rel}, this corresponds to the connection $1$-form of $r^*\nabla^h$ with respect to the section $\PT^h$.
We now compute this $1$-form.

Let $\Gamma:[0,1] \to C^\infty(W,X)$ be a path as above and $\overline\Gamma$ the corresponding lift of the path $\gamma = r \circ \Gamma$ to the total space $r^*\LL^h$.
The connection $1$-form $\vartheta^{r^*\nabla^h}$ of $r^*\nabla^h$ 
is determined by parallel transport along the path $\Gamma$ through the equation
$$
\PP^{r^*\nabla^h}_\gamma: r^*\LL^h_{\gamma(0)} \to r^*\LL^h_{\gamma(t)},\quad
\overline{\Gamma}(0)
\mapsto
\exp \Big( - \int_0^t \vartheta^{r^*\nabla^h}(\overline{\Gamma}')(s) ds \Big)  \cdot \overline{\Gamma}(t)\,.
$$
Comparing with \eqref{eq:def_nabla_PT}, we obtain
\begin{align}
\exp \Big( - \int_0^t (\vartheta^{r^*\nabla^h}(\overline{\Gamma}'))(s) \, ds \Big)
&=
\exp \Big( -2\pi i \int_{[0,t] \times W} F^*\curv(h) \Big) \notag \\
&=
\exp \Big( -2\pi i \int_{[0,t]} \fint_W F^*\curv(h) \Big) \notag \\
&=
\exp \Big( -2\pi i \int_{[0,t]} \Gamma^*\Big(\fint_W \ev_W^*\curv(h)\Big) \Big) \,. \label{eq:nabla_h_form}
\end{align}
Here $\ev_W:C^\infty(W,X) \times W \to X$, $(\Phi,w) \mapsto \Phi(w)$, denotes the evaluation map and $\fint_W$ the fiber integration in the trivial bundle $[0,1] \times W \to [0,1]$. 
Since \eqref{eq:nabla_h_form} holds for any $t \in [0,1]$, we have $\vartheta^{r^*\nabla^h}(\overline{\Gamma}')(s) =  2\pi i \cdot  \Gamma^*(\fint_W \ev_W^*\curv(h))_{s}(\tfrac{\partial\;}{\partial s})$.
This determines the connection $1$-form of $r^*\nabla^h$ with respect to the section $\overline{\Gamma}$ along the path $\gamma: [0,1] \to r(C^\infty(W,X))$.
By Example~\ref{ex:H2rel}, we conclude
\begin{equation}\label{eq:covLL}
\cov([\LL^h,\nabla^h,\PT^h]) 
=
- \fint_W \ev_W^*\curv(h) \,.
\end{equation}

The transgression maps defined in the following sections use fiber integration to generalize the construction of the line bundle with connection $(\LL^h,\nabla^h)$ and the section $\PT^h$ along the restriction map. 

\section{Higher dimensional transgression}\label{subsec:transgression}
In this section, we define transgression of differential characters of arbitrary degree along oriented closed manifolds.
The classical case studied in the literature is transgression along $S^1$ for degree-$2$ and degree-$3$ differential cohomology.
Our construction generalizes these classical cases to transgression along oriented closed manifolds of arbitrary finite dimension.
It turns out that the holonomy defined in Section~\ref{subsec:Hol_PT} is a special case of this transgression.

Let $\Sigma$ be a compact smooth manifold without boundary, and let $X$ be any smooth manifold.
Then the space $C^\infty(\Sigma,X)$ of smooth maps from $\Sigma$ to $X$ is again a smooth space as explained in Section~\ref{sec:smoothspaces}.
The best-known space of this type is the {\em free loop space} $\LL(X) := C^\infty(S^1,X)$ of smooth maps from the circle $S^1$ to $X$.
\index{+LlX@$\LL(X)$, free loop space of $X$}
\index{loop space}%

The evaluation map $\ev_\Sigma$ is defined in the obvious way:
$$
\ev_\Sigma: C^\infty(\Sigma,X) \times \Sigma \to X ,\quad (\phi,s) \mapsto \phi(s) .
$$
\index{+EvSigma@$\ev_\Sigma$, evaluation map}%
\index{evaluation map}%
We consider the pull-back $\ev_\Sigma^*: \widehat H^{k}(X;\Z) \to \widehat H^{k}(C^\infty(\Sigma,X) \times \Sigma;\Z)$.
If $\Sigma$ is oriented, then we can integrate differential characters in $\widehat H^k(C^\infty(\Sigma,X) \times \Sigma;\Z)$ over the fiber of the trivial fiber bundle $C^\infty(\Sigma,X) \times \Sigma \stackrel{\pi}{\twoheadrightarrow} C^\infty(\Sigma,X)$. 

\begin{definition}[Transgression along closed manifold] \index{Definition!transgression along closed manifold}%
Let $\Sigma$ be a compact oriented smooth manifold without boundary, and let $X$ be any smooth manifold. 
\emph{Transgression along $\Sigma$} is the map 
\begin{equation}\label{eq:def_trans}
\tau_\Sigma: \widehat H^*(X;\Z) \to \widehat H^{*-\dim \Sigma}(C^\infty(\Sigma,X);\Z)\,,\quad 
h \mapsto \widehat\pi_! (\ev_\Sigma^*h) \,.
\end{equation}
\index{+TauSigma@$\tau_\Sigma$, transgression along $\Sigma$}
\index{transgression!along closed manifold}%
\index{differential characters!transgression along closed manifold}%
In particular, for $\Sigma=S^1$ we have
$$
\tau_{S^1}: \widehat H^*(X;\Z) \to \widehat H^{*-1}(\LL(X);\Z)  \,,\quad 
h \mapsto \widehat\pi_! (\ev_{S^1}^*h) \,.
$$
\end{definition}

\begin{example}\label{ex:trans_k23}
\index{transgression!along $S^1$}%
For $k=2$, transgression along $\Sigma = S^1$ associates to a $\Ul$-bundle on $X$ its holonomy map $\LL(X)\to \Ul$.

For $k=3$, transgression along $\Sigma = S^1$ has been discussed in quantum field theory to construct the anomaly bundle over loop space \cite{G88,B93}.
In this case, the image of the transgression map has been characterized \cite{W1,W2}.
\end{example}

\begin{example}\label{ex:trans_k-1}
Let $h \in \widehat H^k(X;\Z)$.
Let $\dim \Sigma = k-1$.
Transgression along $\Sigma$ yields a differential character $\tau_\Sigma h \in \widehat H^1(C^\infty(\Sigma,X);\Z)$, which by Example~\ref{ex:H1} corresponds to $\Ul$-valued function on the mapping space $C^\infty(\Sigma,X)$.
We verify that this function coincides with holonomy of $h$ as defined in Section~\ref{subsec:Hol_PT}.
For any fixed $\phi \in C^\infty(\Sigma,X)$, we have the pull-back diagram:
\begin{equation*}
\xymatrix{
\{\phi\} \times \Sigma \ar[d]_{\widehat\pi_!} \ar[r]^{\imath_\phi \times \id} \ar`u`[rr]^{\widetilde\phi}[rr] & C^\infty(\Sigma,X) \times \Sigma \ar[d]_{\widehat\pi_!} \ar[r]^(0.7){\ev_\Sigma} & X \\
\{\phi\} \ar[r]^{\imath_\phi} & C^\infty(\Sigma,X)\,. & \\
}
\end{equation*}
Thus by naturality of fiber integration, we have:
\begin{align}
\hol^h(\phi)\,\,
&\ist{\eqref{eq:def_hol}}\,\,
(\phi^*h)([\Sigma]) \notag \\
&= 
(\widehat\pi_!\widetilde\phi^*h) \notag \\
&=
(\widehat\pi_!(\imath_\phi \times \id)^*\ev_\Sigma^*h)) \notag \\
&\ist{\eqref{eq:fiber_int_natural}}
\imath_\phi^*(\underbrace{\widehat\pi_!\ev_\Sigma^*h}_{=\tau_\Sigma h}) \notag \\
&\ist{\eqref{eq:def_trans}}\,\,
(\tau_\Sigma h)(\phi) \,. \label{eq:trans_hol}
\end{align}
\end{example}

To evaluate the $\Ul$-valued function $\tau_\Sigma h$ on the map $\phi$, we could have used Definition~\ref{def:fiber_int_diff_charact_construction} instead of the pull-back diagram. 
But the argument above can be generalized to compute the holonomy of the character $\tau_\Sigma h$ for transgression of any degree:
Let $h \in \widehat H^k(X;\Z)$, and let $\Sigma_2$ be an oriented closed manifold.
Let $\Sigma_1$ be an oriented closed manifold of dimension $\dim(\Sigma_1) = k-\dim(\Sigma_2)-1$, and let $\phi:\Sigma_1 \to C^\infty(\Sigma_2,X)$ be a smooth map.
By \eqref{eq:trans_hol} we have:
\begin{equation}\label{eq:hol_tau}
\hol^{\tau_{\Sigma_2}h}(\phi) 
=
(\tau_{\Sigma_1}(\tau_{\Sigma_2}h))(\phi). 
\end{equation}
\index{holonomy!of transgressed character}%
We generalize this equation, replacing holonomy by transgression: 
Let $\Sigma_1$ and $\Sigma_2$ be compact oriented smooth manifolds without boundary.
The evaluation in the first entry yields a canonical identification
\begin{equation*}
\ev_1: C^\infty(\Sigma_1 \times \Sigma_2,X) \xrightarrow{\cong} C^\infty(\Sigma_1,C^\infty(\Sigma_2,X)), \quad
f \mapsto (t \mapsto f(t,\cdot)). 
\end{equation*}
\index{+Evone@$\ev_1$, evaluation map}%
\index{evaluation map}%
Using functoriality and naturality of fiber integration, we conclude that higher dimensional transgression is functorial and graded commutative:

\begin{prop}[Functoriality of transgression] \index{Proposition!functoriality of transgression}%
Let $\Sigma_1$ and $\Sigma_2$ be compact oriented smooth manifolds without boundary.
Let $h \in \widehat H^k(X;\Z)$.
Then we have: 
\begin{equation}\label{eq:tautau}
\tau_{\Sigma_1 \times \Sigma_2} h 
=
\ev_1^*(\tau_{\Sigma_1} \circ \tau_{\Sigma_2})h 
= 
(-1)^{\dim \Sigma_1 \cdot \dim \Sigma_2} \tau_{\Sigma_2 \times \Sigma_1} h .
\end{equation}
\index{transgression!functoriality}%
\index{functoriality!of transgression}%
\end{prop}

\begin{proof}
The canonical diffeomorphism $\Sigma_1 \times \Sigma_2 \xrightarrow{\cong} \Sigma_2 \times \Sigma_1$ yields a canonical identification $C^\infty(\Sigma_2 \times \Sigma_1,X) \xrightarrow{\cong} C^\infty(\Sigma_1 \times \Sigma_2,X)$.
The fiber orientation in the trivial fiber bundles with fiber $\Sigma_2 \times \Sigma_1$ is $(-1)^{\dim \Sigma_1 \cdot \dim \Sigma_2}$ times the one in the bundles with fiber $\Sigma_1 \times \Sigma_2$.
According to Proposition~\ref{prop:orient_reversal}, we obtain $\tau_{\Sigma_1 \times \Sigma_2} h = (-1)^{\dim \Sigma_1 \cdot \dim \Sigma_2} \tau_{\Sigma_2 \times \Sigma_1}$. 

The evaluation maps fit into the commutative diagram:
\begin{equation*}
\xymatrix{
C^\infty(\Sigma_1 \times \Sigma_2,X) \times (\Sigma_1 \times \Sigma_2) \ar[r]^(0.48){\ev_1 \times \id} \ar[d]^{\pi^{\Sigma_2}} 
\ar`u`[rrr]^{\ev_{\Sigma_1\times \Sigma_2}}[rrr] 
&  (C^\infty(\Sigma_1,C^\infty(\Sigma_2,X)) \times \Sigma_1) \times \Sigma_2 \ar[d]^{\pi^{\Sigma_2}} \ar[r]^(0.62){\ev_{\Sigma_1}\times \id} & C^\infty(\Sigma_2,X) \times \Sigma_2 \ar[r]^(0.65){\ev_{\Sigma_2}} \ar[d]^{\pi^{\Sigma_2}} & X \\
C^\infty(\Sigma_1 \times \Sigma_2,X) \times \Sigma_1 \ar[r]^(0.48){\ev_1\times \id} \ar[d]^{\pi^{\Sigma_1}} & C^\infty(\Sigma_1,C^\infty(\Sigma_2,X)) \times \Sigma_1 \ar[r]^(0.62){\ev_{\Sigma_1}} \ar[d]^{\pi^{\Sigma_1}} & C^\infty(\Sigma_2,X)  & \\
C^\infty(\Sigma_1 \times \Sigma_2,X) \ar[r]^{\ev_1} & C^\infty(\Sigma_1,C^\infty(\Sigma_2,X)) \,. && \\
}
\end{equation*}
Here $\pi^{\Sigma_i}$ denote the various projections in trivial bundles with fiber $\Sigma_i$ and $\pi^{\Sigma_1 \times \Sigma_2} = \pi^{\Sigma_1} \circ \pi^{\Sigma_2}$ denotes the projection in the trivial bundle with fiber $\Sigma_1 \times \Sigma_2$.  
We decompose $\ev_{\Sigma_1 \times \Sigma_2}$ as in the top row of the above diagram:
$$\ev_{\Sigma_1 \times \Sigma_2} = \ev_{\Sigma_2} \circ (\ev_{\Sigma_1} \times \id_{\Sigma_2}) \circ (\ev_1 \times \id_{\Sigma_1 \times \Sigma_2})\,.
$$
Using naturality of fiber integration, we obtain:
\begin{align*}
\tau_{\Sigma_1 \times \Sigma_2}
&= 
\widehat\pi^{\Sigma_1 \times \Sigma_2}_! (\ev_{\Sigma_1 \times \Sigma_2}^*h) \\
&\ist{\eqref{eq:fiber_int_funct}}
\widehat\pi^{\Sigma_1}_! \pi^{\Sigma_2}_! (\ev_1 \times \id_{\Sigma_1 \times \Sigma_2})^* (\ev_{\Sigma_1} \times \id_{\Sigma_2})^* \ev_{\Sigma_2}^*h \\
&\ist{\eqref{eq:fiber_int_natural}}
\widehat\pi^{\Sigma_1}_! (\ev_1 \times \id_{\Sigma_1})^* \widehat\pi^{\Sigma_2}_! (\ev_{\Sigma_1} \times \id_{\Sigma_2})^* \ev_{\Sigma_2}^*h \\
&\ist{\eqref{eq:fiber_int_natural}}
\ev_1^* \underbrace{(\widehat\pi^{\Sigma_1}_! \circ \ev_{\Sigma_1}^*)}_{=\tau_{\Sigma_1}} \underbrace{(\widehat\pi^{\Sigma_2}_! \circ \ev_{\Sigma_2}^*)}_{=\tau_{\Sigma_2}}h \\
&= 
\ev_1^* (\tau_{\Sigma_1} \circ \tau_{\Sigma_2})h .
\qedhere
\end{align*}
\end{proof}

Holonomy of differential characters is additive with respect to topological sums (i.e.~disjoint union of oriented closed manifolds): for $h \in \widehat H^k(X;\Z)$ and $\phi:\Sigma_1 \sqcup \Sigma_2 \to X$, we have 
$$
\hol^h(\phi) 
= 
h(\phi_*[\Sigma_1 \sqcup \Sigma_2]_{\partial S_k})
=
h({\phi_1}_*[\Sigma_1]_{\partial S_k} + {\phi_2}_*[\Sigma_2]_{\partial S_k})
=
\hol^h(\phi_1) \cdot \hol^h(\phi_2).
$$
Here $\phi_i$ denotes the restriction of $\phi:\Sigma_1 \sqcup \Sigma_2 \to X$ to $\Sigma_i$ for $i=1,2$.

Likewise, transgression along oriented closed manifolds is additive with respect to topological sums: 
Denote by $r_i: C^\infty(\Sigma_1 \sqcup \Sigma_2,X) \to C^\infty(\Sigma_i;X)$, $\phi \mapsto \phi_i$, $i=1,2$, the restriction maps.
Then we have:

\begin{prop}[Additivity of transgression] \index{Proposition!additivity of transgression}%
Let $\Sigma_1$ and $\Sigma_2$ be oriented closed manifolds.
Let $h \in \widehat H^k(X;\Z)$.
Then we have: 
\index{transgression!additivity}%
\begin{equation}\label{eq:tautau2}
\tau_{\Sigma_1 \sqcup \Sigma_2} h 
=
r_1^*(\tau_{\Sigma_1}h) + r_2^*(\tau_{\Sigma_2}h) \,. 
\end{equation}
\end{prop}

\begin{proof}
For $i=1,2$ set: 
\begin{align*}
E   &:= C^\infty(\Sigma_1 \sqcup \Sigma_2,X) \times (\Sigma_1 \sqcup \Sigma_2) \\
D_i &:= C^\infty(\Sigma_1 \sqcup \Sigma_2,X) \times \Sigma_i \\
E_i &:= C^\infty(\Sigma_i,X) \times \Sigma_i \,.
\end{align*}
The canonical inclusions $\Sigma_i \hookrightarrow (\Sigma_1 \sqcup \Sigma_2)$ yield inclusions $j_i: D_i \hookrightarrow E$.
From the restriction maps $r_i$ and the evaluation maps, we obtain the commutative diagram:
\begin{equation}\label{eq:EiE_ev}
\xymatrix{
& E \ar[dr]^{\ev_{\Sigma_1 \sqcup \Sigma_2}} & \\
D_i \ar[ur]^{j_i} \ar[dr]_{r_i \times \id_{\Sigma_i}} && X \\
& E_i \ar[ur]_{\ev_{\Sigma_i}}
}
\end{equation}
Let $z \in Z_{k-1-\dim \Sigma_i}(C^\infty(\Sigma_1 \sqcup \Sigma_2,X);\Z)$.
Choose $\zeta(z) \in \ZZ_{k-1-\dim \Sigma_i}(C^\infty(\Sigma_1 \sqcup \Sigma_2,X))$ and $a(z) \in C_{k-\dim \Sigma_i}(C^\infty(\Sigma_1 \sqcup \Sigma_2,X);\Z)$ such that $[z-\partial a(z)]_{\partial S_{k-\dim \Sigma_i}} = [\zeta(z)]_{\partial S_{k-\dim \Sigma_i}}$.
Moreover, choose $\zeta({r_i}_*z) = {r_i}_*\zeta(z)$ and $a({r_i}_*z) = {r_i}_*a(z)$.
Then we have:
\begin{align*}
[\PBE\zeta(z)]_{\partial S_k}
&= {j_1}_*[\PB_{D_1}\zeta(z)]_{\partial S_k} + {j_2}_*[\PB_{D_2}\zeta(z)]_{\partial S_k} \,. \notag
\end{align*}
Applying the evaluation map $\ev_{\Sigma_1 \sqcup \Sigma_2}$ and using \eqref{eq:EiE_ev}, we obtain:
\begin{align}
(\ev_{\Sigma_1 \sqcup \Sigma_2})_*[\PBE\zeta(z)]_{\partial S_k}
&= 
(\ev_{\Sigma_1 \sqcup \Sigma_2})_*\Big( \sum_{i=1}^2 \,{j_i}_*[\PB_{D_i}\zeta(z)]_{\partial S_k}\,\Big)  \notag \\
&\ist{\eqref{eq:EiE_ev}}\,\,\,\,
\sum_{i=1}^2 {\ev_{\Sigma_i}}_*((r_i \times \id_{\Sigma_i})_*[\PB_{D_i}\zeta(z)]_{\partial S_k})  \notag \\
&= 
\sum_{i=1}^2 {\ev_{\Sigma_i}}_*[\PB_{E_i}\zeta({r_i}_*z)]_{\partial S_k}  \,.\label{eq:psievSi}
\end{align}
In the last equality we have used \eqref{eq:PBnatuerlich} for the pull-back diagram:
\begin{equation*}
\xymatrix{
D_i \ar[d] \ar[rr]^{\hspace{2em}r_i \times \id_{\Sigma_i}} && E_i \ar[d] \\
C^\infty(\Sigma_1 \sqcup \Sigma_2,X)  \ar[rr]_{r_i} && C^\infty(\Sigma_i,X) \,.
}
\end{equation*}
Now we compute the transgression along $\Sigma_1 \sqcup \Sigma_2$:
\begin{align*}
(\tau_{\Sigma_1 \sqcup \Sigma_2}h)(z)
&=
(\widehat\pi_! (\ev_{\Sigma_1 \sqcup \Sigma_2})^*h)(z) \\
&\ist{\eqref{eq:def_fiber_int}}
(\ev_{\Sigma_1 \sqcup \Sigma_2})^*h ([\PBE\zeta(z)]_{\partial S_k}) \cdot \exp \Big( 2\pi i \int_{a(z)} \fint_{\Sigma_1\sqcup \Sigma_2} \curv(\ev_{\Sigma_1 \sqcup \Sigma_2}^*h) \Big) \\
&= 
h\Big((\ev_{\Sigma_1 \sqcup \Sigma_2})_*[\PBE\zeta(z)]_{\partial S_k}\Big) 
 \cdot \exp \Big( 2\pi i \int_{a(z)} \sum_{i=1}^2 \fint_{\Sigma_i} j_i^* (\ev_{\Sigma_1 \sqcup \Sigma_2})^* \curv(h) \Big)  \\
&\ist{\eqref{eq:psievSi},\eqref{eq:EiE_ev}}\,\,\,\,\,\,\,\,\,\,\,\,
h\Big( \sum_{i=1}^2 (\ev_{\Sigma_i})_*[\PB_{E_i}\zeta({r_i}_*z)]_{\partial S_k} \Big) \\
& \qquad \qquad \qquad \qquad \cdot \exp \Big( 2\pi i \int_{a(z)} \sum_{i=1}^2 \fint_{\Sigma_i} (r_i \times \id_{\Sigma_i})^* (\ev_{\Sigma_i})^* \curv(h) \Big)\\
&\ist{\eqref{eq:fiber_int_natural}}\,\,
h\Big( \sum_{i=1}^2 (\ev_{\Sigma_i})_*[\PB_{E_i}\zeta({r_i}_*z)]_{\partial S_k} \Big) 
 \cdot \exp \Big( 2\pi i \sum_{i=1}^2 \int_{a({r_i}_*z)} \fint_{\Sigma_i} (\ev_{\Sigma_i})^* \curv(h) \Big)  \\
&= 
\tau_{\Sigma_1}h({r_1}_*z) + \tau_{\Sigma_2}h({r_2}_*z)  \\
&= 
(r_1^* \tau_{\Sigma_1}h + r_2^* \tau_{\Sigma_2}h)(z) \,.  \qedhere
\end{align*}
\end{proof}

\begin{example}\label{ex:trans_k-2}
Let $W$ be a compact oriented closed $(k-1)$-manifold with boundary $\partial W = \Sigma$.
In Section~\ref{subsec:Hol_PT}, we have constructed a Hermitian line bundle with unitary connection $(\LL^h,\nabla^h)$ on $r(C^\infty(W,X))$.
By Example~\ref{ex:U1Buendel}, this corresponds to a degree-$2$ differential character on $r(C^\infty(W,X)$.
To show that $[\LL^h,\nabla^h]= \tau_{S}h \in \widehat H^2((r(C^\infty(W,X));\Z)$ it suffices to compare the holonomies, since  holonomy classifies line bundles with connection up to isomorphism.\footnote{Here we do not distinguish notationally between $\tau_\Sigma h$ as differential character on $C^\infty(S,X)$ and its restriction to $r(C^\infty(W,X)) \subset C^\infty(S,X)$.}
Let $\gamma:S^1 \to r(C^\infty(W,X))$ be a closed path, and let $f:S^1 \times \Sigma \to X$ be the induced map as in Section~\ref{subsec:Hol_PT}.
We then have:
$$
\hol^{\nabla^h}(\gamma) 
\stackrel{\eqref{eq:hol_LL}}{=} 
\hol^h(f)
\stackrel{\eqref{eq:trans_hol}}{=}
(\tau_{S^1 \times \Sigma}h)(f)
\stackrel{\eqref{eq:tautau}}{=}
\tau_{S^1}(\tau_\Sigma h)(\gamma)
\stackrel{\eqref{eq:hol_tau}}{=}
\hol^{\tau_\Sigma h}(\gamma) \,.
$$
Thus $[\LL^h,\nabla^h]= \tau_{\Sigma}h$.
\end{example}

\begin{remark}
\emph{Transgression and Topological Quantum Field Theories}.
\index{topological quantum field theories}
\index{field theory!topological quantum $\sim$}
Topological quantum field theories in the sense of Atiyah \cite{A89} are symmetric monoidal functors from a cobordism category to the category of complex vector spaces.
In particular, they associate to topological sums of closed oriented manifolds the tensor products of the vector spaces associated to the summands.
Transgression of differential characters has similar functorial properties in the sense that it is additive with respect to topological sums.

Topological quantum field theories associate to an oriented compact manifold with boundary an element in the vector space associated to the boundary.
Similarly, transgression along oriented manifolds with boundary yields a section along the restriction map of the differential character obtained by transgression along the boundary. 
Transgression along manifolds with boundary will be constructed in the following section.
\end{remark}

\section{Transgression along manifolds with boundary}\label{subsec:transgression_bound}
Let $W$ be a compact oriented smooth manifold with boundary $\partial W$.
Restriction to the boundary defines a map $r:C^\infty(W,X) \to C^\infty(\partial W,X)$, $r(\phi) = \phi|_{\partial W}$.
We consider the trivial bundles
\begin{align*}
E &= C^\infty(W,X) \times W \to C^\infty(W,X) \,,\\
\partial E &= C^\infty(W,X) \times \partial W \to C^\infty(W,X) \,
\end{align*}
and the evaluation map
\begin{align*}
\ev_W : C^\infty(W,X) \times W &\to X, \quad (\phi,w) \mapsto \phi(w) \,.
\end{align*}
In analogy to the transgression along oriented closed manifolds, we define:

\begin{definition}[Transgression along manifold with boundary] \index{Definition!transgression along manifold with boundary}%
Let $W$ be a compact oriented smooth manifold with boundary $\partial W$ and let $X$ be a smooth manifold.
Fiber integration for fibers with boundary yields the following two \emph{transgression maps} along $W$ and $\partial W$:   
\index{transgression!along manifold with boundary}%
\index{differential characters!transgression along manifold with boundary}
\begin{align*}
\tau^E: \widehat H^k(X;\Z)
&\to
\widehat H^{k-\dim W+1}(C^\infty(W,X),C^\infty(W,X);\Z) 
\,, \\
h
&\mapsto
\widehat \pi^E_!\ev_W^*h \,, \\
\tau^{\partial E}: \widehat H^k(X;\Z)
&\to
\widehat H^{k-\dim\partial W}(C^\infty(W,X);\Z)\,, \\
h
&\mapsto
\widehat \pi^{\partial E}_!\ev_W^*h \,.
\end{align*}
\index{+TauE@$\tau^E$, transgression along manifolds with boundary}
\index{+TaudelE@$\tau^{\partial E}$, transgression along the boundary}
\end{definition}

\begin{example}
We consider the special case $k=2$ and $W=I=[0,1]$.
\index{transgression!along interval $I$}%
The space $\PP(X):=C^\infty(I,X)$ is called the {\em path space} of $X$.
\index{+PpX@$\PP(X)$, path space of $X$}%
\index{path space}%
In this case, the trivial bundle $\partial E= C^\infty(I,X) \times \partial I = \PP(X) \times \{0,1\} \to \PP(X)$ is a twofold covering.
By Example~\ref{ex:U1Buendel}, any differential character $h \in \widehat H^2(X;\Z)$ corresponds to (the isomorphism class of) a $\Ul$-bundle with connection $(P,\nabla)$ on $X$.
Transgression along $\partial I$ yields a $\Ul$-bundle with connection $\tau^{\partial E}(P,\nabla)$ on the path space $\PP(X)$.
Its fiber over a path $\gamma \in \PP(X)$ is given by $P_{\gamma(0)}^* \otimes P_{\gamma(1)}$.
Transgression along $I$ yields a section $\sigma$ of this bundle along the restriction map $r: \PP(X) \to C^\infty(\{0,1\},X) = X \times X$, $\gamma \mapsto (\gamma(0),\gamma(1))$.
As we have seen in Example~\ref{ex:partrans}, $\sigma(\gamma)$ can be chosen to be the parallel transporter along $\gamma\in\PP(X)$.
\end{example}

In the following we consider the relations between the three transgression maps $\tau_{\partial W}$, $\tau^{\partial E}$ and $\tau^E$.  
We first note that $\tau^Eh$ is a section of $\tau^{\partial E}h$:
\begin{equation}
\vds(\tau^E h)
= \vds(\widehat \pi^E_! \ev_W^*h) 
\stackrel{\eqref{eq:FI_bound_relative}}{=} \widehat\pi^{\partial E}_! \ev_W^*h
= \tau^{\partial E} h \,.
\label{eq:tauEdelE}
\end{equation}
We note further that $\tau^{\partial E}$ is not the same as $\tau_{\partial W}$ defined in Section~\ref{subsec:transgression} (with $\Sigma=\partial W$) since the former takes values in differential characters on $C^\infty(W,X)$ rather than on $C^\infty(\partial W,X)$.
But they are related by the restriction map $r:C^\infty(W,X) \to C^\infty(\partial W,X)$, $\phi \mapsto \phi|_{\partial W}$.
\index{restriction map}%
We have the pull-back diagram:
\begin{equation}\label{eq:Rr}
\xymatrix{
\partial E = C^\infty(W,X) \times \partial W \ar[r]^(0.53){R} \ar[d] & C^\infty(\partial W,X) \times \partial W \ar[d]\\
C^\infty(W,X) \ar[r]^(0.53){r} & C^\infty(\partial W,X) \,.
}
\end{equation}
By \eqref{eq:fiber_int_natural}, fiber integration is natural with respect to pull-back along smooth maps, hence $\widehat \pi^{\partial E}_! \circ R^* = r^* \circ \widehat \pi^E_!$.
This yields 
\begin{equation}\label{eq:taudelWdelE}
\tau^{\partial E}h
=
\widehat\pi_!^{\partial E}(\ev_W^*h)
=
\widehat\pi_!^{\partial E}(R^* \ev_{\partial W}^*h)
=
r^*\widehat\pi_!(\ev_{\partial W}^*h)
=
r^*(\tau_{\partial W} h) \,.
\end{equation}
Thus the three transgression maps fit into the following commutative diagram:
\begin{equation}\label{eq:diag_3_tau}
\xymatrix{
&&& \widehat H^{k-\dim W+1}(C^\infty(W,X),C^\infty(W,X);\Z) \ar_{\qquad \vds}[d] \\
\widehat H^k(X;\Z)  
\ar[rrru]^{\tau^E} \ar[rrr]^{\tau^{\partial E}} \ar[rrrd]_{\tau_{\partial W}}
&&& \widehat H^{k-\dim\partial W}(C^\infty(W,X);\Z) \\
&&& \widehat H^{k-\dim \partial W}(C^\infty(\partial W,X);\Z) \ar[u]^{r^*}
} 
\end{equation}
In particular, $\tau_{\partial W}h \in \widehat H^{k-\dim \partial W}(C^\infty(\partial W,X);\Z)$ is topologically trivial along $r$: for the pull-back along $r$ of the characteristic class, we find: 
\[
r^*c(\tau_{\partial W}h)
= c(r^*\tau_{\partial W}h) 
\stackrel{\eqref{eq:taudelWdelE}}{=} c(\tau^{\partial E}h)  
\stackrel{\eqref{eq:tauEdelE}}{=} c(\vds(\tau^E h)) 
\stackrel{\eqref{eq:rel_diff_charact_exact_sequence}}{=} 0. 
\]
By Corollary~\ref{cor:existence_sections}, we conclude that $\tau_{\partial W}h$ has sections along the restriction map.
Thus there exist relative characters $f \in \widehat H^{k-\dim \partial W}_r(C^\infty(\partial W,X),C^\infty(W,X);\Z)$ with $\vds(f) = \tau_{\partial W}h$.
It would be nice to extend the transgression maps to a construction of such a section.
In some cases, it is possible to presribe its covariant derivative. 
In more special cases, this uniquely determines the section. 

\emph{Sections for $\tau_{\partial W}h$ with prescribed covariant derivative}.
We want to construct a section $f$ of $\tau_{\partial W}h$ along $r$ with prescribed covariant derivative. 
\index{covariant derivative!prescribed $\sim$}
Assume that $r^*:H^{k-\dim\partial W-1}(C^\infty(\partial W,X);\Ul) \to H^{k-\dim\partial W-1}(C^\infty(W,X);\Ul)$ is the trivial map.
This holds for instance if $W=I$ and $X$ is connected, since in this case the path space $C^\infty(W,X) = \PP(X)$ is contractible.
We start with a pair $(\curv(\tau_{\partial W}h,\chi) \in \Omega^{k-\dim\partial W}_{r,0}(C^\infty(\partial W,X),C^\infty(W,X))$.
Since the map 
$$
(\curv,\cov):\widehat H^{k-\dim\partial W}(C^\infty(\partial W,X),C^\infty(W,X);\Z) \to \Omega^{k-\dim\partial W}_{r,0}(C^\infty(\partial W,X),C^\infty(W,X))
$$
is surjective, we find a relative character $f_0 \in H^{k-\dim \partial W}_r(C^\infty(\partial W,X),C^\infty(W,X);\Z)$ with $(\curv(f_0),\cov(f_0)) = (\curv(\tau_{\partial W}h,\chi)$.
Now take any section $f_1$ of $\tau_{\partial W}h$.
Since $\curv(f_0) = \curv(\tau_{\partial W}h) = \curv(f_1)$, we have $\vds(f_1) - \vds(f_0) = j(u)$ for some $u \in H^{k-\dim\partial W-1}(C^\infty(\partial W,X);\Ul)$.
By the mapping cone sequence for cohomo\-logy with $\Ul$-coefficients and the assumption on the restriction map, we find $\bar u \in H^{k-\dim\partial W-1}(C^\infty(\partial W,X),C^\infty(W,X);\Ul)=\{0\}$ which maps to $u$.
Now put $f:= f_0 + j(\bar u)$.
Then we have $\vds(f) = \vds(f_0) + j(u) = \vds(f_1) = \tau_{\partial W}h$.
Moreover, $\cov(f) = \cov(f_0) = \chi$.
Thus we have found a section $f$ of $\tau_{\partial W}h$ along the restriction map $r$ with prescribed covariant derivative $\cov(f) =\chi$. 
By Corollary~\ref{cor:existence_sections}, the differential form $\chi$ uniquely determines the section $f$ if in addition the map $r^*:H^{k-\dim\partial W-1}(C^\infty(\partial W,X);\Ul) \to H^{k-\dim\partial W-1}(C^\infty(W,X);\Ul)$ is surjective, i.e.~if $H^{k-\dim\partial W-1}(C^\infty(W,X);\Ul)= \{0\}$.
Thus we have proved:

\begin{cor}[Transgression with prescribed covariant derivative I]\label{cor:rel_trans_chi} \index{Corollary!transgression with prescribed covariant derivative I}%
Let $X$ be a smooth manifold, and let $h \in \widehat H^k(X;\Z)$.
Let $W$ be an oriented manifold with boundary.
Assume $H^{k-\dim\partial W-1}(C^\infty(W,X);\Ul)= \{0\}$.
Let $\chi \in \Omega^{k-\dim\partial W-1}(C^\infty(W,X))$ be a differential form such that $(\curv(\tau_{\partial W}h),\chi) \in \Omega^{k-\dim\partial W}_{r,0}(C^\infty(\partial W,X),C^\infty(W,X))$.

Then the transgression maps $\tau_{\partial W}$, $\tau^{\partial E}$ and $\tau^E$ defined in Sections~\ref{subsec:transgression} and \ref{subsec:transgression_bound} uniquely determine a relative differential character $\tau^\chi_{W,\partial W}h \in \widehat H^{k-\dim \partial W}_r(C^\infty(\partial W,X),C^\infty(W,X);\Z)$ satisfying
\index{transgression!with prescribed covariant derivative}%
\begin{align*}
\vds(\tau^\chi_{W,\partial W}h) &= \tau_{\partial W}h \\
\cov(\tau^\chi_{W,\partial W}h) &= \chi \,.  
\end{align*}
\index{+TauWdelW@$\tau_{W,\partial W}$, transgression with presribed covariant derivative}
\end{cor}
A distinguished form $\chi \in \Omega^{k-\dim\partial W-1}(C^\infty(W,X))$ is obtained by integrating $\ev_W^*\curv(h)$ over the fiber of the trivial bundle $C^\infty(W,X) \times W \to C^\infty(W,X)$. 
This will be discussed in the remainder of this section: 

\emph{Sections for $\tau_{\partial W}h$ with covariant derivative determined by transgression}.
Transgression along $W$ yields the form $\cov(\tau^Eh) \in \Omega^{k-\dim\partial W-1}(C^\infty(W,X))$ as a natural candidate for the covariant derivative of a section $\tau_{\partial W}h$ along the restriction map.
The pair $(\curv(\tau_{\partial W}h,\cov(\tau^Eh))$ is $d_r$-closed since 
$$
r^*\curv(\tau_{\partial W}h)
=
\curv(r^*\tau_{\partial W}h)
\stackrel{\eqref{eq:taudelWdelE}}{=}
\curv(\tau^{\partial E}h)
\stackrel{\eqref{eq:tauEdelE}}{=}
d\cov(\tau^Eh) \,.
$$
It remains to check that it has integral periods.
In general this might not be the case.

\begin{cor}[Transgression with prescribed covariant derivative II]\label{cor:rel_trans2} \index{Corollary!transgression with prescribed covariant derivative II}
Let $X$ be a smooth manifold, and let $h \in \widehat H^k(X;\Z)$.
Let $W$ be an oriented manifold with boundary.
Assume $H^{k-\dim\partial W-1}(C^\infty(W,X);\Ul)= \{0\}$.
Assume further that $(\curv(\tau_{\partial W}h),\cov(\tau^Eh)) \in \Omega^{k-\dim\partial W}_{r,0}(C^\infty(\partial W,X),C^\infty(W,X))$.

Then the transgression maps $\tau_{\partial W}$, $\tau^{\partial E}$ and $\tau^E$ defined in Sections~\ref{subsec:transgression} and \ref{subsec:transgression_bound} uniquely determine a relative differential character $\tau_{W,\partial W}h \in \widehat H^{k-\dim \partial W}_r(C^\infty(\partial W,X),C^\infty(W,X);\Z)$ satisfying
\begin{align*}
\vds(\tau_{W,\partial W}h) &= \tau_{\partial W}h  \\
\cov(\tau_{W,\partial W}h) &= \cov(\tau^Eh) \stackrel{\eqref{eq:def_fiber_int_relative}}{=} (-1)^{k-\dim W} \fint_W \ev_W^*\curv(h) \,. 
\end{align*}
\index{+TauWdelW@$\tau_{W,\partial W}$, transgression with presribed covariant derivative}
\end{cor}

\begin{example}
Let $k=\dim W$.
In this case, the assumption on $r^*$ is automatically satisfied.
Given a differential character $h\in\widehat H^k(X;\Z)$, we obtain the relative character \mbox{$\tau_{W,\partial W}h \in \widehat H^{1}_r(C^\infty(\partial W,X),C^\infty(W,X);\Z)$}, in other words a $\Ul$-valued function $\tau_{\partial W}h=\hol^h$ on $C^\infty(\partial W,X)$ together with a real-valued function $\cov(\tau^Eh)$ on $C^\infty(W,X)$.
The condition $\vds(\tau_{W,\partial W})=\tau_{\partial W}h$ says that $\hol^h\circ r = \exp \circ 2\pi i\cdot\cov(\tau^Eh)$.
In the special case $k=2$, this is the well-known fact that the holonomy along a contractible loop is given by the integral of the curvature over a spanning disk.
\end{example}

\begin{example}
Let $k=\dim W+1$.
In Section~\ref{subsec:Hol_PT} we have constructed a Hermitian line bundle with connection $(\LL^h,\nabla^h)$ on $r(C^\infty(W,X))$ together with a section $\PT^h$ along the restriction map $r$.
By Example~\ref{ex:H2rel}, this determines a relative differential character $[\LL^h,\nabla^h,\PT^h] \in \widehat H^2_r(r(C^\infty(W,X)),C^\infty(W,X);\Z)$.
By Example~\ref{ex:trans_k-2}, we have $\vds([\LL^h,\nabla^h,\PT^h]) = [\LL^h,\nabla^h] = \tau_{\partial W} h$.
Moreover, by \eqref{eq:covLL}, we have $\cov([\LL^h,\nabla^h,\PT^h]) = -\fint_W \ev_W^*\curv(h) = \cov(\tau^Eh)$.
Under the assumption of Corollary~\ref{cor:rel_trans2}, we conclude $[\LL^h,\nabla^h,\PT^h] = \tau_{W,\partial W}h$.
\end{example}

\section{Chain field theories}\label{subsubsec:Chain_field}
Topological quantum field theories in the sense of Atiyah \cite{A89} are symmetric monoidal functors from a cobordism category to the category of complex vector spaces. 
\index{topological quantum field theories}
\index{field theory!topological quantum $\sim$}
This concept of topological field theories has been modified in several directions, e.g.~by replacing the source or target category. 

Chain field theories in the sense of \cite{T04} are a modification of topological field theories where the source category is replaced by a category with smooth cycles as objects and chains as morphisms.
Chain field theories are closely related to differential characters.
Using the notion of thin chains, we generalize \cite[Thm.~3.5]{BTW04} from 2-dimensional thin invariant field theories to chain field theories of arbitrary dimension: chain field theories are invariant under thin $2$-morphisms.

We briefly recall the notion of chain field theories:
The objects of the category $\Ch^{n+1}(X)$ are smooth singular $n$-cycles in $X$.
\index{+ChainX@$\Ch^{n+1}(X)$, chain category}%
\index{chain category}%
A morphism from $x$ to $x'$ is an $(n+1)$-chain $a$ such that $\partial a = x'-x$.
Taking the additive group structure of $Z_n(X;\Z)$ and $C_{n+1}(X;\Z)$ as the tensor product turns $\Ch^{n+1}(X)$ into a strict monoidal category, more precisely a strict symmetric monoidal groupoid, see \cite[Prop.~1.1]{T04}.

Let $z \in Z_{n+1}(X;\Z)$ and let $x$ be any object in the category $\Ch^{n+1}(X)$.
Then we have $\partial z = 0= x-x$.
This yields a 1-1 correspondence of the automorphism group of any object of $\Ch^{n+1}(X)$ with the group $Z_{n+1}(X;\Z)$ of smooth singular $(n+1)$-cycles in $X$. 

Let $a,a' \in C_{n+1}(X;\Z)$ and $x,x' \in Z_n(X;\Z)$ with $\partial a = \partial a' = x'-x$.
Then the chains $a,a'$ yield morphisms $x\xrightarrow{a}x'$, $x\xrightarrow{a'}x'$ in $\Ch^{n+1}(X)$ between the same objects $x,x'$.
A chain $b \in C_{n+2}(X;\Z)$ satisfying $\partial b = a'-a$ is called a \emph{$2$-morphism} from the morphism $x\xrightarrow{a}x'$ to the morphism $x\xrightarrow{a'}x'$.\footnote{In \cite[p.~91]{T04}, this is called a chain deformation.}
\index{twomorphism@$2$-morphism}%
We write $a \xRightarrow{b} a'$ for a $2$-morphism from $x\xrightarrow{a}x'$ to $x\xrightarrow{a'}x'$.
\index{+ATToa@$a \xRightarrow{b} a'$, $2$-morphism in $\Ch^{n+1}(X)$}
If $b \in C_{n+2}(X;\Z)$ is thin in the sense of Definition~\ref{def:thin}, i.e.~$b \in S_{n+2}(X;\Z)$, then we call $a \xRightarrow{b} a'$ a \emph{thin $2$-morphism}.
\index{thin $2$-morphism}%
\index{twomorphism@$2$-morphism!thin $\sim$}%

Denote by $\C$-$\tt{Lines}$ the category whose objects are Hermitian lines and whose morphisms are isometries.   
\index{+CLines@$\C$-$\tt{Lines}$, category of Hermitian lines}
\index{category of Hermitian lines}%
A \emph{chain field theory} on $X$ is defined to be a functor of symmetric monoidal tensor categories $E:\Ch^{n+1}(X) \to \mbox{$\C$-$\tt{Lines}$}$ with an additional smoothness condition.
\index{chain field theories}%
\index{field theory!chain $\sim$}%
\index{+EChnXCLines@$E:\Ch^{n+1}(X) \to \mbox{$\C$-$\tt{Lines}$}$, chain field theory}
To formulate this condition, note that $E$ maps the automorphism group $Z_{n+1}(X;\Z)$ of the monoidal unit $0$ of $\Ch^{n+1}(X)$ to the automorphism group $\Ul$ of the monoidal unit $\C$ of $\C$-$\tt{Lines}$.
Hence we obtain a homomorphism $Z_{n+1}(X;\Z) \to \Ul$.
The smoothness condition for the functor $E$ is the requirement that there exists a closed differential form $\omega \in \Omega^{n+2}(X)$ such that for any chain $b \in C_{n+2}(X;\Z)$, we have
\begin{equation}\label{eq:ChFT_smooth}
E(0\xrightarrow{\partial b}0) 
=
\exp\Big( 2 \pi i \int_b \omega \Big) \in \Ul \,.
\end{equation}
Thus a chain field theory $E$ induces a homomorphism $Z_{n+1}(X;\Z) \to \Ul$, $z \mapsto E(z)(1)$.
By the smoothness condition \eqref{eq:ChFT_smooth}, this yields a differential character in $\widehat H^{n+1}(X;\Z)$ with curvature $\omega$.
\index{differential characters!chain field theories}%
\index{chain field theories!differential characters}%
Moreover, chain field theories are classified up to equivalence by the differential characters obtained in this manner, see~\cite[Thm.~2.1]{T04}.

For any Hermitian line $L$, the group of isometric automorphisms of $L$ is canonically identified with $\Ul$.
Thus let $E$ be a chain field theory, $x \in Z_n(X;\Z)$ an object, and $z \in Z_{n+1}(X;\Z)$ an automorphism of $x$.
Then the isometry $E(x\xrightarrow{z}x)$ of the Hermitian line $E(x)$ is given as 
\begin{equation}\label{eq:autom}
E(x\xrightarrow{z}x) 
=  
(E(0\xrightarrow{z}0)(1)) \cdot \id_{E(x)} \,.
\end{equation}

By \cite[p.~434]{BTW04}, chain field theories in the sense of \cite{T04} generalize thin invariant field theories in the sense of \cite{BTW04}.
By \cite[Thm.~3.5]{BTW04}, thin invariant field theories are invariant under thin cobordism of morphisms.
In the context of chain field theories, we obtain the analogous result:

\begin{prop}[Thin invariance] \index{Proposition!thin invariance}%
Chain field theories are invariant under thin $2$-morphisms:
\index{thin invariance!of chain field theories}%
\index{chain field theories!thin invariance}%
Let $E:\Ch^{n+1}(X) \to \mbox{$\C$-$\tt{Lines}$}$ be a chain field theory.
Let $x,x' \in Z_n(X;\Z)$ be objects and $x\xrightarrow{a}x'$, $x\xrightarrow{a'}x'$ morphisms in $\Ch^{n+1}(X)$.
Let $b \in S_{k+2}(X;\Z)$ with $\partial b = a'-a$ and $a\xRightarrow{b}a'$ the corresponding thin $2$-morphism.
Then we have 
\[
E(x\xrightarrow{a}x') = E(x\xrightarrow{a'}x') . 
\]
\end{prop}

\begin{proof}
The composition of the morphism $x\xrightarrow{a'}x'$ with the inverse of $x\xrightarrow{a}x'$ yields an automorphism of $x'$.
For the corresponding automorphism of $E(x')$, we have: 
\allowdisplaybreaks{
\begin{align*}
E(x\xrightarrow{a'}x') \circ (E(x\xrightarrow{a}x'))^{-1}
&=
E((x\xrightarrow{a'}x') \circ (x'\xrightarrow{-a}x)) \\
&=
E(x'\xrightarrow{a'-a}x') \\
&\ist{\eqref{eq:autom}}\,\,
(E(0\xrightarrow{a'-a}0)(1)) \cdot \id_{x'} \\
&\ist{\eqref{eq:ChFT_smooth}}\,\,
\exp \Big( 2\pi i \underbrace{\int_b \omega}_{=0} \Big) \cdot  \id_{x'} \\
&= \id_{x'} \,.
\end{align*}
}
Thus $E(x\xrightarrow{a'}x') = E(x\xrightarrow{a}x')$.
\end{proof}

%%%%%%%%%%%%%%%%%%%%%%%%%%%%%%%%%%%%%%%%%%%%%%%%%%%%%%%%%%%%%%%%%%%%%%%%%

%%%%%%%%%%%%%%%%%%%%%%%%%%%%%%%%%%%%%%%%%%%%%%%%%%%%%%%%%%%%%%%%%%%%%%%%%
\nopagebreak[5]
\printindex

%%%%%%%%%%%%%%%%%%%%%%%%%%%%%%%%%%%%%%%%%%%%%%%%%%%%%%%%%%%%%%%%%%%%%%%%%

%%%%%%%%%%%%%%%%%%%%%%%%%%%%%%%%%%%%%%%%%%%%%%%%%%%%%%%%%%%%%%%%%%%%%%%%%
\end{document}